\documentclass[oneside]{amsart}
\usepackage{geometry}
\usepackage{microtype} 
\usepackage{tikz}
\usetikzlibrary{arrows,calc,matrix,topaths,positioning,scopes,shapes,decorations,decorations.markings,decorations.pathreplacing,decorations.pathmorphing} 

\usepackage{mathtools}

\usepackage{stmaryrd}
\usepackage{verbatim}
\usepackage{float}
\usepackage{textcomp}
\usepackage{amscd,amsfonts}
\usepackage{amsthm,amsmath}
\usepackage{amstext,amssymb}
\usepackage{xcolor}
\usepackage{graphicx}
\usepackage{esint}

\usepackage[hidelinks]{hyperref}
\usepackage{marginnote}

\reversemarginpar
\usepackage{enumerate}
\usepackage{enumitem}
\newlist{enumD}{enumerate}{1}
\setlist[enumD]{label=(Step\arabic*)}
\providecommand{\tabularnewline}{\\}
\usepackage[mathlines,running]{lineno}
\usepackage{wasysym}
\usepackage{multicol}

\usepackage{xpatch}
\makeatletter   
\xpatchcmd{\@tocline}
{\hfil\hbox to\@pnumwidth{\@tocpagenum{#7}}\par}
{\ifnum#1<0\hfill\else\dotfill\fi\hbox to\@pnumwidth{\@tocpagenum{#7}}\par}
{}{}
\makeatother    

\makeatletter
\def\l@subsection{\@tocline{2}{0pt}{4pc}{6pc}{}}
\def\l@subsubsection{\@tocline{3}{0pt}{8pc}{8pc}{}}
\makeatother

\usepackage{faktor} 
\makeatletter
\DeclareRobustCommand*{\mfaktor}[3][]
{
	{ \mathpalette{\mfaktor@impl@}{{#1}{#2}{#3}} }
}
\newcommand*{\mfaktor@impl@}[2]{\mfaktor@impl#1#2}
\newcommand*{\mfaktor@impl}[4]{
	\settoheight{\faktor@zaehlerhoehe}{\ensuremath{#1#2{#3}}}%
	\settoheight{\faktor@nennerhoehe}{\ensuremath{#1#2{#4}}}%
	\raisebox{-0.5\faktor@zaehlerhoehe}{\ensuremath{#1#2{#3}}}%
	\mkern-4mu\diagdown\mkern-5mu%
	\raisebox{0.5\faktor@nennerhoehe}{\ensuremath{#1#2{#4}}}%
}
\makeatother

\numberwithin{equation}{section}
\numberwithin{figure}{section}
\theoremstyle{plain}
\newtheorem{thm}{\protect\theoremname}
\theoremstyle{definition}
\newtheorem{defn}[thm]{\protect\definitionname}
\theoremstyle{plain}
\newtheorem{prop}[thm]{\protect\propositionname}
\theoremstyle{remark}
\newtheorem{rem}[thm]{\protect\remarkname}
\theoremstyle{definition}

\theoremstyle{plain}
\newtheorem{lem}[thm]{\protect\lemmaname}
\theoremstyle{plain}
\newtheorem{cor}[thm]{\protect\corollaryname}
\theoremstyle{definition}
\newtheorem{notation}[thm]{\protect\notationname}
\newtheorem{claim}{\protect\claimname}
\theoremstyle{definition}

\providecommand{\corollaryname}{Corollary}
\providecommand{\definitionname}{Definition}
\providecommand{\examplename}{Example}
\providecommand{\lemmaname}{Lemma}
\providecommand{\propositionname}{Proposition}
\providecommand{\remarkname}{Remark}
\providecommand{\theoremname}{Theorem}
\providecommand{\notationname}{Notation}
\providecommand{\claimname}{Claim}

\tikzstyle{sq}=[draw, fill, rectangle, minimum size=3pt, inner sep=0pt]		
\tikzstyle{colie}   = [circle,thin, minimum width=4pt, draw, inner sep=0pt, path picture={\draw (path picture bounding box.south east) -- (path picture bounding box.north west) (path picture bounding box.south west) -- (path picture bounding box.north east);}]
\tikzstyle{coliel}   = [circle,thin, minimum width=6pt, draw, inner sep=0pt, path picture={\draw (path picture bounding box.south east) -- (path picture bounding box.north west) (path picture bounding box.south west) -- (path picture bounding box.north east);}]
\tikzstyle{lie}=[draw,thin,fill, circle, minimum size=4pt, inner sep=0pt]	
\tikzstyle{prelie}=[draw,thin, fill=blue!50, circle, minimum size=4pt, inner sep=0pt]	
\tikzstyle{coprod}=[draw,thin, circle, minimum size=4pt, inner sep=0pt]	
\tikzstyle{prod}=[draw,thin, fill=black, circle, minimum size=4pt, inner sep=0pt]	
\tikzstyle{circblack}=[draw,thin,fill, circle, minimum size=4pt, inner sep=0pt]

\newcommand{\alex}[1]{\texttt{\textcolor{blue}{#1}}}
\newcommand{\ie}{\emph{i.e.} }
\newcommand{\ot}{\otimes}
\newcommand{\cat}[1]{\mathbf{#1}}
\newcommand{\ca}[1]{\mathcal{#1}}
\def\co{\colon\thinspace}
\newcommand{\GL}{\operatorname{GL}}
\newcommand{\id}{\operatorname{id}}
\def\R{{\mathbb R}}
\def\N{{\mathbb N}}
\def\Hom{\operatorname{Hom}}
\def\Der{\operatorname{Der}}
\def\End{\operatorname{End}}
\def\Aut{\operatorname{Aut}}
\newcommand{\E}{\mathsf{E}}
\newcommand{\idealR}{\mathsf{R}}
\newcommand{\e}{\mathbf{e}}
\newcommand{\pd}{\partial}
\newcommand{\llb}{\llbracket}
\newcommand{\rrb}{\rrbracket}
\newcommand{\Diop}{\mathsf{P}}
\newcommand{\sgn}{\text{sgn}}
\newcommand{\cobd}[1]{ \mathbf{\Omega}(\mathsf{#1}) }
\newcommand{\degd}{\lambda}
\newcommand{\bibl}{\mathsf{BIB}^{\degd}}
\newcommand{\bib}{\mathsf{BIB}}
\newcommand{\vspan}{\operatorname{span}}
\newcommand{\antishrieck}{\textnormal{\textexclamdown}}
\def\redd{\textnormal{\texttt{red}}}

\begin{document}

		\title[BIB-gebras, double Poisson gebras and pre-Calabi-Yau algebras]
	{Balanced infinitesimal bialgebras, double Poisson gebras and pre-Calabi-Yau algebras}

	\author{Alexandre Quesney}
	\address{
		Universidad Politécnica de Madrid (UPM)
		Campus de Montegancedo / Avenida de Montepríncipe s.n.
		28660 Boadilla del Monte, Madrid,
		Spain
	}
	\email{alexandre.quesney@upm.es}

	\keywords{pre-Calabi-Yau algebra, infinitesimal bialgebra, double Poisson gebra, Lie bialgebra}
	\subjclass[2020]{Primary: 18M85; Secondary: 14A22, 17Bxx}
	
	\date{\today}
	
	\begin{abstract}
		We consider the properad that governs the balanced infinitesimal bialgebras equipped with a coproduct of degree $1-d$. 
		This properad naturally encodes a part of the structure of the pre-Calabi-Yau algebras of degree $d$.  
		We compute the cobar construction of its Koszul dual coproperad and show that its gebras lie between  the homotopy double Poisson gebras and the pre-Calabi-Yau algebras. 
	Finally, we show that, if one is willing to consider their curved version, the two resulting notions of curved homotopy balanced infinitesimal bialgebra and curved homotopy double Poisson gebra are equivalent. 
	A relationship with the homotopy odd Lie bialgebras is also discussed. 
	\end{abstract}
	
	\maketitle
	\tableofcontents

	\section*{Introduction}
The pre-Calabi-Yau algebras are a type of structures that generalizes the Calabi-Yau structures on $A_\infty$--algebras that are not necessarily compact or smooth.  
They appeared in the literature under different names, for instance as $V_\infty$--algebras (introduced by T. Tradler and M. Zeinalian in \cite{Tradelr-Zenalian-string}) or $A_\infty$-algebras with boundary (introduced by P. Seidel in \cite{SeidelI}); the term pre-Calabi-Yau algebras was coined by M. Kontsevich and Y. Vlassopoulos (see \cite{KV-talk} and \cite{KTV}).  
Motivations for the introduction and study of these structures lie in symplectic geometry,  in non-commutative geometry and also in string topology. 
They may be seen as non-commutative Poisson structures in derived non-commutative geometry: W. Yeung showed in \cite{Yeung-preCYmod} 
that any pre-Calabi-Yau structure induce a shifted Poisson structure on the associated derived moduli stack of representations. 
In fact, the pre-Calabi-Yau structures generalize several non-commutative Poisson structures, like 
M. Van den Bergh's double Poisson gebras (see the work of N. Iyudu, M. Kontsevich and Y. Vlassopoulos \cite{IK,KYV}), and double quasi-Poisson gebras  (see the work of D. Fernández and E. Herscovich \cite{Fernandez-Herscovich-qp}). 
From another perspective, one may observe that the notion of a pre-Calabi-Yau algebra is homotopical by nature. 
This point of view has been recently developed by J. Leray and B. Vallette in \cite{Leray-Vallette}, where the homotopy theory of both the double Poisson gebras and the pre-Calabi-Yau algebras has been studied.

In the present note, following the properadic approach of \cite{Leray-Vallette}, we provide another description of these structures; we show that they are closely related to the \emph{balanced infinitesimal bialgebras}, thereby highlighting the relevance of the latter in non-commutative geometry.  
\\

The infinitesimal bialgebras (also called $\epsilon$--bialgebras), introduced by S. Joni and G.-C. Rota in \cite{joni1979coalgebras} and further studied by M. Aguiar in \cite{aguiar2000infinitesimal,aguiar2004infinitesimal}, are vector spaces equipped with a product and a coproduct that are both associative and that satisfy a Leibniz compatibility.  
Such gebras turned out to be useful, for instance, in the study of $\mathbf{cd}$--index of polytopes; see \cite{A2,E-R}. 
The \emph{balanced} $\epsilon$--bialgebras were introduced in \cite{aguiar2001associative}.
They are $\epsilon$--bialgebras that satisfy the so-called balanced condition (see \eqref{eq: graph relation balanceator}), which ensures that the structure obtained by anti-symmetrizing both the product and the coproduct is a that of a Lie bialgebra. 
Further aspects of the theory of these gebras has been developed in \emph{loc. cit.}, like a relation with the associative Yang-Baxter equation and the construction of an analogue of Drinfeld's double of a Lie bialgebra, revealing an interesting connection between the balanced $\epsilon$--bialgebras and the Lie bialgebras.

The balanced infinitesimal bialgebras with a coproduct of degree $1-d$ are governed by a dioperad, which we denote by $\bib^{1-d}$. 
They are naturally related to two other gebras: 
$V$--gebras and double Poisson gebras. 
The first ones were considered by T. Tradler and M. Zeinalian in \cite{Tradelr-Zenalian-string} in the context of string topology; they are associative algebras equipped with a symmetric and invariant $2$--tensor of degree $-d$. 
Such a structure can be encoded by a dioperad, which we denote by $\mathsf{V}^{-d}$. 
The latter was shown to be Koszul by K. Poirier and T. Tradler  in \cite{PT19}. 
Note that the minimal dioperadic resolution of $\mathsf{V}^{-d}$ leads to the notion of $\mathsf{V}_\infty$--gebras, or pre-Calabi-Yau algebras, of degree $d$. 
The second type of gebras, the double Poisson gebras, was introduced by M. Van den Bergh in \cite{VdB-Poisson} as non-commutative Poisson structures; 
they are associative algebras equipped with an anti-symmetric double bracket of degree $2-d$ that satisfies  compatibility relations (double Jacobi and double Leibniz relations). 
Such a structure can be encoded by a dioperad, which we denote by $\mathsf{DPois}^{2-d}$. 
The latter was shown to be Koszul by J. Leray in \cite{Leray-II}; its minimal resolution has been computed by J. Leray and B. Vallette in \cite{Leray-Vallette}. 

In this paper, we consider $d=2$ for simplicity and we drop the upper indices in the notations. 
The three dioperads  $\mathsf{V}$, $\mathsf{BIB}$ and  $\mathsf{DPois}$ are multiplicative and admit a quadratic presentation with two generators. 
One generator is the image of the generator of $\mathsf{Ass}$ (the operad that governs the associative algebras) and the other one is $c\in \mathsf{V}(2,0)$, $\delta\in \mathsf{BIB}(2,1)$ and $\mathfrak{b}\in \mathsf{DPois}(2,2)$, which, respectively, encode symmetric $2$--tensors of degree $-2$,  coproducts of degree $-1$ and double brackets of degree $0$. 	
These multiplicative dioperads are naturally related by the following morphisms 
\begin{equation}\label{eq1}
	\mathsf{V} \to \mathsf{BIB} \to 	\mathsf{DPois}, 
\end{equation}
where the left-hand side one sends $c$ to zero, and the right-hand side one sends $\delta$ to zero. 
This is, by fact, a rather trivial relationship. 
However, when considering the homotopy version of these gebras, one obtains a richer theory. 
For reasons that will be clear later, we move from the framework of dioperads to the one of properads. 
This is not mandatory but it turns out to be enlightening.   
So, from now on, $\mathsf{V}$, $\mathsf{BIB}$ and  $\mathsf{DPois}$ denote the properads that govern the aforementioned gebras.

A pre-Calabi-Yau algebra (or pre-CY algebra) of degree $d=2$ on a vector space $A$ is a collection of degree 
$-2 + (n_1+\cdots+n_k)$ maps 
\begin{equation*}
	m_{(k)}^{n_1,\dots,n_k}\co A^{\ot n_1}\ot \cdots \ot A^{\ot n_k}\to A^{\ot k} 
\end{equation*}
indexed by $n_i\geq 0$ if $k\geq 2$ and $n_1\geq 2$ if $k=1$, 
that satisfy certain compatibility relations.  
Such a structure can be encoded by a properad (see \cite{KTV}), which we denote by $\mathsf{pCY}$. 
On the one side, it has be shown in \cite{PT19}  (see also \cite{Leray-Vallette}) that $\mathsf{pCY}$ is isomorphic to the cobar construction of the genus zero part of the Koszul dual coproperad of $\mathsf{V}$. 
On the other side, pre-CY algebras have been compared to the homotopy double Poisson gebras by J. Leray and B. Vallette, in \cite{Leray-Vallette}. 
A closer look at the minimal resolution $\mathsf{DPois}_{\infty}:=\cobd{DPois^{\text{\textexclamdown}}}$ reveals that the operations encoded by this properad are also of the form 
$m_{(k)}^{n_1,\dots,n_k}\co A^{\ot n_1}\ot \cdots \ot A^{\ot n_k}\to A^{\ot k}$ but with the stronger restriction that $n_i\geq 1$ for all $i$. 
In fact,  they obtain a surjection of properads 
\begin{equation*}%
	\mathsf{pCY} \twoheadrightarrow  \mathsf{DPois}_{\infty}.
\end{equation*}

In this paper, we study the cobar construction of the Koszul dual coproperad of $\bib$, denoted by $\cobd{\bib^{\text{\textexclamdown}}}$. 
Due to the balanced condition, the properad $\bib$ and its Koszul dual properad $\bib^!$ are delicate to handle. 
Our main technical result is a combinatorial description of $\bib^!$. 
It allows us to compute the properad $\cobd{\bib^{\text{\textexclamdown}}}$ and to show that the resulting  gebras are a particular type of pre-CY algebras. 
In fact, the only missing operations are those of the form $m_{(k)}^{0,\dots,0}\co \Bbbk \to A^{\ot k}$. 
More precisely, we obtain the following result. 
\begin{thm}\label{th1}[Propositions  \ref{cor: surj cob bib dp} and \ref{cor: surj pCY cob bib}]
	There are surjective morphisms of properads 
	\begin{equation*}
		\mathsf{pCY}  \twoheadrightarrow \cobd{\bib^{\textnormal{\textexclamdown}}} \twoheadrightarrow  \mathsf{DPois}_{\infty}. 
	\end{equation*}
\end{thm}	
This is in stark contrast with the trivial morphisms in \eqref{eq1}. 
It is the manifestation that the Koszul dual (co)properads of $\mathsf{V}$, $\bib$ and $\mathsf{DPois}$ are related by non-trivial morphisms which are not morphisms of quadratic data. 

On the other side, the extra work of dealing with properads rather than with dioperads brings us the following extra piece of information:  
the decomposition maps of the Koszul dual coproperad $\bib^{\textnormal{\textexclamdown}}$ produce only elements of genus zero, that is, $\bib^{\textnormal{\textexclamdown}}$ a codioperad. 
This is a special property of $\bib$, which is shared by the properad $\mathsf{DPois}$ but not by the properad $\mathsf{V}$ (see \cite{Leray-Vallette}). 
It tells us that the differential of the properadic  cobar construction $\cobd{\bib^{\textnormal{\textexclamdown}}}$ splits any element of $\bib^{\textnormal{\textexclamdown}}$ into (a sum of) two elements related by a sole edge.

Theorem \ref{th1} also tells us that $\cobd{\bib^{\text{\textexclamdown}}}$--gebras are closer to pre-CY algebras than $\mathsf{DPois}_{\infty}$--gebras are. 	
However, the authors of \cite{Leray-Vallette} showed that the lack of operations encoded by $\mathsf{DPois}_{\infty}$, when compared to pre-CY-algebras, can be compensated by considering the \textit{curved} homotopy double Poisson gebras. 
This naturally led to consider a curved version of pre-CY algebras, so that one has an equivalence between curved homotopy double Poisson gebras and curved pre-CY algebras; see Corollary 2.43 of \textit{loc. cit.}

We show that a similar result holds true for $\bib$--gebras. 
One way to introduce the curved homotopy double Poisson gebras and the curved homotopy balanced $\epsilon$--bialgebras is by adding a unit to their Koszul dual properads $\mathsf{DPois}^!$ and $\bib^!$. 
In fact, we simply show that these two resulting properads are isomorphic, leading to the same deformation theory. 
\\

Finally, let us present a complementary point of view on our Theorem \ref{th1}.  
Recall that balanced infinitesimal bialgebras may be seen associative analogues of the Lie bialgebras. 
In properadic terms, one has the following commutative diagram. 
\begin{equation*}
	\begin{tikzpicture}
		[>=stealth,thick,draw=black!65, arrow/.style={->,shorten >=1pt}, point/.style={coordinate}, pointille/.style={draw=red, top color=white, bottom color=red},scale=0.7, photon/.style={decorate,decoration={snake,post length=1mm}},baseline=0pt]
		\matrix[row sep=5mm,column sep=12mm,ampersand replacement=\&]
		{
			\node (-10) {$\mathsf{Lie}$};			\& \node (-11) {$\mathsf{Ass}$} 	;\\
			\node (00) {$\mathsf{LieB^1}$};						\& \node (01)  {$\bib^1$} 	;\\
		}; 
		\path
		(-10)     edge[above,->]      		node {} 		(-11)
		(-10)     edge[below,->]      		node {}	  		(00)
		(-11)     edge[right,->]      		node {}	  		(01)
		(00)     edge[right,->]      		node {}	  		(01)
		;
	\end{tikzpicture}
\end{equation*}
Here we changed the degrees: $\mathsf{LieB^1}$ denotes the properad that governs the Lie bialgebras with a Lie cobracket of degree 1, and $\bib^1$ denotes the properad that governs the balanced infinitesimal bialgebra with a coproduct of degree +1. 
Recall that the top morphism $\mathsf{Lie}\to \mathsf{Ass}$ is induced by the morphism of quadratic data that sends the generator that encodes Lie brackets to the anti-symmetrization of the generator that encodes associative products. 
The bottom morphism is also induced by this latter morphism, both at the level of the generator that encodes Lie brackets and, dually, at the level of the generator that encodes Lie cobrackets. 
Since morphisms of quadratic data induce morphisms between the associated Koszul dual coproperads, by functoriality of the cobar construction one obtains the following diagram of properads. 
\begin{equation*}
	\begin{tikzpicture}
		[>=stealth,thick,draw=black!65, arrow/.style={->,shorten >=1pt}, point/.style={coordinate}, pointille/.style={draw=red, top color=white, bottom color=red},scale=0.7, photon/.style={decorate,decoration={snake,post length=1mm}},baseline=0pt]
		\matrix[row sep=5mm,column sep=12mm,ampersand replacement=\&]
		{
			\node (-10) {$\mathsf{Lie}_\infty$};			\& \node (-11) {$\mathsf{Ass}_\infty$} 	;\\
			\node (00) {$\mathsf{LieB}^{1}_{\infty}$};		\& \node (01)  {$\cobd{(\bib^1)^{\text{\textexclamdown}}}$} 	;\\
		}; 
		\path
		(-10)     edge[above,->]      		node {} 		(-11)
		(-10)     edge[below,->]      		node {}	  		(00)
		(-11)     edge[right,->]      		node {}	  		(01)
		(00)     edge[right,->]      		node {}	  		(01)
		;
	\end{tikzpicture}
\end{equation*}
The horizontal morphisms are therefore the generalization of the anti-symmetrization morphism; we describe the bottom morphism in Corollary \ref{cor: Liebi to bib KD}. 

The latter diagram has the following geometric interpretation. 
On the one hand, from the celebrated work of M. Kontsevich  \cite{Kontsevich-Deformation}, we know that $L_\infty$--algebras correspond to  pointed formal dg-manifolds equipped with a homological vector field.  
In the same spirit, S. Merkulov showed in \cite{Merkulov2004} that the homotopy category of  $\mathsf{LieB}^1$--gebras  is equivalent to the derived category of the Poisson pointed formal dg-manifolds; the objects of this category are pointed formal dg-manifolds equipped with some integrable polyvector fields. 
Note that these polyvector fields have components of the form $L^n_m\co \Lambda^n V \to \operatorname{Sym}^m V$ restricted to $m,n\geq 1$. 
\\
On the other hand, following Kontsevich-Soibelman  \cite{KS06}, an $A_\infty$--algebra should be thought of as a non-commutative pointed formal dg-manifold with an integrable vector field.  
More generally, a pre-CY algebra may be thought of as a non-commutative pointed formal dg-manifold with an integrable polyvector field; see \cite{KTV}. 
From Theorem \ref{th1}, a restriction on the components of such a polyvector field leads to an $\cobd{(\bib^1)^{\text{\textexclamdown}}}$--gebra. 
This provides a non-commutative analogue of Merkulov's homotopy Poisson pointed formal dg-manifolds. 
Details and more will be given in a subsequent paper.

	\subsection*{Organization of the paper and technical results} 
	In Section \ref{sec: conv}, we set up our conventions on properads. 
	\\
	In Section \ref{sec: BIB}, we define the properad $\bibl$ that governs the balanced infinitesimal bialgebras with a coproduct of degree $\degd$, and we give the quadratic presentation of its Koszul dual properad $(\bibl)^!$. 
	We postpone the study of the latter to Section \ref{sec: cobar}. 
	\\
	From Section \ref{sec: curved}, we focus on $\lambda=-1$ for simplicity and set $\bib:=\bib^{-1}$. 
	In this section, we show that curved homotopy balanced infinitesimal bialgebras and curved homotopy double Poisson gebras are equivalent notions.  
	To do so, we consider the unital extensions of $\bib$ and of $\mathsf{DPois}$, and we show that  they are isomorphic. \\
	In Section \ref{sec: cobar},  we make explicit the cobar construction $\cobd{\bib^{\text{\textexclamdown}}}$ and the relation with $\mathsf{DPois}_{\infty}$. 
	To do so, we provide a combinatorial description of $\bib^!$ by embedding it into the unital extension of  $\mathsf{DPois}^!$, which is more convenient to handle. 
	This is the main technical result of this paper, stated as Theorem  \ref{prop: iota inclusion}. 
	\\
	In Section \ref{sec: pCY}, we relate  $\cobd{\bib^{\text{\textexclamdown}}}$ with the properad $\mathsf{pCY}$. 
	\\
	In Section \ref{sec: LieB}, we relate $\cobd{\bib^{\text{\textexclamdown}}}$ with the properad $\mathsf{LieB}^{-1}_\infty$.

	\subsection*{Acknowledgments} 
	
	The author is indebted to B. Vallette both for his corrections to an earlier version of this article and for sharing many enlightening comments. 
	The author also wishes to thank M. Livernet for helpful discussions.  
	Part of this work has benefited financial support from FEDER Andalucía 2014-2020, proyecto UMA18-FEDERJA-183.

	\section{Preliminaries}\label{sec: conv}
	
	\subsection{General conventions}
	The symmetric group on $n$ letters is denoted by $\Sigma_n$; a permutation $\sigma \in \Sigma_n$ is determined by the tuple $(\sigma(1),\sigma(2),\dots,\sigma(n))$. 
	The sub-group of $\Sigma_n$ generated by the permutation $(2,3,\dots,n,1)$ is denoted by $\Sigma^{\text{cyc}}_n$.

	All throughout this paper, the ground field $\Bbbk$ is of characteristic zero. 
	For two differential graded modules $M$ and $N$, we let $M\ot N$ denote their tensor product over $\Bbbk$. 
	The suspension functor is denoted by $s$; to a graded module $M$, it associates the graded module $sM$ (also denoted $M[-1]$) such that $(sM)_n= M_{n-1}=M[-1]_n$, for $n\in \mathbb{Z}$. 
	Similarly, the desuspension functor is denoted by $s^{-1}$. 
	We use the Koszul sign rule and the Koszul sign convention.

	\subsection{Properads}

	Our convention on properads follows \cite{Vallette2007,HLV20} to which we refer for more details. 
	In this section, we briefly recall these notions; our presentation closely follows \cite{Leray-Vallette}. 
	
	By a (left reduced) $\Sigma$--bimodule $\E$ we mean a collection of $(\Sigma_m,\Sigma_n)$--(dg-)bimodules $\E(m,n)$ indexed by $m,n\geq 0$ such that $\E(0,n)=0$ for any $n\in \mathbb{N}$.  
	
	We let $\mathsf{I}$ be the $\Sigma$--bimodule given by $\mathsf{I}(1,1)=\Bbbk$ and $\mathsf{I}(m,n)=0$ otherwise. 
We let $G$ be the set of connected directed graphs. 

For a graph $g\in G$, we let $g(\E)$ be the module obtained by decorating each vertex of $g$ by an element of $\E$ of corresponding arity. 
Explicitly, 
\begin{equation*}
	g(\E) = \bigotimes_{v \in V(g)} \E(out(v),in(v)) ,
\end{equation*}
where $V(g)$ is the set of internal vertices of $g$; and $in(v)$ and $out(v)$ stand, respectively,  for the number of inputs of $v$ and the number of outputs of $v$. 
For the trivial graph $g=|$, one has $g(\E):=\mathsf{I}$. 

Let $\mathcal{G}$ be the endofunctor of the category of $\Sigma$--bimodules given by

\begin{equation}\label{eq: monad G}
	\mathcal{G}(\E) = \bigoplus_{g\in G} g(\E).  
\end{equation}

It has a structure of monad. The morphism $\mathcal{G}\circ \mathcal{G}\to \mathcal{G}$ is induced by the morphism $\mathcal{G}(\mathcal{G}(\E))\to \mathcal{G}(\E)$ that  forgets the nesting of connected graphs in $\mathcal{G}(\mathcal{G}(\E))$. 
The unit $id\to \mathcal{G}$ is given by the canonical inclusion $\E\to \mathcal{G}(\E)$. 
\begin{defn}
	A properad is an algebra over the monad $\mathcal{G}$. 
\end{defn}

A properad can be equivalently described as a monoid in the monoidal category $(\Sigma\cat{-mod},\boxtimes,\mathsf{I})$ of $\Sigma$--bimodules equipped with the \textit{connected composition product} $\boxtimes$ defined in \cite{Vallette2007}. 
The dual notion of a coproperad, however, has a few important subtleties. 
Here, we refer as a \textit{coproperad} to a comonoid in $(\Sigma\cat{-mod},\boxtimes,\mathsf{I})$. 

Recall that a \textit{coaugmented} coproperad $\mathsf{C}$ is a coproperad equipped with a morphism of coproperads $\eta\co \mathsf{I}\to \mathsf{C}$ such that its composition with the counit is the identity. 
To such a coaugmented coproperad, we let $\overline{\mathsf{C}}:=\text{coker}(\eta)$ be its coaugmentation coideal.

Dual to $\mathcal{G}$, one can consider the comonad 
\begin{equation*}
	\mathcal{G}^c(\E) = \bigoplus_{g\in \overline{G}} g(\E), 
\end{equation*}
where $\overline{G}:= G\setminus \{|\}$. 
Its comonadic structure is given by partitioning a graph into connected directed sub-graphs. 
Coalgebras over this comonad differ from coproperads in general. 
An important class of coproperads are the \textit{conilpotent} ones, which are those that arise as  coalgebras over $\mathcal{G}^c$ equipped with a coaugmented counit. 
For such a conilpotent coproperad, one may consider its infinitesimal coproduct
\begin{equation*}
	\Delta_{(1,1)}\co \overline{\mathsf{C}} \to \mathcal{G}^c(\overline{\mathsf{C}}) \to \mathcal{G}^c(\overline{\mathsf{C}})^{(2)},
\end{equation*}
where the second morphism is the projection on the summand $\mathcal{G}^c(\overline{\mathsf{C}})^{(2)}$ of $\mathcal{G}^c(\overline{\mathsf{C}})$ that is made up of graphs with two vertices.

Given a conilpotent (or more generally, a coaugmented) coproperad $\mathsf{C}$, one may construct its cobar construction $\cobd{\mathsf{C}}$: it is the quasi-free properad 
  $\cobd{\mathsf{C}} :=( \mathcal{G}(s^{-1}\overline{\mathsf{C}}), \partial )$, where $\partial$ is the sum of the unique derivation that extends the internal differential of $\mathsf{C}$ and of the unique derivation that extends the infinitesimal coproduct of $\mathsf{C}$.

For a quadratic properad $\mathsf{P}=\mathcal{G}(\E)/(\idealR)$ denote by $\mathsf{P}^{\text{\textexclamdown}}$ its \textit{Koszul dual coproperad}; it is a conilpotent coproperad. 
The \emph{Koszul dual properad} $\mathsf{P}^!$ is the linear dual of $\mathsf{P}^{\text{\textexclamdown}}$. 
Explicitly, if $\E$ is finite-dimensional, one has  
 \begin{equation*}
 	\mathsf{P}^! \cong \mathcal{G}(s^{-1}\E^*)/\left(s^{-2}\idealR^\perp \right). 
 \end{equation*}
Here we have followed the convention from  \cite[Lemma 1.30]{Leray-Vallette}.

	\section{The properad of balanced infinitesimal bialgebras}\label{sec: BIB}
	
	In this section, we introduce the properad $\bibl$ that governs the balanced infinitesimal bialgebras with a coproduct of degree $\degd$. 
	Unlike the properad that governs the infinitesimal bialgebras, the presence of the balanced relation in $\bibl$ makes the properad  $(\bibl)^!$ delicate to handle. 
	To remedy this difficulty, we will embed it (see Section \ref{sec: cobar}) into a more accessible properad, the unital extension of  $\mathsf{DPois}^!$.

	\subsection{The properad $\bibl$}
	
	For $\degd\in \mathbb{Z}$, we let $\bibl$ be the properad that governs the balanced infinitesimal bialgebras with a degree $\degd$ coproduct. 
	Explicitly,  it is the quadratic properad 
	\begin{equation*}
		\bibl = \mathcal{G}(\mathsf{E})/ (\mathsf{R})
	\end{equation*}
	where the only non trivial components of the graded $\Sigma$--bimodule $\mathsf{E}$ are 
	\begin{equation*}
		\mathsf{E}(1,2)=\Bbbk[\Sigma_2]\ot 1_1 = 
		\vspan\langle	
		\begin{tikzpicture}
			[scale=.35,auto=left,baseline=0.25cm]  
			\node (in1) at (1,0) {};
			\node (n2) at (1,1) [prod] {};
			\node[above,yshift=-0.1cm] (out1) at (0,2) {\scriptsize{$1$}};
			\node[above,yshift=-0.1cm] (out2) at (2,2) {\scriptsize{$2$}};
			
			\foreach \from/\to in {in1/n2,n2/out1,n2/out2}
			\draw[gray!20!black] (\from) -- (\to);
		\end{tikzpicture}
		,
		\begin{tikzpicture}
			[scale=.35,auto=left,baseline=0.25cm]  
			\node (in1) at (1,0) {};
			\node (n2) at (1,1) [prod] {};
			\node[above,yshift=-0.1cm] (out1) at (0,2) {\scriptsize{$2$}};
			\node[above,yshift=-0.1cm] (out2) at (2,2) {\scriptsize{$1$}};
			
			\foreach \from/\to in {in1/n2,n2/out1,n2/out2}
			\draw[gray!20!black] (\from) -- (\to);
		\end{tikzpicture}
		\rangle
		\text{ and } 
		\mathsf{E}(2,1)= 1_1\ot \Bbbk [\Sigma_2][-\degd] = 
		\vspan\langle 
		\begin{tikzpicture}
			[scale=.35,auto=left,baseline=0.25cm]  
			\node[below,yshift=0.1cm] (in1) at (0,0) {\scriptsize{$1$}};
			\node[below,yshift=0.1cm] (in2) at (2,0) {\scriptsize{$2$}};
			\node (n1) at (1,1) [coprod] {};
			\node (out1) at (1,2)  {};
			
			\foreach \from/\to in {in1/n1,in2/n1,n1/out1}
			\draw[gray!20!black] (\from) -- (\to);
		\end{tikzpicture}
		,
		\begin{tikzpicture}
			[scale=.35,auto=left,baseline=0.25cm]  
			\node[below,yshift=0.1cm] (in1) at (0,0) {\scriptsize{$2$}};
			\node[below,yshift=0.1cm] (in2) at (2,0) {\scriptsize{$1$}};
			\node (n1) at (1,1) [coprod] {};
			\node (out1) at (1,2)  {};
			
			\foreach \from/\to in {in1/n1,in2/n1,n1/out1}
			\draw[gray!20!black] (\from) -- (\to);
		\end{tikzpicture}
		\rangle.
	\end{equation*}
	Note that $\mathsf{E}(2,1)$ is concentrated in degree $\degd$. 
	The $\Sigma$--bimodule $\mathsf{R}$  is generated by the following four relations: 
	\begin{equation*}
		\begin{tikzpicture}
			[decoration={
				markings,
				mark=at position 0.6 with {\arrow{>}}},
			>=stealth,gray!20!black,,scale=.3,auto=left,baseline=-0.45cm]
			\begin{scope}[yscale=-1,xscale=1]  
				\node (in1) at (0,-0.5) {};
				\node (in2) at (2,-0.5) {};
				\node (in3) at (3,-0.5) {};
				\coordinate (itin1) at (3,1) {};
				\node (n2) at (1,1) [prod] {};
				\node (n1) at (2,2.5) [prod] {};;
				\coordinate (out2) at (2,3)  {};
				
				\foreach \from/\to in {in3/itin1,itin1/n1,in1/n2,n2/n1,in2/n2,out2/n1}
				\draw[gray!20!black] (\from) -- (\to);
			\end{scope}
		\end{tikzpicture}
		-
		\begin{tikzpicture}
			[decoration={
				markings,
				mark=at position 0.6 with {\arrow{>}}},
			>=stealth,gray!20!black,,scale=.3,auto=left,baseline=-0.45cm]  
			\begin{scope}[yscale=-1,xscale=1]
				\node (in1) at (0,-0.5) {};
				\node (in2) at (1,-0.5) {};
				\node (in3) at (3,-0.5) {};
				\coordinate (itin1) at (0,1) {};
				\node (n2) at (2,1) [prod] {};
				\node (n1) at (1,2.5) [prod] {};;
				\coordinate (out2) at (1,3)  {};
				
				\foreach \from/\to in {in1/itin1,itin1/n1,in2/n2,n2/n1,in3/n2,out2/n1}
				\draw[gray!20!black] (\from) -- (\to);
			\end{scope}
		\end{tikzpicture};
		\hspace{1.5cm}
		\begin{tikzpicture}
			[decoration={
				markings,
				mark=at position 0.6 with {\arrow{>}}},
			>=stealth,gray!20!black,,scale=.3,auto=left,baseline=0.45cm]  
			\node (in1) at (0,-0.5) {};
			\node (in2) at (2,-0.5) {};
			\node (in3) at (3,-0.5) {};
			\coordinate (itin1) at (3,1) {};
			\node (n2) at (1,1) [coprod] {};
			\node (n1) at (2,2.5) [coprod] {};;
			\coordinate (out2) at (2,3)  {};
			
			\foreach \from/\to in {in3/itin1,itin1/n1,in1/n2,n2/n1,in2/n2,out2/n1}
			\draw[gray!20!black] (\from) -- (\to);
		\end{tikzpicture}
		-(-1)^{\degd}
		\begin{tikzpicture}
			[decoration={
				markings,
				mark=at position 0.6 with {\arrow{>}}},
			>=stealth,gray!20!black,,scale=.3,auto=left,baseline=0.45cm]  
			\node (in1) at (0,-0.5) {};
			\node (in2) at (1,-0.5) {};
			\node (in3) at (3,-0.5) {};
			\coordinate (itin1) at (0,1) {};
			\node (n2) at (2,1) [coprod] {};
			\node (n1) at (1,2.5) [coprod] {};;
			\coordinate (out2) at (1,3)  {};
			
			\foreach \from/\to in {in1/itin1,itin1/n1,in2/n2,n2/n1,in3/n2,out2/n1}
			\draw[gray!20!black] (\from) -- (\to);
		\end{tikzpicture};
		\hspace{1.5cm}
		\begin{tikzpicture}
			[scale=.3,auto=left,baseline=0.45cm]  
			\node (in1) at (0,0) {};
			\node (in2) at (2,0) {};
			\node (n1) at (1,1.25) [coprod] {};;
			\node (n2) at (1,2.25) [prod] {};
			\node (out1) at (0,3.5)  {};
			\node (out2) at (2,3.5)  {};
			
			\foreach \from/\to in {in1/n1,in2/n1,n1/n2,n2/out1,n2/out2}
			\draw[gray!20!black] (\from) -- (\to);
		\end{tikzpicture}
		-
		\begin{tikzpicture}
			[scale=.3,auto=left,baseline=0.45cm]  
			\node (in1) at (1,0) {};
			\node (in2) at (3,0) {};
			\coordinate (itin2) at (3,1) {};
			\node (n1) at (1,1) [prod] {};
			\node (n2) at (2,2.5) [coprod] {};;
			\coordinate (itout1) at (0,2.5)  {};
			\node (out1) at (0,3.5)  {};
			\node (out2) at (2,3.5)  {};
			
			\foreach \from/\to in {in1/n1,n1/n2,n1/itout1,itout1/out1,n2/out2,in2/itin2,itin2/n2}
			\draw[gray!20!black] (\from) -- (\to);
		\end{tikzpicture}
		-
		\begin{tikzpicture}
			[scale=.3,auto=left,baseline=0.45cm]  
			\node (in1) at (0,0) {};
			\node (in2) at (2,0) {};
			\coordinate (itin1) at (0,1) {};
			\node (n2) at (2,1) [prod] {};
			\node (n1) at (1,2.5) [coprod] {};;
			\coordinate (itout2) at (3,2.5)  {};
			\node (out1) at (1,3.5)  {};
			\node (out2) at (3,3.5)  {};
			
			\foreach \from/\to in {in1/itin1,itin1/n1,n1/out1,in2/n2,n2/n1,n2/itout2,itout2/out2}
			\draw[gray!20!black] (\from) -- (\to);
		\end{tikzpicture}
	\end{equation*}
	\begin{equation}\label{eq: graph relation balanceator}
		(-1)^{\degd}
		\begin{tikzpicture}
			[scale=.3,auto=left,baseline=0.45cm]  
			\begin{scope}[yscale=1,xscale=-1]
				\node[below,yshift=0.1cm] (in1) at (1,0) {\scriptsize{$1$}};
				\node[below,yshift=0.1cm] (in2) at (2.5,0) {\scriptsize{$2$}};
				\coordinate (itin2) at (2.5,1) {};
				\node (n1) at (1,1) [prod] {};
				\node (n2) at (1,2.5) [coprod] {};;
				\coordinate (itout1) at (0.25,1.75)  {};
				\coordinate (itout2) at (2.5,2.5)  {};
				\node[above,yshift=-0.1cm] (out1) at (1,3.5)  {\scriptsize{$1$}};
				\node[above,yshift=-0.1cm] (out2) at (2.5,3.5)  {\scriptsize{$2$}};
				
				\foreach \from/\to in {in1/n1,n1/itout1,itout1/n2,n2/out1,in2/itin2,itin2/n2,n1/itout2,itout2/out2}
				\draw[gray!20!black] (\from) -- (\to);
			\end{scope}
		\end{tikzpicture}
		+
		\begin{tikzpicture}
			[scale=.3,auto=left,baseline=0.45cm]  
			\begin{scope}[yscale=1,xscale=-1]
				\node[below,yshift=0.1cm] (in1) at (1,0) {\scriptsize{$2$}};
				\node[below,yshift=0.1cm] (in2) at (2.5,0) {\scriptsize{$1$}};
				\coordinate (itin2) at (2.5,1) {};
				\node (n1) at (1,1) [prod] {};
				\node (n2) at (1,2.5) [coprod] {};;
				\coordinate (itout1) at (0.25,1.75)  {};
				\coordinate (itout2) at (2.5,2.5)  {};
				\node[above,yshift=-0.1cm] (out1) at (1,3.5)  {\scriptsize{$2$}};
				\node[above,yshift=-0.1cm] (out2) at (2.5,3.5)  {\scriptsize{$1$}};
				
				\foreach \from/\to in {in1/n1,n1/itout1,itout1/n2,n2/out1,in2/itin2,itin2/n2,n1/itout2,itout2/out2}
				\draw[gray!20!black] (\from) -- (\to);
			\end{scope}
		\end{tikzpicture}
		-
		\begin{tikzpicture}
			[scale=.3,auto=left,baseline=0.45cm]  
			\node[below,yshift=0.1cm] (in1) at (1,0) {\scriptsize{$1$}};
			\node[below,yshift=0.1cm] (in2) at (2.5,0) {\scriptsize{$2$}};
			\coordinate (itin2) at (2.5,1) {};
			\node (n1) at (1,1) [prod] {};
			\node (n2) at (1,2.5) [coprod] {};;
			\coordinate (itout1) at (0.25,1.75)  {};
			\coordinate (itout2) at (2.5,2.5)  {};
			\node[above,yshift=-0.1cm] (out1) at (1,3.5)  {\scriptsize{$2$}};
			\node[above,yshift=-0.1cm] (out2) at (2.5,3.5)  {\scriptsize{$1$}};
			
			\foreach \from/\to in {in1/n1,n1/itout1,itout1/n2,n2/out1,in2/itin2,itin2/n2,n1/itout2,itout2/out2}
			\draw[gray!20!black] (\from) -- (\to);
		\end{tikzpicture}
		-(-1)^\degd
		\begin{tikzpicture}
			[scale=.3,auto=left,baseline=0.45cm]  
			\node[below,yshift=0.1cm] (in1) at (1,0) {\scriptsize{$2$}};
			\node[below,yshift=0.1cm] (in2) at (2.5,0) {\scriptsize{$1$}};
			\coordinate (itin2) at (2.5,1) {};
			\node (n1) at (1,1) [prod] {};
			\node (n2) at (1,2.5) [coprod] {};;
			\coordinate (itout1) at (0.25,1.75)  {};
			\coordinate (itout2) at (2.5,2.5)  {};
			\node[above,yshift=-0.1cm] (out1) at (1,3.5)  {\scriptsize{$1$}};
			\node[above,yshift=-0.1cm] (out2) at (2.5,3.5)  {\scriptsize{$2$}};
			
			\foreach \from/\to in {in1/n1,n1/itout1,itout1/n2,n2/out1,in2/itin2,itin2/n2,n1/itout2,itout2/out2}
			\draw[gray!20!black] (\from) -- (\to);
		\end{tikzpicture}.
	\end{equation}

	\begin{rem}\label{rem: Liebi bib}
		The balanced infinitesimal bialgebras are the associative analogue of the Lie bialgebras; see \cite{aguiar2004infinitesimal}. 
		In our context, let 	$\mathsf{LieB}^\degd$ be the properad of degree $\degd$ Lie bialgebras. 
		It is generated by a bracket and a cobracket  
			\begin{equation*}
			\begin{tikzpicture}
				[scale=.35,auto=left,baseline=0.25cm]  
				\node (in1) at (1,0) {};
				\node (n2) at (1,1) [prod] {};
				\node[above,yshift=-0.1cm] (out1) at (0,2) {\scriptsize{$1$}};
				\node[above,yshift=-0.1cm] (out2) at (2,2) {\scriptsize{$2$}};
				
				\foreach \from/\to in {in1/n2,n2/out1,n2/out2}
				\draw[gray!20!black] (\from) -- (\to);
			\end{tikzpicture}
			= -
			\begin{tikzpicture}
				[scale=.35,auto=left,baseline=0.25cm]  
				\node (in1) at (1,0) {};
				\node (n2) at (1,1) [prod] {};
				\node[above,yshift=-0.1cm] (out1) at (0,2) {\scriptsize{$2$}};
				\node[above,yshift=-0.1cm] (out2) at (2,2) {\scriptsize{$1$}};
				
				\foreach \from/\to in {in1/n2,n2/out1,n2/out2}
				\draw[gray!20!black] (\from) -- (\to);
			\end{tikzpicture}
			\text{ and } 
			\begin{tikzpicture}
				[scale=.35,auto=left,baseline=.05cm]  
				\node[below,yshift=0.1cm] (in1) at (0,0) {\scriptsize{$1$}};
				\node[below,yshift=0.1cm] (in2) at (2,0) {\scriptsize{$2$}};
				\node (n1) at (1,1) [coprod] {};
				\node (out1) at (1,2)  {};
				
				\foreach \from/\to in {in1/n1,in2/n1,n1/out1}
				\draw[gray!20!black] (\from) -- (\to);
			\end{tikzpicture}
			= (-1)^{\degd+1} 
			\begin{tikzpicture}
				[scale=.35,auto=left,baseline=.05cm]  
				\node[below,yshift=0.1cm] (in1) at (0,0) {\scriptsize{$2$}};
				\node[below,yshift=0.1cm] (in2) at (2,0) {\scriptsize{$1$}};
				\node (n1) at (1,1) [coprod] {};
				\node (out1) at (1,2)  {};
				
				\foreach \from/\to in {in1/n1,in2/n1,n1/out1}
				\draw[gray!20!black] (\from) -- (\to);
			\end{tikzpicture}
		\end{equation*}
	of degree $0$ and $\degd$ respectively. 
	They satisfy the usual Jacobi and the degree $\degd$ coJacobi identity, plus the Leibniz relation: 
		\begin{equation*}
				-
			\begin{tikzpicture}
				[scale=.3,auto=left,baseline=0.45cm]  
				\node (in1) at (0,0) {};
				\node (in2) at (2,0) {};
				\node (n1) at (1,1.25) [coprod] {};;
				\node (n2) at (1,2.25) [prod] {};
				\node (out1) at (0,3.5)  {};
				\node (out2) at (2,3.5)  {};
				
				\foreach \from/\to in {in1/n1,in2/n1,n1/n2,n2/out1,n2/out2}
				\draw[gray!20!black] (\from) -- (\to);
			\end{tikzpicture}
		+
			\begin{tikzpicture}
				[scale=.3,auto=left,baseline=0.45cm]  
				\node (in1) at (1,0) {};
				\node (in2) at (3,0) {};
				\coordinate (itin2) at (3,1) {};
				\node (n1) at (1,1) [prod] {};
				\node (n2) at (2,2.5) [coprod] {};;
				\coordinate (itout1) at (0,2.5)  {};
				\node (out1) at (0,3.5)  {};
				\node (out2) at (2,3.5)  {};
				
				\foreach \from/\to in {in1/n1,n1/n2,n1/itout1,itout1/out1,n2/out2,in2/itin2,itin2/n2}
				\draw[gray!20!black] (\from) -- (\to);
			\end{tikzpicture}
			+
			\begin{tikzpicture}
				[scale=.3,auto=left,baseline=0.45cm]  
				\node (in1) at (0,0) {};
				\node (in2) at (2,0) {};
				\coordinate (itin1) at (0,1) {};
				\node (n2) at (2,1) [prod] {};
				\node (n1) at (1,2.5) [coprod] {};;
				\coordinate (itout2) at (3,2.5)  {};
				\node (out1) at (1,3.5)  {};
				\node (out2) at (3,3.5)  {};
				
				\foreach \from/\to in {in1/itin1,itin1/n1,n1/out1,in2/n2,n2/n1,n2/itout2,itout2/out2}
				\draw[gray!20!black] (\from) -- (\to);
			\end{tikzpicture}
			+
			\begin{tikzpicture}
				[scale=.3,auto=left,baseline=0.45cm]  
				\node[below,yshift=0.1cm] (in1) at (1,0) {};
				\node[below,yshift=0.1cm] (in2) at (2.5,0) {};
				\coordinate (itin2) at (2.5,1) {};
				\node (n1) at (1,1) [prod] {};
				\node (n2) at (1,2.5) [coprod] {};;
				\coordinate (itout1) at (0.25,1.75)  {};
				\coordinate (itout2) at (2.5,2.5)  {};
				\node[above,yshift=-0.1cm] (out1) at (1,3.5)  {};
				\node[above,yshift=-0.1cm] (out2) at (2.5,3.5)  {};
				
				\foreach \from/\to in {in1/n1,n1/itout1,itout1/n2,n2/out1,in2/itin2,itin2/n2,n1/itout2,itout2/out2}
				\draw[gray!20!black] (\from) -- (\to);
			\end{tikzpicture}
		+
			\begin{tikzpicture}
			[scale=.3,auto=left,baseline=0.45cm]  
				\begin{scope}[xscale=-1]
			\node[below,yshift=0.1cm] (in1) at (1,0) {};
			\node[below,yshift=0.1cm] (in2) at (2.5,0) {};
			\coordinate (itin2) at (2.5,1) {};
			\node (n1) at (1,1) [prod] {};
			\node (n2) at (1,2.5) [coprod] {};;
			\coordinate (itout1) at (0.25,1.75)  {};
			\coordinate (itout2) at (2.5,2.5)  {};
			\node[above,yshift=-0.1cm] (out1) at (1,3.5)  {};
			\node[above,yshift=-0.1cm] (out2) at (2.5,3.5)  {};
			
			\foreach \from/\to in {in1/n1,n1/itout1,itout1/n2,n2/out1,in2/itin2,itin2/n2,n1/itout2,itout2/out2}
			\draw[gray!20!black] (\from) -- (\to);
			\end{scope}
		\end{tikzpicture};
		\end{equation*}
		see \cite{Merkulov2004}. 
		One has a morphism of properads 
		\begin{equation*}\label{eq: morph LieBi BIB}
			\mathsf{LieB}^\degd \to \mathsf{BIB}^\degd
		\end{equation*}
		given by anti-symmetrizing the product and "$\degd$--anti-symmetrizing" the coproduct  \ie by sending the degree $\degd$ Lie cobracket to 
		\begin{equation*}
			\begin{tikzpicture}
				[scale=.35,auto=left,baseline=0.25cm]  
				\node[below,yshift=0.1cm] (in1) at (0,0) {\scriptsize{$1$}};
				\node[below,yshift=0.1cm] (in2) at (2,0) {\scriptsize{$2$}};
				\node (n1) at (1,1) [coprod] {};
				\node (out1) at (1,2)  {};
				\foreach \from/\to in {in1/n1,in2/n1,n1/out1}
				\draw[gray!20!black] (\from) -- (\to);
			\end{tikzpicture}  
			- (-1)^\degd 
			\begin{tikzpicture}
				[scale=.35,auto=left,baseline=0.25cm]  
				\node[below,yshift=0.1cm] (in1) at (0,0) {\scriptsize{$2$}};
				\node[below,yshift=0.1cm] (in2) at (2,0) {\scriptsize{$1$}};
				\node (n1) at (1,1) [coprod] {};
				\node (out1) at (1,2)  {};
				\foreach \from/\to in {in1/n1,in2/n1,n1/out1}
				\draw[gray!20!black] (\from) -- (\to);
			\end{tikzpicture}.
		\end{equation*}
	\end{rem}
	
	\begin{rem}
		If $\degd$ is odd, the genus 1 operation in $\mathsf{LieB}^\degd$  obtained by composing the Lie cobracket with the Lie bracket  is equal to zero. 
		The representations of this properad are known as the \emph{involutive} Lie bialgebras. 
	\end{rem}

	\subsection{The Koszul dual properad of $\bibl$}
	
	\begin{prop}
	The Koszul dual properad of $\bibl$ is given by $\mathcal{G}(s^{-1}\E^*)/(s^{-2}\mathsf{R}^{\perp})$, where the only non-trivial components of $s^{-1}\mathsf{E}^*$ are 
	\begin{equation*}
		(s^{-1}\mathsf{E}^*(1,2))_{-1}
		= 
		\vspan \left\langle	
			\begin{tikzpicture}
			[scale=.35,auto=left,baseline=0.35cm]  
			\node (in1) at (1,0) {};
			\node (n2) at (1,1) [colie] {};
			\node[above,yshift=-0.1cm] (out1) at (0,2) {\scriptsize{$1$}};
			\node[above,yshift=-0.1cm] (out2) at (2,2) {\scriptsize{$2$}};
			
			\foreach \from/\to in {in1/n2,n2/out1,n2/out2}
			\draw[gray!20!black] (\from) -- (\to);
		\end{tikzpicture},
			\begin{tikzpicture}
			[scale=.35,auto=left,baseline=0.35cm]  
			\node (in1) at (1,0) {};
			\node (n2) at (1,1) [colie] {};
			\node[above,yshift=-0.1cm] (out1) at (0,2) {\scriptsize{$2$}};
			\node[above,yshift=-0.1cm] (out2) at (2,2) {\scriptsize{$1$}};
			
			\foreach \from/\to in {in1/n2,n2/out1,n2/out2}
			\draw[gray!20!black] (\from) -- (\to);
		\end{tikzpicture}
		\right\rangle 
		\hspace{.25cm}
		\text{ and } 
		\hspace{.25cm}
		(s^{-1}\mathsf{E}^*(2,1))_{-\degd-1}
		=
		\vspan 
		\left \langle
		\begin{tikzpicture}
			[scale=.35,auto=left,baseline=0.15cm]  
			\node[below,yshift=0.1cm] (in1) at (0,0) {\scriptsize{$1$}};
			\node[below,yshift=0.1cm] (in2) at (2,0) {\scriptsize{$2$}};
			\node (n1) at (1,1) [prelie] {};
			\node (out1) at (1,2)  {};
			
			\foreach \from/\to in {in1/n1,in2/n1,n1/out1}
			\draw[gray!20!black] (\from) -- (\to);
		\end{tikzpicture},
		\begin{tikzpicture}
			[scale=.35,auto=left,baseline=0.15cm]  
			\node[below,yshift=0.1cm] (in1) at (0,0) {\scriptsize{$2$}};
			\node[below,yshift=0.1cm] (in2) at (2,0) {\scriptsize{$1$}};
			\node (n1) at (1,1) [prelie] {};
			\node (out1) at (1,2)  {};
			
			\foreach \from/\to in {in1/n1,in2/n1,n1/out1}
			\draw[gray!20!black] (\from) -- (\to);
		\end{tikzpicture}
		\right \rangle,
	\end{equation*}
	and  $s^{-2}\mathsf{R}^{\perp}$ is generated by 
	\begin{equation}\label{rel Koszul 1}
		\begin{tikzpicture}
			[decoration={
				markings,
				mark=at position 0.6 with {\arrow{>}}},
			>=stealth,gray!20!black,,scale=.3,auto=left,baseline=-0.45cm]  
			\begin{scope}[yscale=-1,xscale=1]
				\node (in1) at (0,-0.5) {};
				\node (in2) at (1,-0.5) {};
				\node (in3) at (3,-0.5) {};
				\coordinate (itin1) at (0,1) {};
				\node (n2) at (2,1) [colie] {};
				\node (n1) at (1,2.5) [colie] {};;
				\coordinate (out2) at (1,3)  {};
				
				\foreach \from/\to in {in1/itin1,itin1/n1,in2/n2,n2/n1,in3/n2,out2/n1}
				\draw[gray!20!black] (\from) -- (\to);
			\end{scope}
		\end{tikzpicture}
		+
		\begin{tikzpicture}
			[decoration={
				markings,
				mark=at position 0.6 with {\arrow{>}}},
			>=stealth,gray!20!black,,scale=.3,auto=left,baseline=-0.45cm]
			\begin{scope}[yscale=-1,xscale=1]  
				\node (in1) at (0,-0.5) {};
				\node (in2) at (2,-0.5) {};
				\node (in3) at (3,-0.5) {};
				\coordinate (itin1) at (3,1) {};
				\node (n2) at (1,1) [colie] {};
				\node (n1) at (2,2.5) [colie] {};;
				\coordinate (out2) at (2,3)  {};
				
				\foreach \from/\to in {in3/itin1,itin1/n1,in1/n2,n2/n1,in2/n2,out2/n1}
				\draw[gray!20!black] (\from) -- (\to);
			\end{scope}
		\end{tikzpicture};
		\hspace{0.4cm}
		\begin{tikzpicture}
			[decoration={
				markings,
				mark=at position 0.6 with {\arrow{>}}},
			>=stealth,gray!20!black,,scale=.3,auto=left,baseline=0.45cm]  
			\node (in1) at (0,-0.5) {};
			\node (in2) at (1,-0.5) {};
			\node (in3) at (3,-0.5) {};
			\coordinate (itin1) at (0,1) {};
			\node (n2) at (2,1) [prelie] {};
			\node (n1) at (1,2.5) [prelie] {};;
			\coordinate (out2) at (1,3)  {};
			
			\foreach \from/\to in {in1/itin1,itin1/n1,in2/n2,n2/n1,in3/n2,out2/n1}
			\draw[gray!20!black] (\from) -- (\to);
		\end{tikzpicture}
		+(-1)^\degd
		\begin{tikzpicture}
			[decoration={
				markings,
				mark=at position 0.6 with {\arrow{>}}},
			>=stealth,gray!20!black,,scale=.3,auto=left,baseline=0.45cm]  
			\node (in1) at (0,-0.5) {};
			\node (in2) at (2,-0.5) {};
			\node (in3) at (3,-0.5) {};
			\coordinate (itin1) at (3,1) {};
			\node (n2) at (1,1) [prelie] {};
			\node (n1) at (2,2.5) [prelie] {};;
			\coordinate (out2) at (2,3)  {};
			
			\foreach \from/\to in {in3/itin1,itin1/n1,in1/n2,n2/n1,in2/n2,out2/n1}
			\draw[gray!20!black] (\from) -- (\to);
		\end{tikzpicture};
		\hspace{0.4cm}
		(-1)^{\degd}
		\begin{tikzpicture}
			[scale=.3,auto=left,baseline=0.45cm]  
			\node (in1) at (0,0) {};
			\node (in2) at (2,0) {};
			\node (n1) at (1,1.25) [prelie] {};;
			\node (n2) at (1,2.25) [colie] {};
			\node (out1) at (0,3.5)  {};
			\node (out2) at (2,3.5)  {};
			
			\foreach \from/\to in {in1/n1,in2/n1,n1/n2,n2/out1,n2/out2}
			\draw[gray!20!black] (\from) -- (\to);
		\end{tikzpicture}
		+
		\begin{tikzpicture}
			[scale=.3,auto=left,baseline=0.45cm]  
			\node (in1) at (1,0) {};
			\node (in2) at (3,0) {};
			\coordinate (itin2) at (3,1) {};
			\node (n1) at (1,1) [colie] {};
			\node (n2) at (2,2.5) [prelie] {};;
			\coordinate (itout1) at (0,2.5)  {};
			\node (out1) at (0,3.5)  {};
			\node (out2) at (2,3.5)  {};
			
			\foreach \from/\to in {in1/n1,n1/n2,n1/itout1,itout1/out1,n2/out2,in2/itin2,itin2/n2}
			\draw[gray!20!black] (\from) -- (\to);
		\end{tikzpicture}
		;\hspace{.5cm}
		(-1)^{\degd}
		\begin{tikzpicture}
			[scale=.3,auto=left,baseline=0.45cm]  
			\node (in1) at (0,0) {};
			\node (in2) at (2,0) {};
			\node (n1) at (1,1.25) [prelie] {};;
			\node (n2) at (1,2.25) [colie] {};
			\node (out1) at (0,3.5)  {};
			\node (out2) at (2,3.5)  {};
			
			\foreach \from/\to in {in1/n1,in2/n1,n1/n2,n2/out1,n2/out2}
			\draw[gray!20!black] (\from) -- (\to);
		\end{tikzpicture}
		+
		\begin{tikzpicture}
			[scale=.3,auto=left,baseline=0.45cm]  
			\node (in1) at (0,0) {};
			\node (in2) at (2,0) {};
			\coordinate (itin1) at (0,1) {};
			\node (n2) at (2,1) [colie] {};
			\node (n1) at (1,2.5) [prelie] {};;
			\coordinate (itout2) at (3,2.5)  {};
			\node (out1) at (1,3.5)  {};
			\node (out2) at (3,3.5)  {};
			
			\foreach \from/\to in {in1/itin1,itin1/n1,n1/out1,in2/n2,n2/n1,n2/itout2,itout2/out2}
			\draw[gray!20!black] (\from) -- (\to);
		\end{tikzpicture}
	\end{equation}
	\begin{equation}\label{rel Koszul 2}
		\begin{tikzpicture}
			[scale=.3,auto=left,baseline=0.45cm]  
			\begin{scope}[yscale=1,xscale=-1]
				\node[below,yshift=0.1cm] (in1) at (1,0) {\scriptsize{$2$}};
				\node[below,yshift=0.1cm] (in2) at (2.5,0) {\scriptsize{$1$}};
				\coordinate (itin2) at (2.5,1) {};
				\node (n1) at (1,1) [colie] {};
				\node (n2) at (1,2.5) [prelie] {};;
				\coordinate (itout1) at (0.25,1.75)  {};
				\coordinate (itout2) at (2.5,2.5)  {};
				\node[above,yshift=-0.1cm] (out1) at (1,3.5)  {\scriptsize{$2$}};
				\node[above,yshift=-0.1cm] (out2) at (2.5,3.5)  {\scriptsize{$1$}};
				
				\foreach \from/\to in {in1/n1,n1/itout1,itout1/n2,n2/out1,in2/itin2,itin2/n2,n1/itout2,itout2/out2}
				\draw[gray!20!black] (\from) -- (\to);
			\end{scope}
		\end{tikzpicture}
		- (-1)^\degd	
		\begin{tikzpicture}
			[scale=.3,auto=left,baseline=0.45cm]  
			\begin{scope}[yscale=1,xscale=-1]
				\node[below,yshift=0.1cm] (in1) at (1,0) {\scriptsize{$1$}};
				\node[below,yshift=0.1cm] (in2) at (2.5,0) {\scriptsize{$2$}};
				\coordinate (itin2) at (2.5,1) {};
				\node (n1) at (1,1) [colie] {};
				\node (n2) at (1,2.5) [prelie] {};;
				\coordinate (itout1) at (0.25,1.75)  {};
				\coordinate (itout2) at (2.5,2.5)  {};
				\node[above,yshift=-0.1cm] (out1) at (1,3.5)  {\scriptsize{$1$}};
				\node[above,yshift=-0.1cm] (out2) at (2.5,3.5)  {\scriptsize{$2$}};
				
				\foreach \from/\to in {in1/n1,n1/itout1,itout1/n2,n2/out1,in2/itin2,itin2/n2,n1/itout2,itout2/out2}
				\draw[gray!20!black] (\from) -- (\to);
			\end{scope}
		\end{tikzpicture}
		\text{ ; }\hspace{.7cm}
		\begin{tikzpicture}
			[scale=.3,auto=left,baseline=0.45cm]  
			\begin{scope}[yscale=1,xscale=-1]
				\node[below,yshift=0.1cm] (in1) at (1,0) {\scriptsize{$2$}};
				\node[below,yshift=0.1cm] (in2) at (2.5,0) {\scriptsize{$1$}};
				\coordinate (itin2) at (2.5,1) {};
				\node (n1) at (1,1) [colie] {};
				\node (n2) at (1,2.5) [prelie] {};;
				\coordinate (itout1) at (0.25,1.75)  {};
				\coordinate (itout2) at (2.5,2.5)  {};
				\node[above,yshift=-0.1cm] (out1) at (1,3.5)  {\scriptsize{$2$}};
				\node[above,yshift=-0.1cm] (out2) at (2.5,3.5)  {\scriptsize{$1$}};
				
				\foreach \from/\to in {in1/n1,n1/itout1,itout1/n2,n2/out1,in2/itin2,itin2/n2,n1/itout2,itout2/out2}
				\draw[gray!20!black] (\from) -- (\to);
			\end{scope}
		\end{tikzpicture}
		+
		\begin{tikzpicture}
			[scale=.3,auto=left,baseline=0.45cm]  
			\node[below,yshift=0.1cm] (in1) at (1,0) {\scriptsize{$1$}};
			\node[below,yshift=0.1cm] (in2) at (2.5,0) {\scriptsize{$2$}};
			\coordinate (itin2) at (2.5,1) {};
			\node (n1) at (1,1) [colie] {};
			\node (n2) at (1,2.5) [prelie] {};;
			\coordinate (itout1) at (0.25,1.75)  {};
			\coordinate (itout2) at (2.5,2.5)  {};
			\node[above,yshift=-0.1cm] (out1) at (1,3.5)  {\scriptsize{$2$}};
			\node[above,yshift=-0.1cm] (out2) at (2.5,3.5)  {\scriptsize{$1$}};
			
			\foreach \from/\to in {in1/n1,n1/itout1,itout1/n2,n2/out1,in2/itin2,itin2/n2,n1/itout2,itout2/out2}
			\draw[gray!20!black] (\from) -- (\to);
		\end{tikzpicture};
		\hspace{.7cm}
		\begin{tikzpicture}
			[scale=.25,auto=left,baseline=-.1cm]  
			
			\node (out1) at (0,3) {};
			\node (out2) at (1,3)   {};
			\node (n1) at (0,1) [prelie] {};
			
			\coordinate (int) at (-1,0) {};
			\coordinate (int2) at (1,0) {};
			\node (n0) at (0,-1) [colie] {}; 
			\node (in1) at (0,-3) {};
			
			\foreach \from/\to in {int/n1,n1/out1,in1/n0,n0/int,n0/int2,int2/n1}
			\draw[gray!40!black] (\from) -- (\to);
		\end{tikzpicture}
		\hspace{.4cm}\text{ and }\hspace{.4cm}
		\begin{tikzpicture}
			[scale=.25,auto=left,baseline=-.2cm]  
			
			\node (out1) at (0,3) {};
			\node (out2) at (1,3)   {};
			\node (n1) at (0,1) [prelie] {};
			
			\coordinate (int) at (-1,0) {};
			\coordinate (int2) at (1,0) {};
			\coordinate (bint) at (-1,-1) {};
			\coordinate (bint2) at (1,-1) {};
			\node (n0) at (0,-2) [colie] {}; 
			\node (in1) at (0,-4) {};
			
			\foreach \from/\to in {int/n1,n1/out1,in1/n0,n0/bint,n0/bint2,int2/n1,bint/int2,bint2/int}
			\draw[gray!40!black] (\from) -- (\to);
		\end{tikzpicture}
	.
	\end{equation}	
	\end{prop}
\begin{proof}
The only non immediate fact is that  $s^{-2}\mathsf{R}^{\perp}(2,2)$ is indeed generated by the last two elements of \eqref{rel Koszul 1} and the first two elements of \eqref{rel Koszul 2}. 
First, observe that the dimension of the underlying vector space of the $\Sigma$--bimodule  $\mathcal{G}^{(2)}(\mathsf{E})(2,2)$ is 20. 
Indeed, as a $\Sigma$--bimodule it is generated by the five trees 
\begin{equation*}
		\begin{tikzpicture}
		[scale=.3,auto=left,baseline=0.45cm]  
		\node (in1) at (0,0) {};
		\node (in2) at (2,0) {};
		\node (n1) at (1,1.25) [prelie] {};;
		\node (n2) at (1,2.25) [colie] {};
		\node (out1) at (0,3.5)  {};
		\node (out2) at (2,3.5)  {};
		
		\foreach \from/\to in {in1/n1,in2/n1,n1/n2,n2/out1,n2/out2}
		\draw[gray!20!black] (\from) -- (\to);
	\end{tikzpicture}
	;\hspace{.5cm}
		\begin{tikzpicture}
		[scale=.3,auto=left,baseline=0.45cm]  
		\node (in1) at (1,0) {};
		\node (in2) at (3,0) {};
		\coordinate (itin2) at (3,1) {};
		\node (n1) at (1,1) [colie] {};
		\node (n2) at (2,2.5) [prelie] {};;
		\coordinate (itout1) at (0,2.5)  {};
		\node (out1) at (0,3.5)  {};
		\node (out2) at (2,3.5)  {};
		
		\foreach \from/\to in {in1/n1,n1/n2,n1/itout1,itout1/out1,n2/out2,in2/itin2,itin2/n2}
		\draw[gray!20!black] (\from) -- (\to);
	\end{tikzpicture}
	;\hspace{.5cm}
	\begin{tikzpicture}
		[scale=.3,auto=left,baseline=0.45cm]  
		\node (in1) at (0,0) {};
		\node (in2) at (2,0) {};
		\coordinate (itin1) at (0,1) {};
		\node (n2) at (2,1) [colie] {};
		\node (n1) at (1,2.5) [prelie] {};;
		\coordinate (itout2) at (3,2.5)  {};
		\node (out1) at (1,3.5)  {};
		\node (out2) at (3,3.5)  {};
		
		\foreach \from/\to in {in1/itin1,itin1/n1,n1/out1,in2/n2,n2/n1,n2/itout2,itout2/out2}
		\draw[gray!20!black] (\from) -- (\to);
	\end{tikzpicture}
		;\hspace{.5cm}
				\begin{tikzpicture}
				[scale=.3,auto=left,baseline=0.45cm]  
				\begin{scope}[yscale=1,xscale=-1]
					\node[below,yshift=0.1cm] (in1) at (1,0) {};
					\node[below,yshift=0.1cm] (in2) at (2.5,0) {};
					\coordinate (itin2) at (2.5,1) {};
					\node (n1) at (1,1) [colie] {};
					\node (n2) at (1,2.5) [prelie] {};;
					\coordinate (itout1) at (0.25,1.75)  {};
					\coordinate (itout2) at (2.5,2.5)  {};
					\node[above,yshift=-0.1cm] (out1) at (1,3.5)  {};
					\node[above,yshift=-0.1cm] (out2) at (2.5,3.5)  {};
					
					\foreach \from/\to in {in1/n1,n1/itout1,itout1/n2,n2/out1,in2/itin2,itin2/n2,n1/itout2,itout2/out2}
					\draw[gray!20!black] (\from) -- (\to);
				\end{scope}
			\end{tikzpicture}
			\hspace{.25cm}
			\text{ and }
			\hspace{.25cm}
			\begin{tikzpicture}
				[scale=.3,auto=left,baseline=0.45cm]  
				\node[below,yshift=0.1cm] (in1) at (1,0) {};
				\node[below,yshift=0.1cm] (in2) at (2.5,0) {};
				\coordinate (itin2) at (2.5,1) {};
				\node (n1) at (1,1) [colie] {};
				\node (n2) at (1,2.5) [prelie] {};
				\coordinate (itout1) at (0.25,1.75)  {};
				\coordinate (itout2) at (2.5,2.5)  {};
				\node[above,yshift=-0.1cm] (out1) at (1,3.5)  {};
				\node[above,yshift=-0.1cm] (out2) at (2.5,3.5)  {};
				
				\foreach \from/\to in {in1/n1,n1/itout1,itout1/n2,n2/out1,in2/itin2,itin2/n2,n1/itout2,itout2/out2}
				\draw[gray!20!black] (\from) -- (\to);
			\end{tikzpicture}; 
\end{equation*}
and by taking the action of $\Sigma_2^{\textnormal{op}}\times \Sigma_2$ into account, one gets the resulting dimension for the underlying vector space. 
Similarly, the dimension of $\mathsf{R}(2,2)$ is $4+2=6$ (the  $\Sigma_2^{\textnormal{op}}\times \Sigma_2$--orbit of \eqref{eq: graph relation balanceator} generates a vector space of dimension 2). 
On the other hand, the orthogonal $\mathsf{R}^{\perp}(2,2)$ is of dimension $14$. 
Indeed, the $\Sigma_2^{\textnormal{op}}\times \Sigma_2$--orbits of the last two elements of \eqref{rel Koszul 1}, the first element of \eqref{rel Koszul 2} and the second element of \eqref{rel Koszul 2}, generate a vector space of dimension $2\times 4+2+4$. 
This prove the statement. 
\end{proof}	
	\begin{rem}
		From the first two relations of \eqref{rel Koszul 2}, one obtains the following relation
		\begin{equation}\label{rel Koszul 2b}
			\begin{tikzpicture}
				[scale=.3,auto=left,baseline=0.45cm]  
				\begin{scope}[yscale=1,xscale=-1]
					\node[below,yshift=0.1cm] (in1) at (1,0) {\scriptsize{$2$}};
					\node[below,yshift=0.1cm] (in2) at (2.5,0) {\scriptsize{$1$}};
					\coordinate (itin2) at (2.5,1) {};
					\node (n1) at (1,1) [colie] {};
					\node (n2) at (1,2.5) [prelie] {};;
					\coordinate (itout1) at (0.25,1.75)  {};
					\coordinate (itout2) at (2.5,2.5)  {};
					\node[above,yshift=-0.1cm] (out1) at (1,3.5)  {\scriptsize{$2$}};
					\node[above,yshift=-0.1cm] (out2) at (2.5,3.5)  {\scriptsize{$1$}};
					\foreach \from/\to in {in1/n1,n1/itout1,itout1/n2,n2/out1,in2/itin2,itin2/n2,n1/itout2,itout2/out2}
					\draw[gray!20!black] (\from) -- (\to);
				\end{scope}
			\end{tikzpicture}
			+(-1)^\degd 
			\begin{tikzpicture}
				[scale=.3,auto=left,baseline=0.45cm]  
				\node[below,yshift=0.1cm] (in1) at (1,0) {\scriptsize{$2$}};
				\node[below,yshift=0.1cm] (in2) at (2.5,0) {\scriptsize{$1$}};
				\coordinate (itin2) at (2.5,1) {};
				\node (n1) at (1,1) [colie] {};
				\node (n2) at (1,2.5) [prelie] {};;
				\coordinate (itout1) at (0.25,1.75)  {};
				\coordinate (itout2) at (2.5,2.5)  {};
				\node[above,yshift=-0.1cm] (out1) at (1,3.5)  {\scriptsize{$1$}};
				\node[above,yshift=-0.1cm] (out2) at (2.5,3.5)  {\scriptsize{$2$}};
				
				\foreach \from/\to in {in1/n1,n1/itout1,itout1/n2,n2/out1,in2/itin2,itin2/n2,n1/itout2,itout2/out2}
				\draw[gray!20!black] (\from) -- (\to);
			\end{tikzpicture}
			\in s^{-2}\mathsf{R}^{\perp}(2,2). 
		\end{equation}
	\end{rem}

\begin{rem}
	The properad $(\bibl)^!$ governs \textit{involutive non-commutative Frobenius algebras} with a degree one product and a degree $-\lambda-1$ coproduct that satisfy the  additional conditions given by the first two relations of \eqref{rel Koszul 2}. 
	Let $V$ be such an algebra. 
	Let us write  $\Delta(a)=a^1\ot a^2$ for the coproduct of $a\in V$ (using Sweedler's notation) and $ab$ for the degree one product of $a,b\in V$. 
	These conditions are: 
	\begin{align*}
		 (-1)^{(\degd+1)a +b^1(a+1)} b^1\ot ab^2 
		&= (-1)^{ab+a^1+ a^1a^2 + (\degd+1)b +\degd} 	ba^2\ot a^1 \\
		(-1)^{(\degd+1)a +b^1(a+1)} b^1\ot ab^2
		&= 	(-1)^{1+ab+ab^2} b^1a\ot b^2. 
	\end{align*}
\end{rem}	
	
	Let us establish two technical lemmata. 
	The first one will be used in the proofs of Lemmata \ref{lem: red well-defined} and \ref{lem: red morphism}; the second one will be used in the proof of Lemma \ref{lem: compo genus}. 
	Both lemmata will also be used in the proof of Theorem \ref{prop: iso uBIB and uDPois}. 
	\begin{lem}
	 In $\bibl$	one has the following equalities of level trees:  
		\begin{equation}\label{M1}
			a) 	
.
		\end{equation}
	\end{lem} 
	\begin{proof}
		The equality \eqref{M1} a) is obtained by using the first relation \eqref{rel Koszul 2}, then using the coassociativity relation (\ie the second relation of \eqref{rel Koszul 1}), then using the third relation of \eqref{rel Koszul 1}, and then using the first relation \eqref{rel Koszul 2} again. 
		The proof of the other equalities is similar and is left to the reader. 
	\end{proof}

	For $j,k\geq1$, we let 
\begin{equation*}
	C^{%
.
\end{equation*}
Let 
$\alpha_1,\alpha_2\co \{1,2\}\to \{1,\dots,k+1\}$ and 
$\beta_1,\beta_2\co \{1,2\}\to \{1,\dots,j+1\}$ be as follows. 
\begin{align*}\label{eq: alpha and beta}
	\alpha_1(1):=1 \text{ and } \alpha_1(2):=k+1; &&
	\alpha_2(1):=1 \text{ and } \alpha_2(2):=k; \\
	\beta_1(1):=1  \text{ and } \beta_1(2):=j+1;  &&
	\beta_2(1):=2  \text{ and } \beta_2(2):=j+1. 
\end{align*}

		For all $1\leq a,b\leq 2$, let us consider the elements $C_{%
}_{j+1}
	$. 
		Pictorially, the four cases are:
		\begin{equation*}
			a)~
				%
		\end{equation*}
		\begin{equation*}
				\text{ for: } a)~ j,k\geq 1, 
			\hspace{1cm} b)~  j\geq 2 \text{ and } k\geq 1  
			\hspace{1cm} c)~ j\geq 1 \text{ and } k\geq 2  
			\hspace{2.2cm} \text{ and }  d)~ j,k\geq 2.
		\end{equation*}
		\begin{defn}
			The above elements are called the \emph{pencil boxes of type a), b), c) and d)}, respectively. 
		\end{defn}
		\begin{defn}
			Similarly as above, for the elements 
			$C_{%
}_{j+1}
			$ one has four cases with the same restrictions on $j$ and $k$; they are called the \emph{twisted pencil boxes} (of type a), b), c) or d)). 
		\end{defn}
		For instance, the twisted pencil boxes of type a) are:
		\begin{equation*}
			C_{%
~.
		\end{equation*}

\begin{lem}\label{lem: pencil box}
	In  $\bibl$, the pencil boxes and the twisted pencil boxes are zero. 
\end{lem}
\begin{proof}
			We proceed by direct inspection. 
			We start by considering the pencil boxes. 	
\begin{itemize}
			\item The pencil boxes of type a).  
		For $k=j=1$ the result is tautological. 
		For $k=2$ and $j=1$,  one may use the first relation of \eqref{rel Koszul 2}, then we may apply the associativity relation for the crossed vertices:
		\begin{equation*}
			%
			=0.
		\end{equation*}
		The case $j=2$ and $k=1$ is similar. 
		For $j=k=2$, one may move down the lower blue vertex, then one may apply the relation \eqref{rel Koszul 2b} on the lowest two vertices, and then one may apply the last relation of \eqref{rel Koszul 1}. 
		This exhibits the pencil box of type a) but with $j=1$ and $k=2$ (one blue vertex can be moved outside the box). 
		Finally, for any $j\geq 2$ and $k\geq 2$, one may use the associativity and coassociativity relations to exhibit a pencil box with $j=k=2$. 
		\item The pencil boxes of type b). 
		For $j=2$ and $k=1$, by using the coassociativity of the blue vertex, one obtains the result. 
		For $j\geq 2$ and $k\geq 2$, by using the associativity and coassociativity relations, one obtains a box that is symmetric to the pencil box of type a) and therefore one may proceed similarly. 
		\item The pencil boxes of type c) are treated similarly as the ones of type b).
		\item For the pencil boxes of type d) with $j,k\geq 2$, by  using the associativity and coassociativity relations, one obtains a box that is symmetric to the pencil box of type a) with $j,k\geq 1$. 
		Such a box can be shown to be zero in the same way (according to the symmetry) as performed to prove that the pencil boxes of type a) are zero.   
		\end{itemize}
			Now let us investigate the twisted pencil boxes. 
		We will treat the type a) ($a=b=1$), the other ones are left to the reader. 
		 For $j=1=k$ this is tautological. 
			For $j=2$ and $k=1$, one may use the third relation of \eqref{rel Koszul 1}:
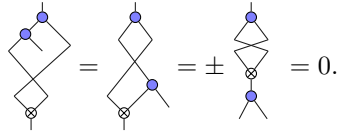
			 \begin{equation*}
			 		\begin{tikzpicture}
			 		[scale=.3,auto=left,baseline=0cm]  
			 		\begin{scope}[yscale=1,xscale=1]
			 			\coordinate (in1) at (1,0) {};
			 			
			 			\coordinate (itin2) at (2.5,1) {};
			 			
			 			\node (n2) at (1.5,2.5) [prelie] {};
			 			\node (n3) at (.75,1.75) [prelie] {};
			 			\node (n1) at (1,-1.75) [colie] {};

			 			\coordinate (nd3) at (1.5,1)  {};
			 			\coordinate (nd2) at (2.75,1.25)  {};
			 			\coordinate (nd3a) at (0,1)  {};
			 			\coordinate (itb1) at (1.75,-1)  {};
			 			\coordinate (itb2) at (.25,-1)  {};
			 			
			 			\node[above,yshift=-0.1cm] (inp1) at (1.5,3.5)  {};
			 			
			 			\node[above,yshift=-0.1cm] (out1) at (1,-3)  {};
			 			\node[above,yshift=-0.1cm] (out2) at (0,-2.5)  {};
			 			\foreach \from/\to in {inp1/n2,n2/n3,n3/nd3,n3/nd3a,n2/nd2,
			 				nd3a/itb1,
			 				itb2/nd2,n1/itb2,n1/itb1,n1/out1}
			 			\draw[gray!20!black] (\from) -- (\to);
			 		\end{scope}
			 	\end{tikzpicture}
			 	=
			 		\begin{tikzpicture}
			 		[scale=.3,auto=left,baseline=0cm]  
			 		\begin{scope}[yscale=1,xscale=1]
			 			\coordinate (in1) at (1,0) {};
			 			
			 			\coordinate (itin2) at (2.5,1) {};
			 			
			 			\node (n2) at (1.5,2.5) [prelie] {};
			 			\node (n3) at (2.25,-.5) [prelie] {};
			 			\node (n1) at (1,-1.75) [colie] {};
			 			
			 			\coordinate (nd2) at (2.25,1.5)  {};
			 			\coordinate (nd2a) at (.75,1.5)  {};
			 			\coordinate (nd3b) at (3,-1.5)  {};
			 			\coordinate (itb2) at (.25,-1)  {};
			 			
			 			\node[above,yshift=-0.1cm] (inp1) at (1.5,3.5)  {};
			 			
			 			\node[above,yshift=-0.1cm] (out1) at (1,-3)  {};
			 			\foreach \from/\to in {inp1/n2,n3/nd2a,n2/nd2,n3/n1,
			 				nd2a/n2,n3/nd3b,
			 				itb2/nd2,n1/itb2,n1/out1}
			 			\draw[gray!20!black] (\from) -- (\to);
			 		\end{scope}
			 	\end{tikzpicture}
			 	=
			 	\pm
			 		\begin{tikzpicture}
			 		[scale=.3,auto=left,baseline=0cm]  
			 		\begin{scope}[yscale=1,xscale=-1]
			 			\coordinate (in1) at (1,0) {};
			 			
			 			\coordinate (itin2) at (2.5,1) {};
			 			\node (n1) at (1,0) [colie] {};
			 			\node (n2) at (1,2.5) [prelie] {};
			 			\node (n3) at (1,-1) [prelie] {};

			 			\coordinate (it1) at (1.75,1)  {};
			 			\coordinate (it2) at (0.25,1)  {};
			 			\coordinate (itb1) at (1.75,1.5)  {};
			 			\coordinate (itb2) at (0.25,1.5)  {};
			 			
			 			\node[above,yshift=-0.1cm] (out1) at (1,3.5)  {};
			 			
			 			\node[above,yshift=-0.1cm] (inp1) at (1.75,-2.5)  {};
			 			\node[above,yshift=-0.1cm] (inp2) at (0,-2.5)  {};
			 			\foreach \from/\to in {n1/it2,itb1/n2,n2/out1,itb2/n2,n1/it1,n3/n1,inp1/n3,inp2/n3,it1/itb2,it2/itb1}
			 			\draw[gray!20!black] (\from) -- (\to);
			 		\end{scope}
			 	\end{tikzpicture}
			 	=0.
			 \end{equation*}
			For $j=2$ and $k=2$, one may use the associativity relation of the crossed vertices, then move the upper crossed vertex below the upper blue vertex, and then apply last relation of \eqref{rel Koszul 1}. 
			This exhibits a twisted box as in the previous case for $j=2$ and $k=1$. 
			Finally, for any $j\geq 2$ and $k\geq 2$, one may use the associativity and coassociativity relations to exhibit a twisted pencil box as in the previous case for $j=k=2$. 	
	\end{proof}

	\section{Curved homotopy balanced infinitesimal bialgebras}\label{sec: curved}

Henceforth (and for the rest of this paper), for simplicity, we consider $\bibl$ for $\lambda=-1$ and we write $\bib$ for $\bib^{-1}$.  
\\

In this section, we consider the notion of \textit{curved} homotopy balanced infinitesimal bialgebras, and we show that it is equivalent to that of curved homotopy double Poisson gebras exposed in \cite{Leray-Vallette}. 
To do so, we closely follow \textit{loc. cit.} by introducing the unital extension of $\bib^!$, and we show that it is isomorphic to the unital extension of $\mathsf{DPois}^!$. 
\\

We let $\mathsf{uBIB}^!$ be the unital extension of $\bib^!= \mathcal{G}(\mathsf{E_{BIB^!}})/(\mathsf{R_{BIB^!}})$. 
It is the properad generated by $\mathsf{E_{uBIB^!}}=\mathsf{E_{BIB^!}} \oplus \mathsf{U}$ where the only non-trivial component of $\mathsf{U}$ is 
\begin{equation*}
	\mathsf{U}(1,0)_{1} = 
	\vspan\langle	

	\rangle,
\end{equation*}
and modded out by the ideal generated by $\mathsf{R_{BIB^!}} \oplus \mathsf{R_U}$ where $\mathsf{R_U}$ is generated by 
\begin{equation*}\label{rel unit Koszul BIB}
	%
.
\end{equation*}

Similarly, let us denote by $\mathsf{uDPois}^!$ the unital extension of the properad $\mathsf{DPois}^!$, as given in \cite[Definition 1.37]{Leray-Vallette}. 
Explicitly, one has $\mathsf{uDPois}^!= \mathcal{G}(\mathsf{E_{uDP^!}})/\mathsf{R_{uDP^!}}$, where 
$\mathsf{E_{uDP^!}}$ is generated as a $\Sigma$--bimodule by 
\begin{equation*}
	%
	\in 	\mathsf{E_{uDP^!}}(1,0)_{1}
\end{equation*}  
and $\mathsf{R_{uDP^!}}$ is generated by 
\begin{equation*}
	%
\end{equation*}

\begin{thm}\label{prop: iso uBIB and uDPois}
	The two properads $\mathsf{uDPois}^!$ and $\mathsf{uBIB}^!$ are isomorphic. 
\end{thm}
\begin{proof}
	Let $f\co \mathcal{G}(\mathsf{E_{uBIB^!}}) \to \mathsf{uDPois}^!$ be the properadic map such that 
	\begin{equation*}
		f\left(%
.
	\end{equation*}
	In particular, one has 
	\begin{equation*}
		f
		\left(%
.
	\end{equation*}
	A direct computation shows that $f$ induces a morphism from $\mathsf{uBIB}^!$ to $\mathsf{uDPois}^!$. 
	In particular, one has:
	\begin{equation*}
		f\left(
			%
	\right).
	\end{equation*}

	In the other way round, one can check that the morphism $g\co \mathcal{G}(\mathsf{E_{uDP^!}})\to \mathsf{uBIB}^!$ defined by 
		\begin{equation*}
		g
		\left(%
	\end{equation*}
	induces a morphism from $\mathsf{uDPois}^!$ to $\mathsf{uBIB}^!$. 
In particular, let us show that the image by $g$ of the last three relations that generate $\mathsf{R_{uDP^!}}$ are zero (the other ones are simpler to verify and are left to the reader). 
To show that the following equality 
\begin{equation*}
	g\left( 
		%
	\right) =0
\end{equation*}
hold true, one may proceed as in the proof of Lemma \ref{lem: red well-defined}:
consider the following level tree
\begin{equation*}
	g\left( 
%
.
\end{equation*}
Move the highest crossed vertex below all the other vertices; it is now at the first level; this produces a sign $(-1)^1$. 
Then, use the first relation of \eqref{rel Koszul 2} on the two vertices that are at levels 2 and 3; this produces a sign $(-1)^1$. Finally, apply the relation a) of \eqref{M1}. This provides the equality. 
On the other hand,  by Lemma \ref{lem: pencil box},  one has 
\begin{equation*}
	g\left(
	%
	=0. 
\end{equation*}
	Finally, to see that $f$ and $g$ are inverse to each other, it is enough to check it on each generator. 
	For this, the only relation that it remains to check is the following one:
	\begin{equation*}
		g
		\left(%
.
	\end{equation*}
\end{proof}

Let us consider the following $\Sigma$--bimodules 
\begin{equation*}
	\mathsf{cDPois}^\antishrieck := (\mathsf{uDPois}^!)^* \text{ and }  
	\mathsf{cBIB}^\antishrieck := (\mathsf{uBIB}^!)^*. 
\end{equation*}
Let us fix $\mathsf{C}$ to be either $\mathsf{cDPois}^\antishrieck $ or $\mathsf{cBIB}^\antishrieck$. 
Let us dualize the infinitesimal composition map: one obtains partial decomposition maps of the following form,   
\begin{equation*}
	\Delta_{(1,1)}\co \mathsf{C} \to \mathsf{C}\boxtimes_{(1,1)} \mathsf{C} \cong \mathcal{G}^c(\mathsf{C})^{(2)}. 
\end{equation*}
In contrast to that of $(\mathsf{DPois}^!)^*$ and $(\mathsf{BIB}^!)^*$, this map does not endow $\mathsf{C}$ with a structure of coproperad. 
Indeed, the iterations of this map produce infinitely many terms. 
What we do have, however, is the following structure. 
\begin{defn}{\cite[Definition 1.15]{Leray-Vallette}} 
	A \emph{partial coproperad} is a $\Sigma$--bimodule $\mathsf{C}$ endowed with a partial decomposition map that satisfies the properties of the restriction of a comonadic product. 
\end{defn}

To a partial coproperad $\mathsf{C}$, one may associate the cobar construction 
$\cobd{\mathsf{C}}$ whose underlying $\Sigma$--bimodule is 
$\mathcal{G}(s^{-1} \mathsf{C})$ (no reduction of $\mathsf{C}$) and the differential is induced by the partial decomposition maps and the internal differential of $\mathsf{C}$. 
We let 
\begin{equation*}
	\mathsf{cDPois}_\infty := \cobd{\mathsf{cDPois}^\antishrieck} \text{ and }  
	\mathsf{cBIB}_\infty := \cobd{\mathsf{cBIB}^\antishrieck}.
\end{equation*}

\begin{thm}
	There is an isomorphism of properads $\mathsf{cBIB}_\infty \cong \mathsf{cDPois}_\infty$. 
\end{thm}
\begin{proof}
	From the isomorphism of properads $\mathsf{uBIB}^!\cong \mathsf{uDPois}^!$ of Theorem \ref{prop: iso uBIB and uDPois}, one obtains an isomorphism of partial coproperads $\mathsf{cDPois}^\antishrieck \cong \mathsf{cBIB}^\antishrieck$. 
	Therefore, one has  an isomorphism between the cobar constructions $\cobd{\mathsf{cDPois}^\antishrieck }$ and $\cobd{\mathsf{cBIB}^\antishrieck }$. 
\end{proof}

Recall from \cite[Definition 1.40]{Leray-Vallette} that a structure of curved homotopy double Poisson gebra on a dg vector space $V$ is a Maurer--Cartan element in the convolution algebra 
\begin{equation*}
	\left( \widehat{\Hom}( \mathsf{cDPois}^\antishrieck, \End_V ), \star, \partial \right) ,
\end{equation*}
where $\End_V$ is the endomorphism properad of $V$. 
Here, since $\mathsf{cDPois}^\antishrieck$ is not coaugmented, it is the full $\Sigma$--bimodule that is considered, that is 
\begin{equation*}
	\widehat{\Hom}( \mathsf{cDPois}^\antishrieck, \End_V ):= 
	\prod_{m\geq 1,n\geq 0}  \Hom_{ \Sigma_m^{\text{op}}\times \Sigma_n } (\mathsf{cDPois}^\antishrieck(m,n), \End_V(m,n) ). 
\end{equation*}
This is mandatory if one wants to obtain the curvature for the corresponding representations. 
Observe that Maurer--Cartan elements in this convolution algebra are in bijection with morphisms of properads 
$ \cobd{\mathsf{cDPois}^\antishrieck}  \to \End_V$.  
A more detailed treatment of this case in the context of operads can be found in \cite[Section 3.3 and, in particular, Proposition 3.18]{DSV-book}. 

In conclusion, one has an identification between curved homotopy double Poisson gebras and \emph{curved homotopy balanced infinitesimal bialgebras} (defined similarly) given by the isomorphism of partial coproperads $\mathsf{cDPois}_\infty \cong \mathsf{cBIB}_\infty$. 

\begin{rem}
	Since $\mathsf{uDPois}^!$ and $\mathsf{uBIB}^!$	   are dioperads, the $\Sigma$--bimodules    $\mathsf{cDPois}^\antishrieck$ and $\mathsf{cBIB}^\antishrieck$ are, in fact, partial codioperads, that is, partial coproperads in which the partial decomposition maps only produce graphs of genus zero. 
\end{rem}

\section{The cobar construction of the Koszul dual coproperad of $\bib$}\label{sec: cobar}

To compute the cobar construction $\cobd{\bib^{\antishrieck}}$, we first embed the properad $\bib^!$ into $\mathsf{uBIB}^!\cong \mathsf{uDPois}^!$, which is more convenient. 
To do so, we start by recalling the combinatorial description of $\mathsf{uDPois}^!$ given in \cite{Leray-Vallette}. 
We reinterpret this description in  $\mathsf{uBIB}^!$ by using the isomorphism $\mathsf{uBIB}^!\cong \mathsf{uDPois}^!$.
This allows us to show that the canonical morphism $\iota\co \bib^!\to \mathsf{uBIB}^!$ is an inclusion (see Theorem  \ref{prop: iota inclusion}).

\subsection{A combinatorial description of $\mathsf{uDPois}^!$ and $\mathsf{uBIB}^!$}

We recall both a basis of the $\Sigma$--bimodule $\mathsf{uDPois}^!$ and the corresponding properadic composition in these terms, as given in \cite{Leray-Vallette}. 
Then, we write the corresponding basis of $\mathsf{uBIB}^!$ under the isomorphism $\mathsf{uDPois}^!\cong \mathsf{uBIB}^!$.  

\begin{rem}
	The description of  $\mathsf{uDPois}^!$ results from a description of $\mathsf{DPois}^!$ and from the inclusion  $\mathsf{DPois}^! \subset \mathsf{uDPois}^!$. 
	The latter is a consequence of an isomorphism $\mathsf{DPois}^!\cong \mathsf{DLie}^! \boxtimes \mathsf{Ass}^!$ and similarly for $\mathsf{uDPois}^!$ (but with $\mathsf{SuAss}^!$ instead of $\mathsf{Ass}^!$). 
	The properad $\bib^!$, however, does not enjoy similar properties; this essentially comes from the balanced condition of $\bib$. 
\end{rem}

Let us recall a basis and the corresponding composition maps of $\mathsf{uDPois}^!$  following \cite[Lemmata 1.33, 1.34 and 1.39]{Leray-Vallette}. 
As a graded vector space, $\mathsf{uDPois}^!(m,n)$ is concentrated in degree $1-n$ and generated by the cyclic words of labeled planar corollas of the following form 
\begin{equation*}
W_{a_1,\dots,a_m}^{J_1,\dots,J_m}:=
	\begin{tikzpicture}
		[scale=.4,auto=left,baseline=0.4cm]  
		\node (out1) at (-2,3)  {};
		\node (out11) at (-1.6,3)  {};
		\node (out12) at (0,3)  {};
		\node (out2) at (1,3) {};
		\node (out21) at (1.4,3) {};
		\node (out22) at (3,3) {};
		\node (out3) at (6,3) {};
		\node (out31) at (6.4,3) {};
		\node (out32) at (8,3) {};
		\node (n1) at (-1,1) [circblack] {};
		\node (n2) at (2,1) [circblack] {};
		\node (n3) at (7,1) [circblack] {};
		\node (in1) at (-1,-.5) {\scriptsize{$a_1$}};
		\node (in2) at (2,-.5) {\scriptsize{$a_2$}};
		\node (in3) at (7,-.5) {\scriptsize{$a_m$}};
		\node (dot1out) at (-.85,2.4) {\scriptsize{$\cdots$}};
		\node (dot2out) at (2.15,2.4) {\scriptsize{$\cdots$}};
		\node (dot3out) at (7.15,2.4) {\scriptsize{$\cdots$}};
		\node (dotintervert) at (4.5,1) {\scriptsize{$\cdots$}};
		\foreach \from/\to in {in1/n1,in2/n2,in3/n3,n1/out1,n1/out12,n1/out11,n2/out2,n2/out21,n2/out22,n3/out3,n3/out31,n3/out32}
		\draw[gray!40!black] (\from) -- (\to);
		\draw [decorate,decoration={zigzag,amplitude=1.3pt,segment length=.8mm}]    (n1) -- (n2);
		\draw [decorate,decoration={zigzag,amplitude=1.3pt,segment length=.8mm}]    (n2) -- (3.5,1);
		\draw [decorate,decoration={zigzag,amplitude=1.3pt,segment length=.8mm}]    (5.5,1) -- (n3);
		\draw [decorate,decoration={brace,amplitude=3pt},xshift=0pt,yshift=-10pt,thick] (-2,3.2) -- (0,3.2) node [above,midway] {\scriptsize{$J_1$}};
		\draw [decorate,decoration={brace,amplitude=3pt},xshift=0pt,yshift=-10pt,thick] (1,3.2) -- (3,3.2) node [above,midway] {\scriptsize{$J_2$}};
		\draw [decorate,decoration={brace,amplitude=3pt},xshift=0pt,yshift=-10pt,thick] (6,3.2) -- (8,3.2) node [above,midway] {\scriptsize{$J_m$}};
	\end{tikzpicture}
\end{equation*}
for $\{a_1,\dots,a_m\}=\{1,\dots,m\}$ and $J_1, J_2,\cdots, J_m$ ordered sets such that their underlying sets form a partition of $\{1,\dots,n\}$. 
The order on the set $J_i$ corresponds to the one induced by the planar structure of the corolla. 
Note that some labeling sets $J_i$ may be empty, in which case the corresponding corollas have no leaves. 
By \textit{cyclic} words, we mean words subject to the following relation: 
\begin{equation}\label{eq: cyclic relation}
	\begin{tikzpicture}
		[scale=.4,auto=left,baseline=0.4cm]  
		\node (out1) at (-2,3)  {};
		\node (out11) at (-1.6,3)  {};
		\node (out12) at (0,3)  {};
		\node (out2) at (1,3) {};
		\node (out21) at (1.4,3) {};
		\node (out22) at (3,3) {};
		\node (out3) at (6,3) {};
		\node (out31) at (6.4,3) {};
		\node (out32) at (8,3) {};
		\node (n1) at (-1,1) [circblack] {};
		\node (n2) at (2,1) [circblack] {};
		\node (n3) at (7,1) [circblack] {};
		\node (in1) at (-1,-.5) {\scriptsize{$a_1$}};
		\node (in2) at (2,-.5) {\scriptsize{$a_2$}};
		\node (in3) at (7,-.5) {\scriptsize{$a_m$}};
		\node (dot1out) at (-.85,2.4) {\scriptsize{$\cdots$}};
		\node (dot2out) at (2.15,2.4) {\scriptsize{$\cdots$}};
		\node (dot3out) at (7.15,2.4) {\scriptsize{$\cdots$}};
		\node (dotintervert) at (4.5,1) {\scriptsize{$\cdots$}};
		\foreach \from/\to in {in1/n1,in2/n2,in3/n3,n1/out1,n1/out12,n1/out11,n2/out2,n2/out21,n2/out22,n3/out3,n3/out31,n3/out32}
		\draw[gray!40!black] (\from) -- (\to);
		\draw [decorate,decoration={zigzag,amplitude=1.3pt,segment length=.8mm}]    (n1) -- (n2);
		\draw [decorate,decoration={zigzag,amplitude=1.3pt,segment length=.8mm}]    (n2) -- (3.5,1);
		\draw [decorate,decoration={zigzag,amplitude=1.3pt,segment length=.8mm}]    (5.5,1) -- (n3);
		\draw [decorate,decoration={brace,amplitude=3pt},xshift=0pt,yshift=-10pt,thick] (-2,3.2) -- (0,3.2) node [above,midway] {\scriptsize{$J_1$}};
		\draw [decorate,decoration={brace,amplitude=3pt},xshift=0pt,yshift=-10pt,thick] (1,3.2) -- (3,3.2) node [above,midway] {\scriptsize{$J_2$}};
		\draw [decorate,decoration={brace,amplitude=3pt},xshift=0pt,yshift=-10pt,thick] (6,3.2) -- (8,3.2) node [above,midway] {\scriptsize{$J_m$}};
	\end{tikzpicture}
	- (-1)^{j_1n+j_m}
	\begin{tikzpicture}
		[scale=.4,auto=left,baseline=0.4cm]  
		\node (out1) at (-2,3)  {};
		\node (out11) at (-1.6,3)  {};
		\node (out12) at (0,3)  {};
		\node (out2) at (1,3) {};
		\node (out21) at (1.4,3) {};
		\node (out22) at (3,3) {};
		\node (out3) at (6,3) {};
		\node (out31) at (6.4,3) {};
		\node (out32) at (8,3) {};
		\node (n1) at (-1,1) [circblack] {};
		\node (n2) at (2,1) [circblack] {};
		\node (n3) at (7,1) [circblack] {};
		\node (in1) at (-1,-.5) {\scriptsize{$a_{2}$}};
		\node (in2) at (2,-.5) {\scriptsize{$a_{3}$}};
		\node (in3) at (7,-.5) {\scriptsize{$a_{1}$}};
		\node (dot1out) at (-.85,2.4) {\scriptsize{$\cdots$}};
		\node (dot2out) at (2.15,2.4) {\scriptsize{$\cdots$}};
		\node (dot3out) at (7.15,2.4) {\scriptsize{$\cdots$}};
		\node (dotintervert) at (4.5,1) {\scriptsize{$\cdots$}};
		\foreach \from/\to in {in1/n1,in2/n2,in3/n3,n1/out1,n1/out12,n1/out11,n2/out2,n2/out21,n2/out22,n3/out3,n3/out31,n3/out32}
		\draw[gray!40!black] (\from) -- (\to);
		\draw [decorate,decoration={zigzag,amplitude=1.3pt,segment length=.8mm}]    (n1) -- (n2);
		\draw [decorate,decoration={zigzag,amplitude=1.3pt,segment length=.8mm}]    (n2) -- (3.5,1);
		\draw [decorate,decoration={zigzag,amplitude=1.3pt,segment length=.8mm}]    (5.5,1) -- (n3);
		\draw [decorate,decoration={brace,amplitude=3pt},xshift=0pt,yshift=-10pt,thick] (-2,3.2) -- (0,3.2) node [above,midway] {\scriptsize{$J_{2}$}};
		\draw [decorate,decoration={brace,amplitude=3pt},xshift=0pt,yshift=-10pt,thick] (1,3.2) -- (3,3.2) node [above,midway] {\scriptsize{$J_{3}$}};
		\draw [decorate,decoration={brace,amplitude=3pt},xshift=0pt,yshift=-10pt,thick] (6,3.2) -- (8,3.2) node [above,midway] {\scriptsize{$J_{1}$}};
	\end{tikzpicture},
\end{equation}
where $j_s:=|J_s|$ for $1\leq s\leq m$.

The cyclic word  $W_{a_1,\dots,a_m}^{J_1,\dots,J_m}$ with $J_i=\{\mu^1_i<\mu^2_i<\cdots<\mu^{j_i}_i\}$ 
 corresponds to the following element: 
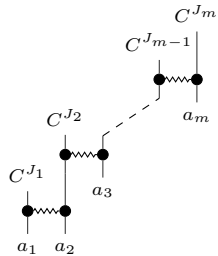
\begin{equation*}
	\begin{tikzpicture}
		[scale=.25,auto=left,baseline=-0cm]  
		\node (out-1) at (-3,0)  {\scriptsize{$C^{J_1}$}};
		\node (out1) at (-1,3) {\scriptsize{$C^{J_2}$}};
		\node (outk-1) at (4,7)   {\scriptsize{$C^{J_{m-1}}$}};
		\node (outk) at (6,8.5)   {\scriptsize{$C^{J_m}$}};
		
		\coordinate (int2) at (1,2) {};
		\coordinate (intk-1) at (4,4) {};

		\node (nk-1) at (4,5) [circblack] {};
		\node (nk) at (6,5) [circblack] {};
		
		\node (n1) at (-1,1) [circblack] {};
		\node (n2) at (1,1) [circblack] {};
		\coordinate (int) at (-1,0) {};
		
		\node (n0) at (-1,-2) [circblack] {}; 
		\node (n-1) at (-3,-2) [circblack] {}; 
		\node (in-1) at (-3,-4) {\scriptsize{$a_1$}};
		\node (in1) at (-1,-4) {\scriptsize{$a_2$}};
		\node (in2) at (1,-1) {\scriptsize{$a_3$}};
		\node (ink) at (6,3) {\scriptsize{$a_m$}};
		\foreach \from/\to in {in-1/n-1,int/n1,in2/n2,n1/out1,n2/int2,n-1/out-1,in1/n0,n0/int,nk-1/outk-1,nk/outk,intk-1/nk-1,nk/ink}
		\draw[gray!40!black] (\from) -- (\to);
		\draw [decorate,decoration={zigzag,amplitude=0.9pt,segment length=.8mm}]    (n1) -- (n2);
		\draw [decorate,decoration={zigzag,amplitude=0.9pt,segment length=.8mm}]    (n-1) -- (n0);
		\draw [decorate,decoration={zigzag,amplitude=0.9pt,segment length=.8mm}]    (nk-1) -- (nk);
		\draw [dashed]    (int2) -- (intk-1);
	\end{tikzpicture}
\end{equation*}
where 
\begin{equation}\label{eq: Cji}
	C^{J_i} = 
\begin{tikzpicture}
	[scale=.3,auto=left,baseline=0.75cm]  
	\begin{scope}[xscale=-1]
		
		\node (out-1) at (-4.5,4.5) {\scriptsize{$\mu^{j_i}_i$}};
		\node (out0) at (-1.75,4.5) {\scriptsize{$\mu^{j_i-1}_i$}};
		\node (out1) at (1.5,4.5) {\scriptsize{$\mu^2_i$}};
		\node (out2) at (3,4.5) {\scriptsize{$\mu^1_i$}};

		\node (n0) at (-2.75,2.75) [colie] {};
		\node (n1) at (0,1.5) [colie] {};
		\node (n2) at (1,1) [colie] {};
		
		\node (int) at (-.5,2.75) {$\cdots$};
		\node (in1) at (1,0) {};
		\foreach \from/\to in {in1/n2,n1/n2,n1/n0,n1/out1,n2/out2,n0/out-1,n0/out0}
		\draw[gray!20!black] (\from) -- (\to);
	\end{scope}
\end{tikzpicture}
\text{ if } j_i\neq 0  \text{ and } 
C^{J_i} = 
\begin{tikzpicture}
	[scale=.35,auto=left,baseline=0.25cm]  
	\node (n1) at (1,1) [circblack] {};
	\node (out1) at (1,0)  {};
	
	\foreach \from/\to in {n1/out1}
	\draw[gray!20!black] (\from) -- (\to);
\end{tikzpicture}
\text{ if } j_i= 0. 
\end{equation}

\begin{notation} \label{notation}
	\begin{itemize}
		\item We may see the $J_i$'s as tuples, which allows us to concatenate them: for two tuples $J$ and $K$, we write their concatenation by $JK$. The empty tuple, which we denote by $\emptyset$, is the unit element.  
		\item 	For a partition $j_1+ \cdots +j_m= n$ with $j_i\geq 0$, we write $(j_k)$ for the tuple $(j_1+ \cdots +j_{k-1}+1,j_1+ \cdots +j_{k-1}+2,\dots,j_1+ \cdots +j_{k})$, and    
		$(0):= \emptyset$. 
		In particular, $W_{a_1,\dots,a_m}^{(j_1),\dots,(j_m)}$ denotes the word such that, for each $1\leq k\leq m$, the $k$--th planar corolla is labeled by $(j_{k})$. 
	\end{itemize}	
\end{notation}

The properadic composition is as follows. 
For $u\in J_{p}$, one has a splitting of the tuple $J_p$ in three parts: $J_p= J^{(1)}_p u J^{(2)}_p$. 
Let $W_J= W_{a_1,a_2,\dots,a_p}^{J_1,J_2,\dots,J_p}$ and $W_K= W_{b_1,b_2,\dots,b_q}^{K_1,K_2,\dots,K_q}$ be two words such that $J_p\ni u$. 

For $q=1$, one has  
\begin{equation*}%
	W_{a_1,a_2,\dots,a_p}^{J_1,J_2,\dots,J_p} \circ_{u,1} W_{1}^{K_1} 
	=(-1)^{j_p^{(2)}(k_1-1)} 
	W_{a_1,a_2,\dots,a_p}^{J'_1,J'_2,\dots,J'_{p-1},J^{(1)'}_pK'_{1}J_p^{(2)'}}. 
\end{equation*}
For $q\geq 2$, one has
\begin{equation}\label{eq: compo word}
	W_{a_1,a_2,\dots,a_p}^{J_1,J_2,\dots,J_p} \circ_{u,b_1} W_{b_1,b_2,\dots,b_q}^{K_1,K_2,\dots,K_q} 
	=(-1)^{\eta} 
	W_{a'_1,a'_2,\dots,a'_p,b'_2,\dots,b'_q}^{J'_1,J'_2,\dots,J'_{p-1},K'_{1}J_p^{(2)'},K'_2,\dots,K'_{q-1},J^{(1)'}_pK'_q},
\end{equation}
where 
\begin{equation}\label{eq: eta compo}
	\eta =  
	j_p^{(1)}(q-1)
	+j_p^{(1)}(  k_1 + ...+ k_{q-1} -q ) 
	+ j_p^{(1)}( j_p^{(2)} -1)
	+j_p^{(2)}
	+j_p^{(2)} ( k_1 -1). 
\end{equation} 
and the prime indices stand for the classical re-indexation (all the terms of $K_i$ are increased by $u-1$, and $j\in J_i$ is unchanged if $j<u$ and is increased by $|K_1|+ \cdots +|K_q|-1$ if $j>u$).  
Pictorially,  the right-sided element is 
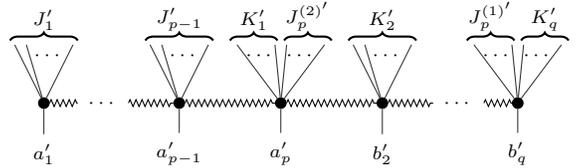
\begin{equation*}
	\begin{tikzpicture}
		[scale=.45,auto=left,baseline=0.4cm]  
		\node (out-1) at (-7,3) {};
		\node (out-11) at (-8.6,3) {};
		\node (out-12) at (-9,3) {};
		\node (out0) at (-5,3)  {};
		\node (out01) at (-4.6,3)  {};
		\node (out02) at (-3,3)  {};
		\node (out1) at (-2.5,3)  {};
		\node (out11) at (-1.2,3)  {};
		\node (out111) at (-.8,3)  {};
		\node (out12) at (.5,3)  {};
		\node (out2) at (1,3) {};
		\node (out21) at (1.4,3) {};
		\node (out22) at (3,3) {};
		\node (out3) at (4.5,3) {};
		\node (out31) at (5.8,3) {};
		\node (out311) at (6.2,3) {};
		\node (out32) at (7.5,3) {};
		
		\node (n-1) at (-8,1) [circblack] {};
		\node (n0) at (-4,1) [circblack] {};
		\node (n1) at (-1,1) [circblack] {};
		\node (n2) at (2,1) [circblack] {};
		\node (n3) at (6,1) [circblack] {};
		
		\node (in-1) at (-8,-.5) {\scriptsize{$a'_1$}};
		\node (in0) at (-4,-.5) {\scriptsize{$a'_{p-1}$}};
		\node (in1) at (-1,-.5) {\scriptsize{$a'_p$}};
		\node (in2) at (2,-.5) {\scriptsize{$b'_2$}};
		\node (in3) at (6,-.5) {\scriptsize{$b'_{q}$}};
		\node (dot-1out) at (-7.85,2.4) {\scriptsize{$\cdots$}};
		\node (dot0out) at (-3.85,2.4) {\scriptsize{$\cdots$}};
		\node (dot1out) at (-1.6,2.4) {\scriptsize{$\cdots$}};
		\node (dot11out) at (-.4,2.4) {\scriptsize{$\cdots$}};
		\node (dot2out) at (2.15,2.4) {\scriptsize{$\cdots$}};
		\node (dot3out) at (5.4,2.4) {\scriptsize{$\cdots$}};
		\node (dot31out) at (6.6,2.4) {\scriptsize{$\cdots$}};
		
		\node (dotintervert) at (-6.25,1) {\scriptsize{$\cdots$}};
		\node (dotintervert) at (4.25,1) {\scriptsize{$\cdots$}};
		\foreach \from/\to in {in-1/n-1,in0/n0,in1/n1,in2/n2,in3/n3,n-1/out-1,n-1/out-12,n-1/out-11,n0/out0,n0/out02,n0/out01,n1/out1,n1/out111,n1/out12,n1/out11,n2/out2,n2/out21,n2/out22,n3/out3,n3/out31,n3/out311,n3/out32}
		\draw[gray!40!black] (\from) -- (\to);
		\draw [decorate,decoration={zigzag,amplitude=1.3pt,segment length=.8mm}]    (n-1) -- (-7,1);
		\draw [decorate,decoration={zigzag,amplitude=1.3pt,segment length=.8mm}]    (-5.5,1) -- (n0);
		\draw [decorate,decoration={zigzag,amplitude=1.3pt,segment length=.8mm}]    (n0) -- (n1);
		\draw [decorate,decoration={zigzag,amplitude=1.3pt,segment length=.8mm}]    (n1) -- (n2);
		\draw [decorate,decoration={zigzag,amplitude=1.3pt,segment length=.8mm}]    (n2) -- (3.5,1);
		\draw [decorate,decoration={zigzag,amplitude=1.3pt,segment length=.8mm}]    (5,1) -- (n3);
		\draw [decorate,decoration={brace,amplitude=3pt},xshift=0pt,yshift=-10pt,thick] (-9,3.2) -- (-7,3.2) node [above,midway] {\scriptsize{$J'_1$}};
		\draw [decorate,decoration={brace,amplitude=3pt},xshift=0pt,yshift=-10pt,thick] (-5,3.2) -- (-3,3.2) node [above,midway] {\scriptsize{$J'_{p-1}$}};
		\draw [decorate,decoration={brace,amplitude=3pt},xshift=0pt,yshift=-10pt,thick] (-2.5,3.2) -- (-1.1,3.2) node [above,midway] {\scriptsize{$K'_1$}};
		\draw [decorate,decoration={brace,amplitude=3pt},xshift=0pt,yshift=-10pt,thick] (-.9,3.2) -- (0.5,3.2) node [above,midway] {\scriptsize{$J^{(2)'}_p$}};
		\draw [decorate,decoration={brace,amplitude=3pt},xshift=0pt,yshift=-10pt,thick] (1,3.2) -- (3,3.2) node [above,midway] {\scriptsize{$K'_2$}};
		\draw [decorate,decoration={brace,amplitude=3pt},xshift=0pt,yshift=-10pt,thick] (4.5,3.2) -- (5.9,3.2) node [above,midway] {\scriptsize{$J^{(1)'}_{p}$}};
		\draw [decorate,decoration={brace,amplitude=3pt},xshift=0pt,yshift=-10pt,thick] (6.1,3.2) -- (7.5,3.2) node [above,midway] {\scriptsize{$K'_{q}$}};
	\end{tikzpicture}.
\end{equation*}

\begin{rem}
	In particular, one has: $W^{(j)}_1\circ_{u,1} W^{(k)}_1 =(-1)^{(j-u)(k-1)}  W^{(j+k-1)}_1$ for all $1\leq u \leq j$ and $k>0$. 
\end{rem}

	Since the words are subject to the cyclic relation \eqref{eq: cyclic relation}, the assignment \eqref{eq: compo word} defines all the infinitesimal compositions. 
	In turn, they produce the composition maps along trees of the properad $\mathsf{uDPois}^!$. 
	On the other hand, in \cite{Leray-Vallette}, it is proven that  the composition maps of the properad $\mathsf{uDPois}^!$ along graphs of positive genus are trivial, that is, it is a dioperad. 
	Therefore, we have defined the whole properadic structure. \\

Let $W_{a_1,\dots,a_m}^{J_1,\dots,J_m}$ be a cyclic word with $J_i=\{\mu^1_i<\mu^2_i<\cdots<\mu^{j_i}_i\}$, and let $j_i:=|J_i|$.  
Under the isomorphism $\mathsf{uDPois}^!\cong \mathsf{uBIB}^!$, if $J_m\neq \emptyset$, then the element corresponding to the class of $W_{a_1,\dots,a_m}^{J_1,\dots,J_m}$  is the class of the level tree 
	\begin{equation}\label{eq: image of word}
		BW_{a_1,\dots,a_m}^{J_1,\dots,J_m}:=
	\begin{tikzpicture}
		[scale=.3,auto=left,baseline=1cm]  
		\begin{scope}
			\node[above,yshift=-0.1cm] (out1k) at (6,8.5)  {\scriptsize{$C^{J_{m-1}}$}};
			\node (n0k) at (7.5,10.5) [] {\scriptsize{$C^{J_m}$}};
			
			\coordinate (itout1k) at (6,7.5)  {};
			\coordinate (itk) at (8.25,6.75)  {};
			
			\node (n1k) at (7.5,6) [colie] {};
			\node (n2k) at (7.5,7.5) [prelie] {};	
			
			\node[below,yshift=0.1cm] (in2k) at (7.5,5) {\scriptsize{$m$}};

			\node[above,yshift=-0.1cm] (out1) at (1,4)  {\scriptsize{$C^{J_{2}}$}};
			
			\coordinate (itout1) at (1,3)  {};
			\coordinate (it) at (3.25,2.25)  {};

			\node (n1) at (2.5,1.5) [colie] {};
			\node (n2) at (2.5,3) [prelie] {};
			\node (n3) at (-.5,0) [prelie] {};
			\node (nU) at (-2,1.1) [] {\scriptsize{$C^{J_1}$}};
			\node (nb1) at (-.5,-1.5) [colie] {};
			
			\coordinate (itoutb1) at (-2,0)  {};
			\coordinate (itb) at (.25,-.75)  {};
			
			\node[below,yshift=0.1cm] (in0) at (-2.75,-2) {\scriptsize{$1$}};	
			\node[below,yshift=0.1cm] (in1) at (-.5,-2.5) {\scriptsize{$2$}};
			\node[below,yshift=0.1cm] (in2) at (2.5,.5) {\scriptsize{$3$}};	
			
			\foreach \from/\to in {in0/n3,in2/n1,in1/nb1,n1/itout1,n2/n3,n1/itout1,itout1/out1,n1/it,it/n2,nb1/itoutb1,nb1/itb,itb/n3,n1k/in2k,n2k/n0k,
				n1k/itout1k,itout1k/out1k,n1k/itk,itk/n2k,nb1/itoutb1,nb1/itb,itb/n3,nU/itoutb1}
			\draw[gray!20!black] (\from) -- (\to);
			\draw [dashed] (n2) -- (n2k);
		\end{scope}
	\end{tikzpicture}
\end{equation}
where  $C^{J_i}$ is as in \eqref{eq: Cji}.

For instance, one has the following correspondence
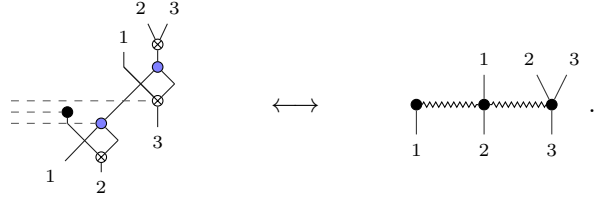
\begin{equation*}
\begin{tikzpicture}
	[scale=.3,auto=left,baseline=0.15cm]  
	\begin{scope}
		\node[above,yshift=-0.1cm] (out1) at (1,3.5)  {\scriptsize{$1$}};
		
		\node[above,yshift=-0.1cm] (out2) at (1.75,4.75)  {\scriptsize{$2$}};
		\node[above,yshift=-0.1cm] (out3) at (3.25,4.75)  {\scriptsize{$3$}};
		
		\coordinate (itout1) at (1,2.5)  {};
		\coordinate (it) at (3.25,1.75)  {};
		
		\node (n0) at (2.5,3.5) [colie] {};
		\node (n1) at (2.5,1) [colie] {};
		\node (n2) at (2.5,2.5) [prelie] {};
		\node (n3) at (0,0) [prelie] {};
		\node (nU) at (-1.5,.5) [lie] {};
		\node (nb1) at (0,-1.5) [colie] {};
		
		\coordinate (itoutb1) at (-1.5,0)  {};
		\coordinate (itb) at (.75,-.75)  {};
		\coordinate (itin3) at (4.5,2.5) {};
		
		\node[below,yshift=0.1cm] (in0) at (-2.25,-2) {\scriptsize{$1$}};	
		\node[below,yshift=0.1cm] (in1) at (0,-2.5) {\scriptsize{$2$}};
		\node[below,yshift=0.1cm] (in2) at (2.5,-.5) {\scriptsize{$3$}};	
		
		\foreach \from/\to in {in0/n3,in2/n1,in1/nb1,n1/itout1,n2/n3,n1/itout1,itout1/out1,n1/it,it/n2,n2/n0,n0/out2,n0/out3, nb1/itoutb1,nb1/itb,itb/n3,nU/itoutb1}
		\draw[gray!20!black] (\from) -- (\to);
		
		\draw (-4,1) node {} -- (n1) [dashed,gray!70!black] node {};
		\draw (-4,0) node {} -- (n3) [dashed,gray!70!black] node {};
		\draw (-4,.5) node {} -- (nU) [dashed,gray!70!black] node {};
	\end{scope}
\end{tikzpicture}
	\hspace{1cm}\longleftrightarrow \hspace{1cm}
	\begin{tikzpicture}
		[scale=.3,auto=left,baseline=0.2cm]  
		\begin{scope}[yscale=1,xscale=-1]
			\node (out1) at (-2,3)  {\scriptsize{$3$}};
			\node (out12) at (0,3)  {\scriptsize{$2$}};
			\node (out2) at (2,3) {\scriptsize{$1$}};
			\node (n1) at (-1,1) [circblack] {};
			\node (n2) at (2,1) [circblack] {};
			\node (n3) at (5,1) [circblack] {};
			\node (in1) at (-1,-1) {\scriptsize{$3$}};
			\node (in2) at (2,-1) {\scriptsize{$2$}};
			\node (in3) at (5,-1) {\scriptsize{$1$}};
			\foreach \from/\to in {in1/n1,in2/n2,n1/out1,n1/out12,n2/out2,in3/n3}
			\draw[gray!40!black] (\from) -- (\to);
			\draw [decorate,decoration={zigzag,amplitude=1pt,segment length=.8mm}]    (n1) -- (n2);
			\draw [decorate,decoration={zigzag,amplitude=1pt,segment length=.8mm}]    (n3) -- (n2);
		\end{scope}
	\end{tikzpicture}.
\end{equation*}
Here we have added dashed lines to emphasize the level of the vertices, which is essential for sign issues.

Observe that, in $\mathsf{uBIB}^!$, the element \eqref{eq: image of word} is equal to 
	\begin{equation}\label{eq: image of word 2}
			rBW_{a_1,\dots,a_m}^{J_1,\dots,J_m}:=
		(-1)^{n(0)}
	\scalebox{.85}{
		\begin{tikzpicture}
			[scale=.45,auto=left,baseline=0cm]  
			\begin{scope}[xscale=-1,yscale=1]
				\node (bin1) at (-5,5) {{$\hat{C}^{J_m}$}};
				\node (bin2) at (-0.25,3.75) { {$\hat{C}^{J_{m-1}}$}};
				\node (bint1) at (2.75,.75) { {$\hat{C}^{J_{2}}$}};
				\node (bint2) at (5.75,-2.25) { {$\hat{C}^{J_{1}}$}};
				
				\node (b2) at (-2,2) [] {$B'_{m-1}$};
				\node (bt1) at (1,-1) [] {$B'_{2}$};
				\node (bt2) at (4,-4) [] {$B'_{1}$};
				
				\node (bout2) at (-3.5,0.5) { {$a_m$}};
				
				\node (boutt1) at  (-.5,-2.5)  {$a_3$};
				\node (boutt2) at (2.5,-5.5) { {$a_2$}};
				\node (boutt3) at (5.5,-5.5) { {$a_1$}};
				
				\node (in1) at (1,0) {};
				\foreach \from/\to in {boutt3/bt2,boutt1/bt1, bout2/b2,boutt2/bt2,b2/bin1,b2/bin2,bt1/bint1,bt2/bint2,bt1/bt2}
				\draw[gray!20!black] (\from) -- (\to);
				\draw (b2) node {} -- (bt1) [dashed] node {};
			\end{scope}
		\end{tikzpicture}
	}
\end{equation}
where, if $j_i\neq 0$, then 
\begin{equation*}
	B_i=
		\begin{tikzpicture}
		[scale=.3,auto=left,baseline=0.45cm]  
		\begin{scope}[yscale=1,xscale=-1]
			\node[below,yshift=0.1cm] (in1) at (1,0) {};
			\node[below,yshift=0.1cm] (in2) at (2.5,0) {};
			\coordinate (itin2) at (2.5,1) {};
			\node (n1) at (1,1) [colie] {};
			\node (n2) at (1,2.5) [prelie] {};;
			\coordinate (itout1) at (0.25,1.75)  {};
			\coordinate (itout2) at (2.5,2.5)  {};
			\node[above,yshift=-0.1cm] (out1) at (1,3.5)  {};
			\node[above,yshift=-0.1cm] (out2) at (2.5,3.5)  {};
			
			\foreach \from/\to in {in1/n1,n1/itout1,itout1/n2,n2/out1,in2/itin2,itin2/n2,n1/itout2,itout2/out2}
			\draw[gray!20!black] (\from) -- (\to);
		\end{scope}
	\end{tikzpicture}
	\hspace{.25cm}
	\text{ and } 
	\hspace{.25cm}
	\hat{C}^{J_i}=C^{J_1},
\end{equation*}
 but if $j_i=0$, then  
\begin{equation*}
	B'_i = 	\begin{tikzpicture}
		[scale=.35,auto=left,baseline=0.25cm]  
		\node[below,yshift=0.1cm] (in1) at (0,0) {};
		\node[below,yshift=0.1cm] (in2) at (2,0) {};
		\node (n1) at (1,1) [prelie] {};
		\node (out1) at (1,2)  {};
		
		\foreach \from/\to in {in1/n1,in2/n1,n1/out1}
		\draw[gray!20!black] (\from) -- (\to);
	\end{tikzpicture}
	\end{equation*}
	and both $\hat{C}^{J_i}$ and its output edge are omitted.  
The sign is given by the number $n(0)$ of $j_i$'s that are zero. 

For instance, 
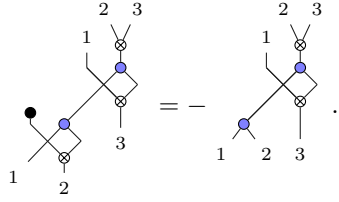
\begin{equation*}
\begin{tikzpicture}
	[scale=.3,auto=left,baseline=0.15cm]  
	\begin{scope}
		\node[above,yshift=-0.1cm] (out1) at (1,3.5)  {\scriptsize{$1$}};
		
		\node[above,yshift=-0.1cm] (out2) at (1.75,4.75)  {\scriptsize{$2$}};
		\node[above,yshift=-0.1cm] (out3) at (3.25,4.75)  {\scriptsize{$3$}};
		
		\coordinate (itout1) at (1,2.5)  {};
		\coordinate (it) at (3.25,1.75)  {};
		
		\node (n0) at (2.5,3.5) [colie] {};
		\node (n1) at (2.5,1) [colie] {};
		\node (n2) at (2.5,2.5) [prelie] {};
		\node (n3) at (0,0) [prelie] {};
		\node (nU) at (-1.5,.5) [lie] {};
		\node (nb1) at (0,-1.5) [colie] {};
		
		\coordinate (itoutb1) at (-1.5,0)  {};
		\coordinate (itb) at (.75,-.75)  {};
		\coordinate (itin3) at (4.5,2.5) {};
		
		\node[below,yshift=0.1cm] (in0) at (-2.25,-2) {\scriptsize{$1$}};	
		\node[below,yshift=0.1cm] (in1) at (0,-2.5) {\scriptsize{$2$}};
		\node[below,yshift=0.1cm] (in2) at (2.5,-.5) {\scriptsize{$3$}};	
		
		\foreach \from/\to in {in0/n3,in2/n1,in1/nb1,n1/itout1,n2/n3,n1/itout1,itout1/out1,n1/it,it/n2,n2/n0,n0/out2,n0/out3, nb1/itoutb1,nb1/itb,itb/n3,nU/itoutb1}
		\draw[gray!20!black] (\from) -- (\to);
		
	\end{scope}
\end{tikzpicture}
	= -
\begin{tikzpicture}
	[scale=.3,auto=left,baseline=0.15cm]  
	\begin{scope}
		\node[above,yshift=-0.1cm] (out1) at (1,3.5)  {\scriptsize{$1$}};
		
		\node[above,yshift=-0.1cm] (out2) at (1.75,4.75)  {\scriptsize{$2$}};
		\node[above,yshift=-0.1cm] (out3) at (3.25,4.75)  {\scriptsize{$3$}};
		
		\coordinate (itout1) at (1,2.5)  {};
		\coordinate (it) at (3.25,1.75)  {};
		
		\node (n0) at (2.5,3.5) [colie] {};
		\node (n1) at (2.5,1) [colie] {};
		\node (n2) at (2.5,2.5) [prelie] {};
		\node (n3) at (0,0) [prelie] {};
		
		\coordinate (itin3) at (4.5,2.5) {};
		
		\node[below,yshift=0.1cm] (in0) at (-1,-1) {\scriptsize{$1$}};	
		\node[below,yshift=0.1cm] (in1) at (1,-1) {\scriptsize{$2$}};
		\node[below,yshift=0.1cm] (in2) at (2.5,-1) {\scriptsize{$3$}};	
		
		\foreach \from/\to in {in0/n3,in2/n1,in1/n3,n3/n2,n1/itout1,n2/n3,n1/itout1,itout1/out1,n1/it,it/n2,n2/n0,n0/out2,n0/out3}
		\draw[gray!20!black] (\from) -- (\to);
	\end{scope}
\end{tikzpicture}.
\end{equation*}
\begin{defn}
	The level trees of the form \eqref{eq: image of word 2} are called \emph{reduced}.
\end{defn}

Consider the sub $\Sigma$--bimodule of $\mathsf{uBIB}^!$ given by 
\begin{equation*}
	\mathsf{uBIB}_{>0}^!(m,n) 
	:=
	\begin{cases}
			\mathsf{uBIB}^!(m,n)    & \text{ if } n>0, \\
			0 				    	& \text{ if } n=0.
	\end{cases}
\end{equation*}
\begin{lem}
The sub $\Sigma$--bimodule	$\mathsf{uBIB}_{>0}^!$ is a sub-properad of $\mathsf{uBIB}^!$. 
\end{lem}
\begin{proof}
	This follows directly from the definition of a properad. 
	It also follows from the following lemma. 
\end{proof}

Let  $\iota\co  \bib^! \to \mathsf{uBIB}^!$ be the canonical morphism of properads. 
 
\begin{lem}
The sub $\Sigma$--bimodule $\mathsf{uBIB}_{>0}^!$ is equal to the image of $\iota$.
\end{lem}
\begin{proof}
	On the one hand, if $Im(\iota)$ denotes the image of $\iota$, one has $Im(\iota) \subset \mathsf{uBIB}_{>0}^!$ since all components $\bib^!(m,0)$ are trivial. 
	On the other hand, $\iota$ is surjective on $\mathsf{uBIB}_{>0}^!$ since any cyclic word has a reduced representative. 
\end{proof}

\subsection{A combinatorial description of $\bib^!$}

Let 
\begin{equation*}
	\redd\co  \mathsf{uBIB}_{>0}^! \to \bib^!
\end{equation*}
be the reduction map: 
for any cyclic word take a representative of the form $BW_{a_1,\dots,a_m}^{J_1,\dots,J_m}$ such that $J_m\neq \emptyset$ (this is always possible in $\mathsf{uBIB}_{>0}^!$). 
Define 
\begin{equation*}
	\redd(BW_{a_1,\dots,a_m}^{J_1,\dots,J_m}):= rBW_{a_1,\dots,a_m}^{J_1,\dots,J_m}. 
\end{equation*} 
The next three Lemmata \ref{lem: red well-defined}, \ref{lem: compo genus} and \ref{lem: red morphism} claim that  $\redd$ is  well-defined and a morphism of properads. 

\begin{thm}\label{prop: iota inclusion}
	The canonical morphism $\iota\co  \bib^! \to \mathsf{uBIB}^!$ is an inclusion. 
\end{thm}
\begin{proof}%
	Supposing Lemmata \ref{lem: red well-defined}, \ref{lem: compo genus} and \ref{lem: red morphism} hold true, we show that the morphism $\iota$ is injective:  
	It is easy to see that for each (of the two) generator $x\in \bib^!$, one has $\redd(\iota(x))=x$. 
	Since $\redd$ and $\iota$ are morphisms of properads, it follows that $\redd\circ \iota=id$. 
\end{proof}

\begin{rem}
	In the recent book \cite{DSV-book}, V. Dotsenko, S. Shadrin and B. Vallette defined the notion of \emph{extendable}  operad, which essentially refers to any operad that can be embedded into a non-trivial unital extension of it. 
	Given a Koszul quadratic operad $\mathsf{P}$, the category of homotopy $\mathsf{P}$--algebras is called \emph{twistable} if the Koszul dual operad $\mathsf{P}^!$ is extendable. 
	This means that homotopy $\mathsf{P}$--algebras admits an interesting curved extension; see Definition 4.24 and the discussion that follows it in \textit{loc. cit.} 
	Theorem \ref{prop: iota inclusion} says that $\bib^!$ is an extendable properad. 
	It would be interesting to understand the twisting procedure for homotopy $\bib$--algebras. 
\end{rem}

\begin{cor}
	The composition maps of the properad $\bib^!$ along the graphs of positive genus are trivial \ie it is a dioperad. 
\end{cor}
\begin{proof}
It follows from the fact that $\mathsf{uDPois}^!\cong \mathsf{uBIB}^!$ is a dioperad and that the morphism $\iota$ is an inclusion.
\end{proof}
\begin{rem}
	This property implies that the decomposition maps of the Koszul dual coproperad $\bib^\antishrieck$ produces only graphs of genus zero \ie it is a codioperad. 
	In particular, the differential of the cobar construction $\cobd{\bib^\antishrieck}$ decomposes (the dual of) any  reduced cyclic word into sums of reduced cyclic words related by one edge. 
\end{rem}

The rest of this subsection is devoted to the Lemmata \ref{lem: red well-defined}, \ref{lem: compo genus} and \ref{lem: red morphism}.  

\begin{lem}\label{lem: red well-defined}
	The map $\redd$ is well-defined. 
\end{lem}
\begin{proof}
	Recall that the cyclic words $BW_{a_1,\dots,a_m}^{J_1,\dots,J_m}$  form a basis of  $\mathsf{uBIB}_{>0}^!(m,n)$. 
	Therefore, what we have to show is that the reduction is compatible with the cyclic relation. 
	That is, let us show that for all cyclic word $BW_{1,\dots,m}^{(j_1),\dots,(j_m)}$  such that $j_m,j_s\neq 0$, one has  
	\begin{equation*}
		rBW_{1,\dots,m}^{(j_1),\dots,(j_m)} =(-1)^{(j_1+ \cdots +j_s)n+ j_1+..+j_{s-1}+j_m} rBW_{s+1,\dots,m,1,\dots,s}^{(j_{s+1}),\dots,(j_m),(j_1),\dots,(j_s)}
	\end{equation*} 
	in $\bib^!(m,n)$. 
	
	In fact, we will show this for the $j_i$'s being at most 1; the general case easily follows. 
	As a technical notation, for all $1\leq i\leq m$, let us set $[i]_J$ to be the cardinal of $\{ 1\leq k\leq i ~|~ j_k\neq 0 \}$. 
	Note that, up to the sign $(-1)^{n(0)}$ (which is not relevant here),   $rBW_{1,\dots,m}^{(j_1),\dots,(j_m)}$ is the class of the following level tree
	\begin{equation*}
		\scalebox{1}
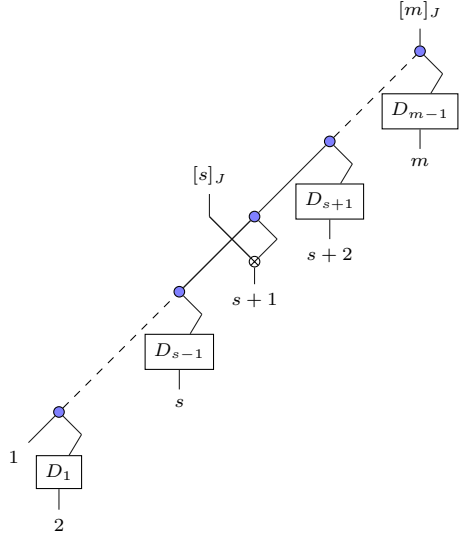
{
			\begin{tikzpicture}
				[scale=.4,auto=left,baseline=-3cm]  
				\begin{scope}[xscale=1,yscale=1]
					\node[above,yshift=-0.1cm] (out1) at (1,3.5)  {\scriptsize{$[s]_J$}};
					\node[above,yshift=-0.1cm] (outn) at (8,9)  {\scriptsize{$[m]_J$}};
					
					\coordinate (itout1) at (1,2.5)  {};
					
					\coordinate (itn) at (8.75,7.25) {};
					\coordinate (its1) at (5.75,4.25) {};					
					\coordinate (it) at (3.25,1.75)  {};
					\coordinate (it1s) at (.75,-.75)  {};
					\coordinate (it1) at (-3.25,-4.75)  {};
					
					\node (ndn) [draw] at (8,6) {\scriptsize{$D_{m-1}$}};
					\node (nds1) [draw] at (5,3) {\scriptsize{$D_{s+1}$}};
					\node (nds) at (2.5,1) [colie] {};
					\node (nd1s) [draw] at (0,-2) {\scriptsize{$D_{s-1}$}};
					\node (nd1) [draw] at (-4,-6) {\scriptsize{$D_1$}};
					
					\node (nn) at (8,8) [prelie] {};
					\node (ns1) at (5,5) [prelie] {};
					\node (ns) at (2.5,2.5) [prelie] {};
					\node (n1s) at (0,0) [prelie] {};
					\node (n1) at (-4,-4) [prelie] {};
					\coordinate (itin1) at (4.5,2.5) {};
					
					\node[below,yshift=0.1cm] (inn) at (8,4.5) {\scriptsize{$m$}};	
					\node[below,yshift=0.1cm] (ins1) at (5,1.5) {\scriptsize{$s+2$}};	
					\node[below,yshift=0.1cm] (ins) at (2.5,0) {\scriptsize{$s+1$}};	
					\node[below,yshift=0.1cm] (in1s) at (0,-3.5) {\scriptsize{$s$}};	
					\node[below,yshift=0.1cm] (in1) at ((-4,-7.5) {\scriptsize{$2$}};	
					\node[below,yshift=0.1cm] (in0) at (-5.5,-5.25) {\scriptsize{$1$}};	
					
					\foreach \from/\to in {in1/nd1,in1s/nd1s,ins1/nds1,inn/ndn,
						in0/n1,nd1/it1,nd1s/it1s,ins/nds,nds1/its1,ndn/itn,
						it1/n1,it1s/n1s,its1/ns1,itn/nn,
						n1s/ns,ns/ns1,nds/itout1,ns/n1s,nds/itout1,itout1/out1,nds/it,it/ns,nn/outn}
					\draw[gray!20!black] (\from) -- (\to);
					
					\draw (ns1) node {} -- (nn) [dashed] node {};
					\draw (n1) node {} -- (n1s) [dashed] node {};
				\end{scope}		
			\end{tikzpicture}
		}
	\end{equation*}
	where the boxes labeled by $D_i$ are either  
	$\begin{tikzpicture}
		[scale=.35,auto=left,baseline=0.35cm]  
		\node (in1) at (1,0) {};
		\node (n2) at (1,1) [colie] {};
		\node[above,yshift=-0.1cm] (out1) at (0,2) {\scriptsize{$[i]_J$}};
		\node[above,yshift=-0.1cm] (out2) at (2,2) {};
		
		\foreach \from/\to in {in1/n2,n2/out1,n2/out2}
		\draw[gray!20!black] (\from) -- (\to);
	\end{tikzpicture}$ if $j_i\neq 0$ or an edge if $j_i=0$. 
	We proceed in 6 steps.
	\begin{enumerate}
		\item We move down all the boxes $D_i$'s to place them below all the blue vertices, keeping their relative height.   
		In the upper part, one has therefore only blue vertices and the crossed vertex with output labeled by $s+1$. 
		Let us represent the first $s-1$ blue vertices by a block $D$, and the last blue $m-s-1$ vertices by  a block $U$, letting alone the $s$-th blue vertex and the crossed vertex. 
		The upper part is thus represented by the level tree on the left in the following equality: 
		
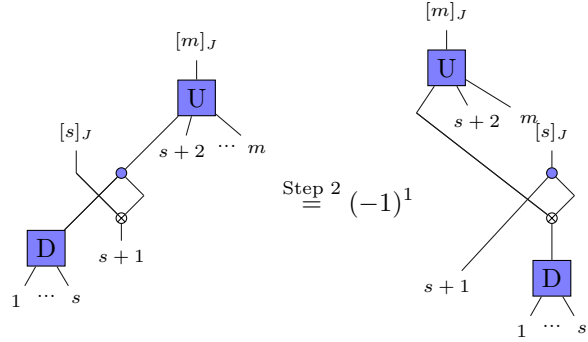
\begin{equation*}
			\begin{tikzpicture}
				[scale=.4,auto=left,baseline=.5cm]  
				\begin{scope}[xscale=1,yscale=1]
					\node[above,yshift=-0.1cm] (out1) at (1,3.5)  {\scriptsize{$[s]_J$}};
					\node[above,yshift=-0.1cm] (outn) at (5,6.5)  {\scriptsize{$[m]_J$}};
					\coordinate (itout1) at (1,2.5)  {};
					\coordinate (it) at (3.25,1.75)  {};
					
					\node (nds) at (2.5,1) [colie] {};
					
					\node (ns1)  [draw] at (5,5) [fill=blue!50] {U};
					\node (ns) at  (2.5,2.5) [prelie] {};
					\node (n1s) [draw]   at (0,0) [fill=blue!50] {D};
					
					\coordinate (itin1) at (4.5,2.5) {};
					
					\node[below,yshift=0.1cm] (inn) at (7,3.5) {\scriptsize{$m$}};	
					\node[below,yshift=0.1cm] () at (6,3.5) {\scriptsize{$...$}};	
					\node[below,yshift=0.1cm] (ins1) at (4.5,3.5) {\scriptsize{$s+2$}};	
					\node[below,yshift=0.1cm] (ins) at (2.5,0) {\scriptsize{$s+1$}};	
					\node[below,yshift=0.1cm] (in1s) at (1,-1.5) {\scriptsize{$s$}};	
					\node[below,yshift=0.1cm] () at (0,-1.5) {\scriptsize{$...$}};	
					\node[below,yshift=0.1cm] (in0) at (-1,-1.5) {\scriptsize{$1$}};	
					
					\foreach \from/\to in {in1s/n1s,ins1/ns1,ins/nds,in0/n1s,inn/ns1,
						n1s/ns,ns/ns1,nds/itout1,ns/n1s,nds/itout1,itout1/out1,nds/it,it/ns,
						ns1/outn}
					\draw[gray!20!black] (\from) -- (\to);
				\end{scope}		
			\end{tikzpicture}
			\stackrel{\text{Step 2}}{=} (-1)^{1}
			\begin{tikzpicture}
				[scale=.4,auto=left,baseline=.5cm]  
				\begin{scope}[xscale=1,yscale=1]
					
					\node[above,yshift=-0.1cm] (outn) at (-1,7.5)  {\scriptsize{$[m]_J$}};
					\node (ns1)  [draw] at (-1,6) [fill=blue!50] {U};
					\node[below,yshift=0.1cm] (inn) at (1.75,4.5) {\scriptsize{$m$}};	
					\node[below,yshift=0.1cm] (ins1) at (0,4.5) {\scriptsize{$s+2$}};	
					
					\node[above,yshift=-0.1cm] (out1) at (2.5,3.5)  {\scriptsize{$[s]_J$}};
					\coordinate (itout1) at (-2,4.5)  {};							
					\coordinate (it) at (3.25,1.75)  {};
					\node (nds) at (2.5,1) [colie] {};
					\node (ns) at  (2.5,2.5) [prelie] {};

					\node (n1s) [draw]   at (2.5,-1) [fill=blue!50] {D};				
					\node[below,yshift=0.1cm] (ins) at (-1,-1) {\scriptsize{$s+1$}};	
					\node[below,yshift=0.1cm] (in1s) at (3.5,-2.5) {\scriptsize{$s$}};	
					\node[below,yshift=0.1cm] () at (2.5,-2.5) {\scriptsize{$...$}};	
					\node[below,yshift=0.1cm] (in0) at (1.5,-2.5) {\scriptsize{$1$}};	
					
					\foreach \from/\to in {in1s/n1s,ins1/ns1,ins/ns,in0/n1s,inn/ns1,
						n1s/nds,nds/itout1,nds/itout1,itout1/ns1,nds/it,it/ns,
						ns/out1,ns1/outn}
					\draw[gray!20!black] (\from) -- (\to);
				\end{scope}		
			\end{tikzpicture}
		\end{equation*} 
		The boxes $D_1,\dots,D_{m-1}$ (not represented in the above picture) are linked to the outputs $2,\dots,s,s+2,\dots,m-1$ respectively. 
		\item We use the first relation of \eqref{rel Koszul 2}.
		\item We use relation \eqref{M1} a) iteratively in order to ``move" the block $U$  below the block $B_s$. 
		\item We exchange the height of the block $U$ with the block $D$. 
		\item We use relation \eqref{M1} b) iteratively in order to ``move" the block $D$ above the block $B_s$.  
		\item We move back the $D_i$'s. 
	\end{enumerate}
		Note that if $m=2$, in which case there is no $D$ nor $U$, the proof amounts to the sole use of \eqref{rel Koszul 2}. 
		
	The sign is obtained as follows: 
	First recall that $\lambda=-1$, so the blue vertex has degree $0$ and the crossed vertex has degree $-1$. 
	Let us split the boxes $D_i$'s in two blocks: $A$ stands for the block gathering $D_1,\dots,D_{s-1}$ and $B$ stands for $D_{s+1},\dots,D_{m-1}$. 
	In step 1, the unique movement that produces a sign is that $B$ passes through the level of $B_s$ (which has degree one, since it has exactly one crossed vertex). 
	Therefore this step produces the sign $(-1)^{B}$. 
	In Step 6, signs are produced by the fact that heights of $A$ and $B$ are exchanged and by the fact  that $A$ passes through the level of $B_s$. 
	Therefore, one has a sign $(-1)^{AB+ A}$. 
	In conclusion, Step 1 and Step 6 give the sign $(-1)^{AB + A+B} = (-1)^{(A-1)(B-1) +1} 
	= (-1)^{([s-1]_J-1)([m-1]_J-[s]_J-1) +1} $. 
	\\
	Step 2 to 5 gives $(-1)^{ 1}$ if $s<m$.
	
	The final sign is then given by the parity of $([s-1]_J-1)([m-1]_J-[s]_J-1)$. 
	If, among the $j_i$'s, $r$ of them are zero, $r_1$ being between $1$ and $s-1$, and $r_2$ being between $s+1$ and $m-1$, then, modulo 2, one has 
	 \begin{multline*}
	 	([s-1]_J-1)([m-1]_J-[s]_J-1) 
	 	=(s-1-r_1-1)(m-1-r-s-r_1-1) 
	 	= (s-r_1)(m-r_2-s)\\
	 	= (s-r_1)(m-r_2) + (1-r_1)s.
	 \end{multline*} 
	 On the other hand, modulo 2, one has 
	 \begin{multline*}
	(j_1+ \cdots +j_s)n+ j_1+..+j_{s-1}+j_m 
	= (s-r_1)(m-r) + s-1-r_1+1
	= (s-r_1)(m-r) + s-1-r_1+1 \\
	=(s-r_1)(m-r_2) + (s-r_1)r_1+  s-r_1
	= (s-r_1)(m-r_2) + s(r_1-1). 	
	 \end{multline*}
\end{proof}

\begin{lem}\label{lem: compo genus}
	 In $\bib^!$, any infinitesimal composition of reduced cyclic words along graphs of positive genus is zero. 
\end{lem}

 \begin{proof}
	We will prove the statement for compositions along graphs of genus one, the general case easily follows. 
	In other words, we have to show the following Claim \ref{claim 0}. 
	
	\begin{claim}\label{claim 0}
		For all reduced cyclic words $rBW_1\in \bib^!(p,n)$ and $rBW_2\in \bib^!(q,n')$, 
		and for all injections $\alpha\co \{1,2 \}\to \{1,\dots, n\}$ and  $\beta\co \{1,2\}\to \{1,\dots, q\}$, one has 
		\begin{equation*}
			rBW_1 \circ_{\alpha,\beta} rBW_2=0. 
		\end{equation*}
	\end{claim}
	In fact, this follows from the following Claim \ref{claim 1}. 	
	
	\begin{claim} \label{claim 1}
	Let $rBW_1:= rBW_{1,\dots,p}^{(j_1),(0),\dots,(0),(j_p)}$ and $rBW_2:= rBW_{1,\dots,q}^{(1),(0),\dots,(0),(1)}$ be two reduced cyclic words in $\mathsf{BIB}^!$  such that $j_1,j_p\neq 0$. %
	For any injective maps 	$\alpha\co \{1,2 \}\to \{1,\dots,j_1+j_p\}$ and  $\beta\co \{1,2\}\to \{1,\dots,q\}$ such that  
	$\alpha(1)\in \{1,\dots,j_1\}$ and $\alpha(2)\in \{j_1+1,\dots,j_1+j_p\}$ (it is $\{1,...,j_1\}$ if $p=1$),  
	$\beta(1)\in \{1,2\}$ and $\beta(2)=q$, one has $rBW_1 \circ_{\alpha,\beta} rBW_2=0$. 
	\end{claim}
	\begin{proof}[Proof of Claim \ref{claim 0} supposing Claim \ref{claim 1}]
	Consider $rBW_1, rBW_2, \alpha$ and $\beta$ as in Claim \ref{claim 0}; write $rBW_2$ as $rBW_{1,...,q}^{K_1,...,K_q}$. 
	Suppose $\beta(1)<\beta(2)$ (this is always possible since there is no condition on $\alpha$). 
	
	First, let us see that $rBW_2$ and $\beta$ can be supposed as in Claim \ref{claim 1}. 
	If $\beta(1)=k>2$ then the part of $rBW_2$ that is made up of $B_i$ and $\hat{C}^{K_i}$ for $1\leq i\leq k-1$  can be moved below $rBW_1$; therefore it does not interfere with the result. 
	Therefore, one may suppose $\beta(1)\in \{1,2\}$. 
	Similarly, one may suppose $\beta(2)=q$. 
	Finally, it remains to show that one can suppose  $rBW_2$ of the form $rBW_{1,\dots,q}^{(1),(0),\dots,(0),(1)}$. 
	This is clear since for each $2\leq i\leq k-1$ such that $K_i\neq \emptyset$, the subtree of $rBW_2$ made up of the crossed vertices of $B_i$ and $\hat{C}^{K_i}$ can be moved below $rBW_1$. 
	
	Now, let us see that $rBW_1$ and $\alpha$ can be supposed as in Claim \ref{claim 1}. 
	Write  $rBW_1$ as $rBW_{1,...p}^{J_1,...,J_p}$. 
	\begin{itemize}
		\item If $p=1$, then there is nothing to do. 
		\item Suppose $p\geq 2$.	
		\begin{itemize}
		\item If $\alpha(1)$ and $\alpha(2)$ belong to the same $J_k$, then one can actually suppose $p=1$. 
		Indeed, for the two maximal subtrees of  $rBW_1$ that are not  $\hat{C}^{J_k}$, one has: one is below  $\hat{C}^{J_k}$, and the other can be moved above $rBW_2$.  
		\item If $\alpha(1)$ and $\alpha(2)$ belong to two different $J_i's$, then, by using the cyclic property on $rBW_1$ (Lemma \ref{lem: red well-defined}), one may suppose $\alpha(1)<\alpha(2)$ and $\alpha(2)\in J_p$. 
		As before, one may suppose that $\alpha(1)\in J_1$. 
		\end{itemize}
	\end{itemize}
	\end{proof}
	
	\begin{proof}[Proof of Claim \ref{claim 1}]
	The general idea to prove this claim consists in using elementary moves (that is, the relations of \eqref{rel Koszul 1} and \eqref{rel Koszul 2}) to obtain a (twisted or not) pencil box, which by Lemma \ref{lem: pencil box} is zero. 
	One may distinguish two cases: 
	\begin{enumerate}
		\item If $p=1$, then $j_p=j_1$, and therefore $\alpha(2)\in\{1,\dots,j_1\}$. 
		By using the associativity of the crossed vertices, one obtains a pencil box, which is a twisted one if $\alpha(2)<\alpha(1)$. 
		\item If $p\geq 2$, one obtains a tree of one of the following forms, in which the red and green dashed lines represent the two possibilities $\beta(1)=1$ or $\beta(1)=2$:   
		\begin{equation*}
			a) 
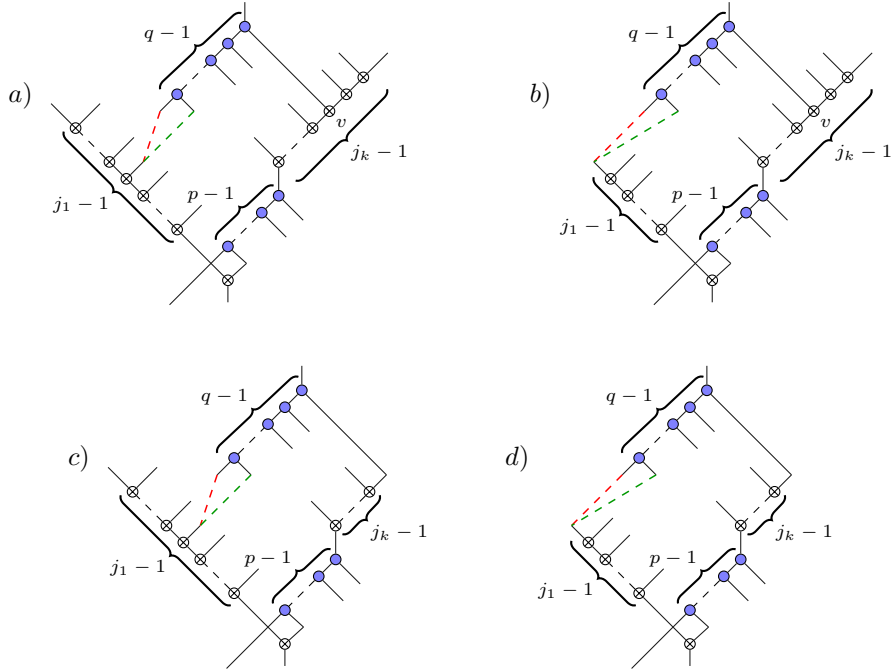
\begin{tikzpicture}
				[scale=.45,auto=left,baseline=-1cm]  
				\begin{scope}[xscale=1,yscale=1]
					
					\node (in) at  (0,1) [] {};
					\node (n4) at  (0,0) [prelie] {};
					\node (n3) at  (-.5,-.5) [prelie] {};
					\node (n2) at  (-1,-1) [prelie] {};
					\node (n1) at  (-2,-2) [prelie] {};
					
					\coordinate (od1) at  (-1.5,-2.5);
					\node (od2) at  (0,-2) [] {};
					\node (od3) at  (.5,-1.5) [] {};
					
					\coordinate (intL) at (-2.5,-2.5);
					\coordinate (ndin) at (2.5,-2.5);
					\node (nd4) at  (3.5,-1.5) [colie] {};
					\node (nd3) at  (3,-2) [colie] {};
					\node (nd2) at  (2.5,-2.5) [colie] {};
					\node (nd1) at  (2,-3) [colie] {};
					\node (nd0c) at  (1,-4) [colie] {};
					
					\node (v) at  (2.85,-2.85) [] {\scriptsize{$v$}};
					
					\node (idb) at  (-.5,-3.5) [] {};
					\node (idc) at  (0,-3) [] {};
					\node (id1) at  (1,-2) [] {};
					\node (id3) at  (2,-1) [] {};
					\node (id4) at  (2.5,-.5) [] {};
					\node (id4b) at  (4.5,-.5) [] {};
					
					\node (m1) at  (1,-5) [prelie] {};
					\node (m2) at  (.5,-5.5) [prelie] {};
					\node (m3) at  (-.5,-6.5) [prelie] {};
					
					\node (outm1) at  (2,-6) [] {};
					\node (outm2) at  (1.5,-6.5) [] {};
					\node (outm3) at  (-2.5,-8.5) [] {};
					
					\coordinate (intm) at (.05,-7);
					
					\node (mb1) at  (-.5,-7.5) [colie] {};
					\node (mb2) at  (-2,-6) [colie] {};
					\node (mb3) at  (-3,-5) [colie] {};
					\node (mb4) at  (-3.5,-4.5) [colie] {};
					\node (mb5) at  (-4,-4) [colie] {};
					\node (mb6) at  (-5,-3) [colie] {};
					
					\node (inmb6l) at  (-6,-2) [] {};
					\node (inmb6) at  (-4,-2) [] {};
					\node (inmb5) at  (-3,-3) [] {};
					\node (inmb3) at  (-2,-4) [] {};
					\node (inmb2) at  (-1,-5) [] {};
					
					\coordinate (intmb) at (-3,-4);
					\coordinate (out) at (-.5,-8.15);
					
					\foreach \from/\to in {	in/n4,
						od1/n1,od2/n2,od3/n3,
						idc/nd0c,id1/nd1,nd3/id3,nd4/id4,nd4/id4b,
						m1/nd0c,
						m1/m2,
						outm1/m1,outm2/m2,outm3/m3,
						m3/intm,intm/mb1,
						intmb/mb4,mb1/mb2,mb3/mb4,mb4/mb5,
						inmb6l/mb6,inmb6/mb6,inmb5/mb5,inmb3/mb3,inmb2/mb2,
						nd1/ndin,nd2/nd3,nd3/nd4,
						intL/n1,n2/n3,n3/n4,n4/ndin,
						out/mb1}
					\draw[gray!20!black] (\from) -- (\to);
					\draw[dashed] (nd1) -- (nd0c);
					\draw[dashed] (n1) -- (n2);
					\draw[dashed] (m2) -- (m3);
					\draw[dashed] (mb2) -- (mb3);
					\draw[dashed] (mb5) -- (mb6);
					\draw[line width=.2mm,red,dashed] (intL) -- (intmb);
					\draw[line width=.2mm,green!60!black,dashed] (od1) -- (intmb);
					
					\draw [decorate,decoration={brace,amplitude=3pt},xshift=0pt,yshift=-10pt,thick] (-2.5,-1.5) -- (-.1,.75) node [midway] {\scriptsize{$q-1$}};
					\draw [decorate,decoration={brace,amplitude=3pt},xshift=0pt,yshift=-10pt,thick] (-.85,-5.95) -- (.85,-4.35) node [midway] {\scriptsize{$p-1$}};
					\draw [decorate,decoration={brace,amplitude=3pt},xshift=0pt,yshift=-10pt,thick] (4.2,-1.5) -- (1.5,-4.2) node [midway] {\scriptsize{$j_k-1$}};
					\draw [decorate,decoration={brace,amplitude=3pt},xshift=0pt,yshift=-10pt,thick] (-2.1,-6.) -- (-5.35,-2.75) node [midway] {\scriptsize{$j_1-1$}};
				\end{scope}		
			\end{tikzpicture}
			\hspace{1.5cm}
				b)
				\begin{tikzpicture}
					[scale=.45,auto=left,baseline=-1cm]  
					\begin{scope}[xscale=1,yscale=1]
						
						\node (in) at  (0,1) [] {};
						\node (n4) at  (0,0) [prelie] {};
						\node (n3) at  (-.5,-.5) [prelie] {};
						\node (n2) at  (-1,-1) [prelie] {};
						\node (n1) at  (-2,-2) [prelie] {};
						
						\coordinate (od1) at  (-1.5,-2.5);
						\node (od2) at  (0,-2) [] {};
						\node (od3) at  (.5,-1.5) [] {};
						
						\coordinate (intL) at (-2.5,-2.5);
						\coordinate (ndin) at (2.5,-2.5);
						\node (nd4) at  (3.5,-1.5) [colie] {};
						\node (nd3) at  (3,-2) [colie] {};
						\node (nd2) at  (2.5,-2.5) [colie] {};
						\node (nd1) at  (2,-3) [colie] {};
						\node (nd0c) at  (1,-4) [colie] {};
						
						\node (v) at  (2.85,-2.85) [] {\scriptsize{$v$}};

						\node (idb) at  (-.5,-3.5) [] {};
						\node (idc) at  (0,-3) [] {};
						\node (id1) at  (1,-2) [] {};
						\node (id3) at  (2,-1) [] {};
						\node (id4) at  (2.5,-.5) [] {};
						\node (id4b) at  (4.5,-.5) [] {};
						
						\node (m1) at  (1,-5) [prelie] {};
						\node (m2) at  (.5,-5.5) [prelie] {};
						\node (m3) at  (-.5,-6.5) [prelie] {};
						
						\node (outm1) at  (2,-6) [] {};
						\node (outm2) at  (1.5,-6.5) [] {};
						\node (outm3) at  (-2.5,-8.5) [] {};
						
						\coordinate (intm) at (.05,-7);
						
						\node (mb1) at  (-.5,-7.5) [colie] {};
						\node (mb2) at  (-2,-6) [colie] {};
						\node (mb3) at  (-3,-5) [colie] {};
						\node (mb4) at  (-3.5,-4.5) [colie] {};
						\node (mbp) at  (-2.5,-3.5) [] {};
						
						\node (inmb6) at  (-4,-2) [] {};
						\node (inmb5) at  (-3,-3) [] {};
						\node (inmb3) at  (-2,-4) [] {};
						\node (inmb2) at  (-1,-5) [] {};
						
						\coordinate (intmb) at (-4,-4);
						\coordinate (out) at (-.5,-8.15);
						
						\foreach \from/\to in {	in/n4,
							od1/n1,od2/n2,od3/n3,
							idc/nd0c,id1/nd1,nd3/id3,nd4/id4,nd4/id4b,
							m1/nd0c,
							m1/m2,
							outm1/m1,outm2/m2,outm3/m3,
							m3/intm,intm/mb1,
							intmb/mb4,mb1/mb2,mb3/mb4,mb4/mbp,
							inmb3/mb3,inmb2/mb2,
							nd1/ndin,nd2/nd3,nd3/nd4,
							intL/n1,n2/n3,n3/n4,n4/ndin,
							out/mb1}
						\draw[gray!20!black] (\from) -- (\to);
						\draw[dashed] (nd1) -- (nd0c);
						\draw[dashed] (n1) -- (n2);
						\draw[dashed] (m2) -- (m3);
						\draw[dashed] (mb2) -- (mb3);
						
						\draw[line width=.2mm,red,dashed] (intL) -- (intmb);
						\draw[line width=.2mm,green!60!black,dashed] (od1) -- (intmb);
						
						\draw [decorate,decoration={brace,amplitude=3pt},xshift=0pt,yshift=-10pt,thick] (-2.5,-1.5) -- (-.1,.75) node [midway] {\scriptsize{$q-1$}};
						\draw [decorate,decoration={brace,amplitude=3pt},xshift=0pt,yshift=-10pt,thick] (-.85,-5.95) -- (.85,-4.35) node [midway] {\scriptsize{$p-1$}};
						\draw [decorate,decoration={brace,amplitude=3pt},xshift=0pt,yshift=-10pt,thick] (4.2,-1.5) -- (1.5,-4.2) node [midway] {\scriptsize{$j_k-1$}};
						\draw [decorate,decoration={brace,amplitude=3pt},xshift=0pt,yshift=-10pt,thick] (-2.1,-6.) -- (-4,-4.1) node [midway] {\scriptsize{$j_1-1$}};
					\end{scope}		
				\end{tikzpicture}
			\end{equation*}
			\begin{equation*}
				c)
			\begin{tikzpicture}
				[scale=.45,auto=left,baseline=-1cm]  
				\begin{scope}[xscale=1,yscale=1]
					
					\node (in) at  (0,1) [] {};
					\node (n4) at  (0,0) [prelie] {};
					\node (n3) at  (-.5,-.5) [prelie] {};
					\node (n2) at  (-1,-1) [prelie] {};
					\node (n1) at  (-2,-2) [prelie] {};
					
					\coordinate (od1) at  (-1.5,-2.5);
					\node (od2) at  (0,-2) [] {};
					\node (od3) at  (.5,-1.5) [] {};
					
					\coordinate (intL) at (-2.5,-2.5);
					\coordinate (ndin) at (2.5,-2.5);
					\node (nd2) at  (2.5,-2.5) [] {};
					\node (nd1) at  (2,-3) [colie] {};
					\node (nd0c) at  (1,-4) [colie] {};
					
					\node (idb) at  (-.5,-3.5) [] {};
					\node (idc) at  (0,-3) [] {};
					\node (id1) at  (1,-2) [] {};

					\node (m1) at  (1,-5) [prelie] {};
					\node (m2) at  (.5,-5.5) [prelie] {};
					\node (m3) at  (-.5,-6.5) [prelie] {};
					
					\node (outm1) at  (2,-6) [] {};
					\node (outm2) at  (1.5,-6.5) [] {};
					\node (outm3) at  (-2.5,-8.5) [] {};
					
					\coordinate (intm) at (.05,-7);
					
					\node (mb1) at  (-.5,-7.5) [colie] {};
					\node (mb2) at  (-2,-6) [colie] {};
					\node (mb3) at  (-3,-5) [colie] {};
					\node (mb4) at  (-3.5,-4.5) [colie] {};
					\node (mb5) at  (-4,-4) [colie] {};
					\node (mb6) at  (-5,-3) [colie] {};
					
					\node (inmb6l) at  (-6,-2) [] {};
					\node (inmb6) at  (-4,-2) [] {};
					\node (inmb5) at  (-3,-3) [] {};
					\node (inmb3) at  (-2,-4) [] {};
					\node (inmb2) at  (-1,-5) [] {};
					
					\coordinate (intmb) at (-3,-4);
					\coordinate (out) at (-.5,-8.15);
					
					\foreach \from/\to in {	in/n4,
						od1/n1,od2/n2,od3/n3,
						idc/nd0c,id1/nd1,
						m1/nd0c,
						m1/m2,
						outm1/m1,outm2/m2,outm3/m3,
						m3/intm,intm/mb1,
						intmb/mb4,mb1/mb2,mb3/mb4,mb4/mb5,
						inmb6l/mb6,inmb6/mb6,inmb5/mb5,inmb3/mb3,inmb2/mb2,
						nd1/ndin,
						intL/n1,n2/n3,n3/n4,n4/ndin,
						out/mb1}
					\draw[gray!20!black] (\from) -- (\to);
					\draw[dashed] (nd1) -- (nd0c);
					\draw[dashed] (n1) -- (n2);
					\draw[dashed] (m2) -- (m3);
					\draw[dashed] (mb2) -- (mb3);
					\draw[dashed] (mb5) -- (mb6);
					\draw[line width=.2mm,red,dashed] (intL) -- (intmb);
					\draw[line width=.2mm,green!60!black,dashed] (od1) -- (intmb);
					
					\draw [decorate,decoration={brace,amplitude=3pt},xshift=0pt,yshift=-10pt,thick] (-2.5,-1.5) -- (-.1,.75) node [midway] {\scriptsize{$q-1$}};
					\draw [decorate,decoration={brace,amplitude=3pt},xshift=0pt,yshift=-10pt,thick] (-.85,-5.95) -- (.85,-4.35) node [midway] {\scriptsize{$p-1$}};
					\draw [decorate,decoration={brace,amplitude=3pt},xshift=0pt,yshift=-10pt,thick] (2.3,-2.75) -- (1.2,-3.9) node [midway] {\scriptsize{$j_k-1$}};
					\draw [decorate,decoration={brace,amplitude=3pt},xshift=0pt,yshift=-10pt,thick] (-2.1,-6.) -- (-5.35,-2.75) node [midway] {\scriptsize{$j_1-1$}};
				\end{scope}		
			\end{tikzpicture}
			\hphantom{1.5cm}
			d)
				\begin{tikzpicture}
				[scale=.45,auto=left,baseline=-1cm]  
				\begin{scope}[xscale=1,yscale=1]
					
					\node (in) at  (0,1) [] {};
					\node (n4) at  (0,0) [prelie] {};
					\node (n3) at  (-.5,-.5) [prelie] {};
					\node (n2) at  (-1,-1) [prelie] {};
					\node (n1) at  (-2,-2) [prelie] {};
					
					\coordinate (od1) at  (-1.5,-2.5);
					\node (od2) at  (0,-2) [] {};
					\node (od3) at  (.5,-1.5) [] {};
					
					\coordinate (intL) at (-2.5,-2.5);
					\coordinate (ndin) at (2.5,-2.5);
					\node (nd2) at  (2.5,-2.5) [] {};
					\node (nd1) at  (2,-3) [colie] {};
					\node (nd0c) at  (1,-4) [colie] {};
					
					\node (idb) at  (-.5,-3.5) [] {};
					\node (idc) at  (0,-3) [] {};
					\node (id1) at  (1,-2) [] {};

					\node (m1) at  (1,-5) [prelie] {};
					\node (m2) at  (.5,-5.5) [prelie] {};
					\node (m3) at  (-.5,-6.5) [prelie] {};
					
					\node (outm1) at  (2,-6) [] {};
					\node (outm2) at  (1.5,-6.5) [] {};
					\node (outm3) at  (-2.5,-8.5) [] {};
					
					\coordinate (intm) at (.05,-7);
					
					\node (mb1) at  (-.5,-7.5) [colie] {};
					\node (mb2) at  (-2,-6) [colie] {};
					\node (mb3) at  (-3,-5) [colie] {};
					\node (mb4) at  (-3.5,-4.5) [colie] {};
					\node (mbp) at  (-2.5,-3.5) [] {};
					
					\node (inmb3) at  (-2,-4) [] {};
					\node (inmb2) at  (-1,-5) [] {};
					
					\coordinate (intmb) at (-4,-4);
					\coordinate (out) at (-.5,-8.15);
					
					\foreach \from/\to in {	in/n4,
						od1/n1,od2/n2,od3/n3,
						idc/nd0c,id1/nd1,
						m1/nd0c,
						m1/m2,
						outm1/m1,outm2/m2,outm3/m3,
						m3/intm,intm/mb1,
						intmb/mb4,mb1/mb2,mb3/mb4,mbp/mb4,
						inmb3/mb3,inmb2/mb2,
						nd1/ndin,
						intL/n1,n2/n3,n3/n4,n4/ndin,
						out/mb1}
					\draw[gray!20!black] (\from) -- (\to);
					\draw[dashed] (nd1) -- (nd0c);
					\draw[dashed] (n1) -- (n2);
					\draw[dashed] (m2) -- (m3);
					\draw[dashed] (mb2) -- (mb3);
					\draw[line width=.2mm,red,dashed] (intL) -- (intmb);
					\draw[line width=.2mm,green!60!black,dashed] (od1) -- (intmb);
					
					\draw [decorate,decoration={brace,amplitude=3pt},xshift=0pt,yshift=-10pt,thick] (-2.5,-1.5) -- (-.1,.75) node [midway] {\scriptsize{$q-1$}};
					\draw [decorate,decoration={brace,amplitude=3pt},xshift=0pt,yshift=-10pt,thick] (-.85,-5.95) -- (.85,-4.35) node [midway] {\scriptsize{$p-1$}};
					\draw [decorate,decoration={brace,amplitude=3pt},xshift=0pt,yshift=-10pt,thick] (2.3,-2.75) -- (1.2,-3.9) node [midway] {\scriptsize{$j_k-1$}};
					\draw [decorate,decoration={brace,amplitude=3pt},xshift=0pt,yshift=-10pt,thick] (-2.1,-6.) -- (-4,-4.1) node [midway] {\scriptsize{$j_1-1$}};
				\end{scope}		
			\end{tikzpicture}
		\end{equation*}
		By using the third and fourth relations of \eqref{rel Koszul 1}, one may move the $j_k-1$ crossed vertices outside the box. 
		Then, by using the associativity and coassociativity relations, one obtains a tree with a pencil box in it.  
		To be more specific, for the trees a) and b) one may move down the crossed vertices that are below $v$ by using the third relation of \eqref{rel Koszul 1} several times. 
		Then, one may move up $v$  by using the last relation of \eqref{rel Koszul 1}; this places all the other crossed vertices that were above $v$ also outside the box. 
		Finally, one may apply the associativity and coassociativity relations to obtain a pencil box. 
		For the trees c) and d), one may proceed similarly but with the sole use of the third relation of \eqref{rel Koszul 1} several times to move all the $j_k-1$ crossed vertices outside the box. The Claim is proven. 
	\end{enumerate}
	\end{proof}
 \end{proof}

\begin{lem}\label{lem: red morphism}
	The map $\redd$ is a morphism of properads. 
\end{lem}
Essentially, this amounts to understanding the compositions \eqref{eq: compo word}  in terms of cyclic words $BW_{1,\dots,m}^{(j_1),\dots,(j_m)}$ in $\mathsf{uBIB}^!$ and see that they do not involve the unit relations, that is, they restrict to reduced words in $\bib^!$.  
\begin{proof}
	Let $BW_J:= BW_{1,\dots,p}^{(j_1),\dots,(j_p)}$ and $BW_K:= BW_{1,\dots,q}^{(k_1),\dots,(k_q)}$ in $\mathsf{uBIB}^!_{>0}$ such that $j_p,k_q\neq 0$.   
	Let $u$ be an element of the tuple $(j_p)$. 
	
	If  $q=1$ it is almost immediate to see that $\redd(BW_J \circ_{u,1} BW_K)= \redd(BW_J) \circ_{u,1} \redd(BW_K)$. 
	
	Suppose $q\geq 2$, and let  $1 \leq t\leq q$. 
	We want to prove that 
	\begin{equation*}
		\redd(BW_J \circ_{u,t} BW_K)= \redd(BW_J) \circ_{u,t} \redd(BW_K). 
	\end{equation*} 
	Note that it is not enough to show this for $t=1$ since $BW_K$ may not have $k_{t-1}> 0$. 
	We first consider the case $t=1$, which we prove in details. 
	Then, we sketch a proof for the other cases; it involves more steps but follows the line of the former case. 

	Recall that $n(0)_J$ (resp. $n(0)_K$) denotes the number of $j_a$'s (resp. $k_a$'s) that are zero. 
	Also, $[i]_J$ denotes the cardinal of $\{ 1\leq a\leq i ~|~ j_a\neq 0 \}$ and similarly for $[i]_K$. 
	Clearly, $n(0)_J= p-[p]_J$. 
	To alleviate notations, let us write $C^{j_i}$ for $C^{(j_i)}$. 
	
	Let $t=1$. Let us represent $\redd(BW_J) \circ_{u,1}\redd(BW_{2})$ by the following level tree, in which we write only the blue boxes' last input, the other ones being implicit. 
	\begin{equation*}
		(-1)^{n(0)_J+n(0)_K}
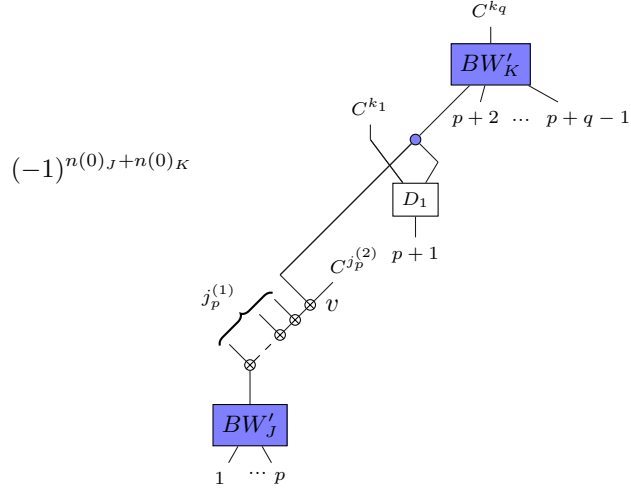
		\begin{tikzpicture}
			[scale=.4,auto=left,baseline=.5cm]  
			\begin{scope}[xscale=1,yscale=1]
				\node[above,yshift=-0.1cm] (out1) at (1,3.25)  {\scriptsize{$C^{k_1}$}};
				\node[above,yshift=-0.1cm] (outn) at (5,6.5)  {\scriptsize{$C^{k_q}$}};
				
				\coordinate (itout1) at (1,2.5)  {};
				
				\coordinate (it) at (3.25,1.75)  {};

				\node[draw] (nds) at (2.5,.5)  {\scriptsize$D_1$};
				
				\node (ns1)  [draw] at (5,5) [fill=blue!50] {\small{$BW_K'$}};
				\node (ns) at  (2.5,2.5) [prelie] {};
				\node (nCb)   at (.5,-1.5) {\scriptsize{$C^{j_p^{(2)}}$}};
				\coordinate (n1s)   at (-2,-2) {};
				\node (nd1) at  (-1,-3) [colie,label=right:{$v$}] {};
				\node (nd0c) at  (-1.5,-3.5) [colie] {};
				\node (nd0b) at  (-2,-4) [colie] {};
				\node (nd0a) at  (-3,-5) [colie] {};
				\node (nW) at  (-3,-7) [draw,fill=blue!50] {\small{$BW_J'$}};
				
				\node (inC1) at  (-4,-4) [] {};					
				\node (inC2) at  (-3,-3) [] {};					
				\node (inC3) at  (-2.5,-2.5) [] {};

				\node[below,yshift=0.1cm] (inn) at (8.25,3.5) {\scriptsize{$p+q-1$}};	
				\node[below,yshift=0.1cm] () at (6,3.25) {\scriptsize{$...$}};	
				\node[below,yshift=0.1cm] (ins1) at (4.5,3.5) {\scriptsize{$p+2$}};	
				
				\node[below,yshift=0.1cm] (ins) at (2.5,-1) {\scriptsize{$p+1$}};	
				
				\node[below,yshift=0.1cm] (in1s) at (-2,-8.5) {\scriptsize{$p$}};	
				\node[below,yshift=0.1cm] () at (-2.75,-8.5) {\scriptsize{$...$}};	
				\node[below,yshift=0.1cm] (in0) at (-4,-8.5) {\scriptsize{$1$}};

				\foreach \from/\to in {in0/nW,in1s/nW,ins1/ns1,ins/nds,inn/ns1,
					nW/nd0a,nd0b/nd0c,nd0c/nd1,nd1/n1s,nd1/nCb,
					nd0a/inC1,nd0b/inC2,	nd0c/inC3,
					n1s/ns,ns/ns1,nds/itout1,ns/n1s,nds/itout1,itout1/out1,nds/it,it/ns,
					ns1/outn}
				\draw[gray!20!black] (\from) -- (\to);
				\draw[dashed] (nd0a) -- (nd0b);
				\draw [decorate,decoration={brace,amplitude=3pt},xshift=0pt,yshift=-10pt,thick] (-4,-4) -- (-2.25,-2.25) node [midway] {\scriptsize{$j_p^{(1)}$}};
			\end{scope}		
		\end{tikzpicture}
	\end{equation*}
	
	\begin{itemize}
		\item If $j_p^{(2)}>0$. 
		We move $C^{j_p^{(2)}}$ in between $C^{k_1}$ and $BW_K'$. 
		We also move the vertex labeled by $v$ immediately above $D_1$. 
		Then, we apply the second relation of \eqref{rel Koszul 2} to the subtree formed by $v$ and the next upper blue vertex.  
		These steps produce the sign: $(-1)^{(j_p^{(2)}-1)(D_1 + C^{k_1}) + D_1+1}= (-1)^{(j_p^{(2)}-1)k_1 + [1]_K +1}$. 
		\\
		If $k_1>0$, there is an additional step. 
		We use the associativity relation on $v$ and the vertex of $D_1$, then we move the highest of the two obtained vertices immediately above the blue vertex. 
		Then we use the associativity relation several times to get the corolla $C^{k_1+j_p^{(2)}}$. 
		The sign is $(-1)^{1+  k_1-1}= (-1)^{k_1}$, which can be written as $(-1)^{[1]_Kk_1}$ to keep track of the condition $k_1>0$. 
		\item If $j_p^{(1)}>0$. 
		We move the $j_p^{(1)}$ crossed vertices immediately above $D_1$. 
		Then, via the third relation of \eqref{rel Koszul 1}, we move them above the blue vertex. 
		We repeat these steps to reach the top. 
		The resulting sign is $(-1)^{j_p^{(1)}(D_1+D_2...+D_{q-1}+ C^{k_1 +j_p^{(2)}}+ C^{k_2} + \cdots + C^{k_{q-1}}) } 
		=(-1)^{j_p^{(1)}\big(j_p^{(2)} + k_1+ k_2 + \cdots + k_{q-1} \big) }$.  
	\end{itemize}
	The sign obtained by these steps is therefore $(-1)^{j_p^{(1)}(j_p^{(2)} + k_1+ k_2 + \cdots + k_{q-1} )+ j_p^{(2)}k_1 + [1]_K +1 }$ if $j_p^{(2)}>0$ and $(-1)^{j_p^{(1)}(j_p^{(2)} + k_1+ k_2 + \cdots + k_{q-1} )}$ otherwise. 
	Except if $j_p^{(2)}>0$ and $k_1=0$ it coincides with $\eta$ from \eqref{eq: eta compo}. 
	Note that the difference between $n(0)_J+n(0)_K$ and $n(0)_{J\circ K}$ (the number of $j_a$'s being zero in $\redd(BW_J \circ_{u,t} BW_K)$) is one if and only if $j_p^{(2)}>0$ and $k_1=0$; otherwise it is zero.  
	Therefore, the equality is proven. 
	\\
	
	We now sketch a proof for $t\geq 2$. 
	Note that if $k_{t-1}>0$, one may use the cyclic relation from Lemma \ref{lem: red well-defined} to be in the setting of the previous case. 
	For $k_{t-1}=0$, we proceed as follows. 
	First, let us  remark that 
	\begin{equation*}
		BW_J \circ_{u,t} BW_{K} = (-1)^{\epsilon_t+\eta}
		BW
		_{  t,\dots,    p+t-2,     p+t-1,                 p+t,\dots,       q+p-1, 1,\dots,  t-2, t-1}
		^{(j'_1),\dots,(j'_{p-1}), (k'_{t})(j_p^{(2)'}), (k'_{t+1}),\dots,(k'_q),(k'_1),\dots, (k_{t-2}),(j^{(1)'}_p)},
	\end{equation*}
where $\epsilon_t= (k_1+ \cdots +k_{t-1})n_K+ k_1+..+k_{t-2}+k_q$ is the sign due to the cyclic relation on $BW_K$ in order to place the $t$-th output at the left most position.  
	On the other side, one has 
	\begin{equation*}
		\redd(BW_J) \circ_{u,t} \redd(BW_{K}) = 
		(-1)^{n(0)_J+n(0)_K}
		\pgfdeclarelayer{background}
		\pgfsetlayers{background,main} 
		\scalebox{.75}
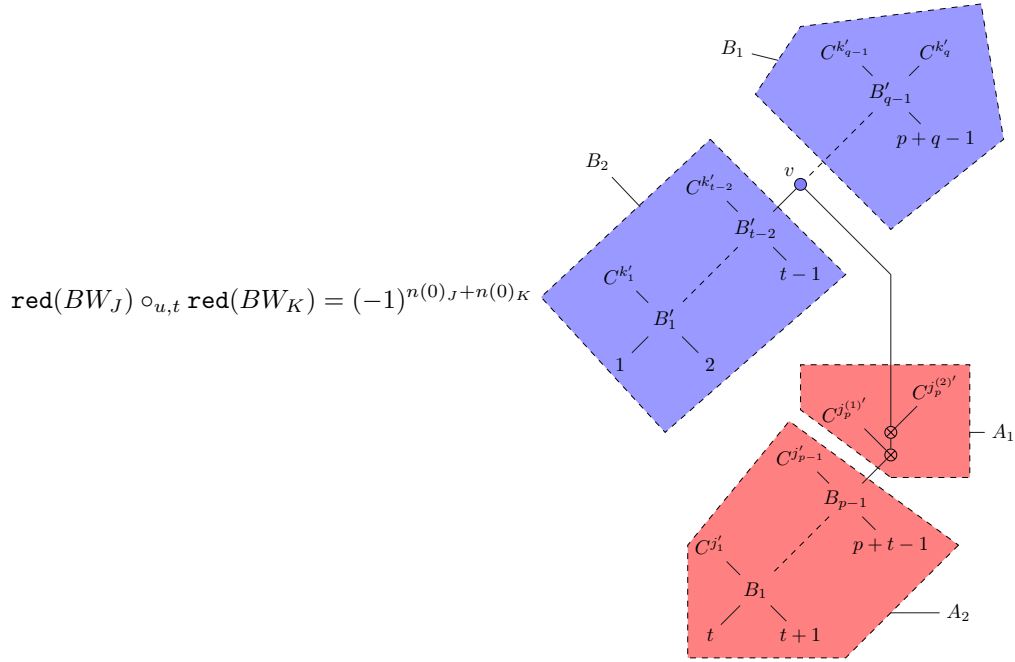
{
				\begin{tikzpicture}
			[scale=.4,auto=left,baseline=-3cm]  
			\begin{scope}[xscale=-1]
				\node (bin1) at (-4,4) {{$C^{k'_q}$}};
				\node (bin2) at (0,4) { {$C^{k'_{q-1}}$}};
				\node (bint2) at (6,-2) { {$C^{k'_{t-2}}$}};
				\node (binq) at (10,-6) { {$C^{k'_{1}}$}};
				
				\node (acin1l) at (-4,-11) { {$C^{j^{(2)'}_p}$}};
				\node (acin1r) at (0,-12) { {$C^{j^{(1)'}_p}$}};
				\node (ain2) at (2,-14) { {$C^{j'_{p-1}}$}};
				\node (ainp) at (6,-18) { {$C^{j'_{1}}$}};
				
				\node (b2) at (-2,2) [] {$B'_{q-1}$};
				\node (v) at (2.5,-1.45) [] {$v$};
				\node (bt1) at (2,-2) [draw,thin, fill=blue!50, circle, minimum size=6pt, inner sep=0pt] {};
				\node (bt2) at (4,-4) [] {$B'_{t-2}$};
				\node (bq) at (8,-8) [] {$B'_{1}$};
				
				\node (ac1) at (-2,-14) [coliel] {};			
				\node (ac2) at (-2,-13) [coliel] {};			
				
				\node (a2) at (0,-16) [] {$B_{p-1}$};
				\node (ap) at (4,-20) [] {$B_{1}$};
				
				\node (bout2) at (-4,0) { {$p+q-1$}};
				\coordinate (boutt1) at (-2,-6);
				\node (boutt2) at (2,-6) { {$t-1$}};
				\node (boutq) at (6,-10) { {$2$}};
				\node (bout11) at (10,-10) { {$1$}};
				
				\node (aout2) at (-2,-18) { {$p+t-1$}};
				\node (aoutp) at (2,-22) { {$t+1$}};
				\node (aoutpl) at (6,-22) { {$t$}};

				\node (A) at (-7,-13) [] { {$A_1$}};
				\node (A2) at (-5,-21) [] { {$A_2$}};
				\node (B) at (11,-1) [] { {$B_2$}};
				\node (B2) at (5,4) [] { {$B_1$}};
				\coordinate (AA) at (-5.5,-13);
				\coordinate (AA2) at (-2,-21);
				\coordinate (BB) at (9.1,-3);
				\coordinate (BB2) at (3,3.5);

				\node (in1) at (1,0) {};
				\foreach \from/\to in {aout2/a2,aoutp/ap,aoutpl/ap,a2/ain2,a2/ac1,ac1/ac2,ac2/acin1l,ac1/acin1r,ap/ainp, ac2/boutt1,boutt1/bt1, bout2/b2,boutt2/bt2,bout11/bq,boutq/bq,b2/bin1,b2/bin2,bt2/bint2,bq/binq,bt1/bt2,
					A/AA,A2/AA2,B/BB,B2/BB2}
				\draw[gray!20!black] (\from) -- (\to);
				\draw (b2) node {} -- (bt1) [dashed] node {};
				\draw (bt2) node {} -- (bq) [dashed] node {};
				\draw (a2) node {} -- (ap) [dashed] node {};
				
				\begin{pgfonlayer}{background}
					\draw[dashed,fill=red!50] (-2,-15) to (-5.5,-15)  -- (-5.5,-10) -- (2,-10) -- (2,-12) -- (-2,-15);
					\draw[dashed,fill=red!50] (-5,-18) to (0,-23)  -- (7,-23) -- (7,-18) -- (2.5,-12.5)-- (-5,-18);
					\draw[dashed,fill=blue!40] (6,0) to (0,-6)  -- (8,-13) -- (13.5,-7) -- (6,0);
					\draw[dashed,fill=blue!40] (-6,6) to (2,5)  -- (4,2) -- (-2,-4) -- (-7,0)-- (-6,6);
				\end{pgfonlayer}
			\end{scope}
		\end{tikzpicture}
		}
	\end{equation*}
	To show that the above level tree is equal to $\redd(BW_J \circ_{r,t} BW_{K})$  in $\bib^!$  we proceed in two steps. 
	The first one consists in moving the corolla labeled by $A_1$. 
	\begin{enumerate}
		\item\label{step 11} The right part $j_1^{(2)'}$ is moved next to $k_{t}'$.  
		Since $t<q$, this is performed by means of relation \eqref{M1} c), among adequate level changes. 
		\item\label{step 12} The left part $j_1^{(1)'}$ is moved just above the blue vertex $v$. 
		The crossed vertex to which $C^{j_1^{(1)'}}$ is attached to is moved just below the blue vertex $v$. 
	\end{enumerate}
	Note that, if $j_1^{(1)}=0$ and/or $j_1^{(2)}=0$, then  \eqref{step 11} and/or \eqref{step 12} are/is skipped. 
	
	The second step consists in moving the part $A_2$ between part $B_1$ and part $B_2$; this is performed as follows. 
	First, the levels of part $A_2$ and part $B_2$ are exchanged. 
	Then, if $j_1^{(1)}= 0$, an iterative use of the coassociativity relation (second relation of \eqref{rel Koszul 1}) provides the result. 
	If $j_1^{(1)} \neq 0$, it is an iterative use of relation \eqref{M1} b) that is performed. 
	In both cases, the relation is used $p-1$ times. 
	The resulting level tree coincides with $\redd(BW_J \circ_{u,t} BW_{K})$. 
	\\
	
	To end this proof, since $\mathsf{uBIB}_{>0}^!$ is a dioperad (it is a sub properad of $\mathsf{uBIB}^!\cong \mathsf{uDPois}^!$ which is a dioperad), it remains to show that any composition of reduced cyclic words of $\mathsf{BIB}^!$ along graphs of positive genus is zero. This is Lemma \ref{lem: compo genus}.   
\end{proof}

	\subsection{The cobar construction of $\bib^\antishrieck$}\label{subsec: dual cobar}

	Let us denote by $\widetilde{rBW}^{J_1,\dots,J_m}_{a_1,\dots,a_m}$ the element in $s^{-1}(\bib^!(m,n))^*$ that is dual to ${rBW}^{J_1,\dots,J_m}_{a_1,\dots,a_m}\in \bib^!(m,n)$. 
	Its degree is therefore $n-2$. 
	
	In the following graphical representation, we simplified the drawing: 
	instead of drawing a corolla with, say, $|J_i|$ input edges, we draw a single edge to a "leaf" labeled by $J_i$.  
	By applying the definition of the cobar differential $\partial$, one obtains that the boundary of a reduced cyclic word $\widetilde{rBW}^{J_1,\dots,J_m}_{1,\dots,m}$ is 
	
	\begin{align*}
		\partial	\left(
		\begin{tikzpicture}
			[scale=.4,auto=left,baseline=0.4cm]  
			\node (out1) at (-1,3)  {\scriptsize{$J_1$}};
			\node (out2) at (1,3) {\scriptsize{$J_2$}};
			\node (out3) at (5,3) {\scriptsize{$J_m$}};
			\node (n1) at (-1,1) [circblack] {};
			\node (n2) at (1,1) [circblack] {};
			\node (n3) at (5,1) [circblack] {};
			\node (in1) at (-1,-.5) {\scriptsize{$1$}};
			\node (in2) at (1,-.5) {\scriptsize{$2$}};
			\node (in3) at (5,-.5) {\scriptsize{$m$}};
			\node (dotintervert) at (3,1) {\scriptsize{$\cdots$}};
			\foreach \from/\to in {in1/n1,in2/n2,in3/n3,n1/out1,n2/out2,n3/out3}
			\draw[gray!40!black] (\from) -- (\to);
			\draw [decorate,decoration={zigzag,amplitude=.7pt,segment length=1.5mm}]    (n1) -- (2.25,1);
			\draw [decorate,decoration={zigzag,amplitude=.7pt,segment length=1.5mm}]    (3.75,1) -- (n3);
		\end{tikzpicture}
		\right)
		&=
		\sum_{\substack{\upsilon\in \Sigma^{\text{cyc}}_m \\  1\leq i \leq m-1\\ J_{\upsilon(i)}^{(1)}J_{\upsilon(i)}^{(2)} = J_{\upsilon(i)}\\  J_{\upsilon(m)}^{(1)}J_{\upsilon(m)}^{(2)} = J_{\upsilon(m)} \\ |J_{\upsilon(m)}^{(1)}|+|J_{\upsilon(i)}^{(2)}|\neq 0 \text{ if } i=1 }}
		(-1)^{\zeta_1}
		\begin{tikzpicture}
			[scale=.4,auto=left,baseline=1.3cm]  
			\begin{scope}[yscale=1,xscale=1]
				\node (out0) at (-5,3)  {\scriptsize{$J_{\upsilon(1)}$}};
				\node (out1) at (-1,3)  {\scriptsize{$J_{\upsilon(i-1)}$}};
				\node (out2) at (1,3) {\scriptsize{$J_{\upsilon(m)}^{(1)}$}};
				\node (out22) at (3,3) {\scriptsize{$J_{\upsilon(i)}^{(2)}$}};
				\node (out3) at (2,7) {\scriptsize{$J_{\upsilon(i)}^{(1)}$}};
				\node (out4) at (5,7) {\scriptsize{$J_{\upsilon(i+1)}$}};
				\node (out5) at (9,7) {\scriptsize{$J_{\upsilon(m)}^{(2)}$}};
				\node (n0) at (-5,1) [circblack] {};
				\node (n1) at (-1,1) [circblack] {};
				\node (n2) at (2,1) [circblack] {};
				\node (n3) at (2,5) [circblack] {};
				\node (n4) at (5,5) [circblack] {};
				\node (n5) at (9,5) [circblack] {};
				\node (in0) at (-5,-.5) {\scriptsize{$\upsilon(1)$}};
				\node (in1) at (-1,-.5) {\scriptsize{$\upsilon(i-1)$}};
				\node (in2) at (2,-.5) {\scriptsize{$\upsilon(i)$}};
				\node (in4) at (5,3.5) {\scriptsize{$\upsilon(i+1)$}};
				\node (in5) at (9,3.5) {\scriptsize{$\upsilon(m)$}};
				\node (dot1out) at (-2.75,1) {\scriptsize{$\cdots$}};
				\node (dot1out) at (6.75,5) {\scriptsize{$\cdots$}};
				\foreach \from/\to in {in0/n0,n0/out0,in1/n1,in2/n2,n1/out1,n2/out2,n2/out22,n2/n3,n3/out3,in4/n4,n4/out4,in5/n5,n5/out5}
				\draw[gray!40!black] (\from) -- (\to);
				\draw [decorate,decoration={zigzag,amplitude=.7pt,segment length=1.5mm}]    (n0) -- (-3.5,1);
				\draw [decorate,decoration={zigzag,amplitude=.7pt,segment length=1.5mm}]    (-2,1) -- (n2);
				\draw [decorate,decoration={zigzag,amplitude=.7pt,segment length=1.5mm}]    (n3) --(6,5);
				\draw [decorate,decoration={zigzag,amplitude=.7pt,segment length=1.5mm}]    (7.5,5) -- (n5);
			\end{scope}
		\end{tikzpicture}
		\\
		&+
		\sum_{\substack{\upsilon\in \Sigma^{\text{cyc}}_m \\   J_{\upsilon(m)}^{(1)}J_{\upsilon(m)}^{(2)} J_{\upsilon(m)}^{(3)}= J_{\upsilon(m)} \\ |J_{\upsilon(m)}^{(2)}|> 1 } }
		(-1)^{\zeta_2}
			\begin{tikzpicture}
	[scale=.4,auto=left,baseline=1.3cm]  
	\begin{scope}[yscale=1,xscale=1]
		\node (out0) at (-5,3)  {\scriptsize{$J_{\upsilon(1)}$}};
		\node (out1) at (-1.5,3)  {\scriptsize{$J_{\upsilon(m-1)}$}};
		\node (out2) at (1,3) {\scriptsize{$J_{\upsilon(m)}^{(1)}$}};
		\node (out22) at (3,3) {\scriptsize{$J_{\upsilon(m)}^{(3)}$}};
		\node (out3) at (2,7) {\scriptsize{$J_{\upsilon(m)}^{(2)}$}};
		\node (n0) at (-5,1) [circblack] {};
		\node (n1) at (-1.5,1) [circblack] {};
		\node (n2) at (2,1) [circblack] {};
		\node (n3) at (2,5) [circblack] {};
		\node (in0) at (-5,-.5) {\scriptsize{$\upsilon(1)$}};
		\node (in1) at (-1.5,-.5) {\scriptsize{$\upsilon(m-1)$}};
		\node (in2) at (2,-.5) {\scriptsize{$\upsilon(m)$}};	
		\node (dot1out) at (-3,1) {\scriptsize{$\cdots$}};
		\foreach \from/\to in {in0/n0,n0/out0,in1/n1,in2/n2,n1/out1,n2/out2,n2/out22,n2/n3,n3/out3}
		\draw[gray!40!black] (\from) -- (\to);
		\draw [decorate,decoration={zigzag,amplitude=.7pt,segment length=1.5mm}]    (n0) -- (-3.75,1);
		\draw [decorate,decoration={zigzag,amplitude=.7pt,segment length=1.5mm}]    (-2.25,1) -- (n2);
	\end{scope}
\end{tikzpicture}
\end{align*}
	where, setting $\varepsilon_0:=0$ and $\varepsilon_s:= (j_1+ \cdots +j_{s-1})n+ j_1+..+j_{s-2}+j_m$, one has:  
	\begin{multline*}
	\zeta_1=	j_{\upsilon(m)}^{(1)}(j_{\upsilon(i)}+ \cdots + j_{\upsilon(m-1)} ) 
		+j_{\upsilon(i)}^{(2)} j_{\upsilon(i)}^{(1)}  
		+ \varepsilon_{\upsilon(1)-1} %
\\		+ \left(  j_{\upsilon(1)} + ...+ j_{\upsilon(i-1)} + j_{\upsilon(m)}^{(1)} + j_{\upsilon(i)}^{(2)} \right) 
		\cdot  \left(  j_{\upsilon(i)}^{(1)} + j_{\upsilon(i+1)} + ...+j_{\upsilon(m-1)}  + j_{\upsilon(m)}^{(2)}-1 \right)+1
		\end{multline*}
	and 
	\begin{equation*}
	\zeta_2=	j_{\upsilon(m)}^{(3)}\left(j_{\upsilon(m)}^{(2)}-1\right) 
		+ \varepsilon_{\upsilon(1)-1} 
		+ \left( j_{\upsilon(1)} + ...+ j_{\upsilon(m-1)} +j_{\upsilon(m)}^{(1)} + j_{\upsilon(m)}^{(3)} \right)   j_{\upsilon(m)}^{(2)}  +1. 
	\end{equation*}
	
	This follows from the computation of the infinitesimal decomposition maps of $\bib^\antishrieck$, which are dual to the infinitesimal composition maps of $\bib^! \subset \mathsf{uBIB}^!\cong \mathsf{uDPois}^!$. 
	The computation follows the lines of \cite[Proof of Lemma 1.34]{Leray-Vallette} and in not reproduced here. 

\begin{prop}\label{cor: surj cob bib dp}
	There is a surjection of properads 
	\begin{equation*}
	\cobd{\bib^\antishrieck } \twoheadrightarrow \cobd{\mathsf{DPois}^{\textnormal{\textexclamdown}}}. 
	\end{equation*}
\end{prop}
	\begin{proof}
	This follows from the bottom inclusion  in the commutative diagram of properads. 
	\begin{equation*}
			\begin{tikzpicture}
			\matrix[row sep=-.5mm,column sep=3mm,ampersand replacement=\&,baseline=0pt]
			{
				\node (-10) {$\mathsf{uDPois}^!\cong \mathsf{uBIB}^!$};	\\	%
				\node (00) {$\cup \hspace{1cm} \cup$};	\\		%
				\node (00) {$\mathsf{DPois}^{!} \subset \bib^!$};	\\		%
			}; 
		\end{tikzpicture}
	\end{equation*}
	By duality, one obtains a surjective morphism of coproperads $\bib^\antishrieck \twoheadrightarrow \mathsf{DPois}^{\antishrieck}$. 
	It sends to zero any word $\widetilde{rBW}^{J_1,\dots,J_m}_{a_1,\dots,a_m}$ with at least one $J_i=\emptyset$, and it is the identity on the other words. 
\end{proof}

\section{Comparison with  pre-CY algebras}\label{sec: pCY}

	Recall that $\mathsf{pCY}$ denotes the properad that governs the pre-CY algebras of dimension $2$. 

	In \cite[Proposition 3.2]{Leray-Vallette}, the properad $\mathsf{pCY}$ is described as the cobar construction of the partial coproperad 
	\begin{equation*}
	\mathsf{C_{pCY}} := \left( \mathsf{uDPois}^!/\mathsf{E_{uDP^!}}(1,0)_{1} \right)^*. 
	\end{equation*}
	From the Section \ref{sec: curved}, we know that $\mathsf{uDPois}^!\cong  \mathsf{uBIB}^!$, from which one may infer the isomorphism  
	\begin{equation*}
		\mathsf{C_{pCY}} 
		\cong  	 
		\left( \mathsf{uBIB}^!/\mathsf{E_{uBIB^!}}(1,0)_{1} \right)^* .
	\end{equation*}
	With this description, the following result is straightforward. 
	
	\begin{prop}\label{cor: surj pCY cob bib}
		There is a surjection of properads 
		\begin{equation*}
			\mathsf{pCY}  \twoheadrightarrow \cobd{\bib^{\textnormal{\textexclamdown}} } .  
		\end{equation*}
	\end{prop}
	\begin{proof}
		It is enough to show the inclusion of properads $\bib^! \subset  \mathsf{uBIB}^!/\mathsf{E_{uBIB^!}}(1,0)_{1}$, which is clear from the work of the previous section. 
	The above surjection sends identically any reduced cyclic word $\widetilde{rBW}^{J_1,\dots,J_m}_{a_1,\dots,a_m}$ with at least one  $J_i\neq\emptyset$, and projects the reduced cyclic words $\widetilde{rBW}^{\emptyset,\dots,\emptyset}_{a_1,\dots,a_m}$ to zero. 
		\end{proof}

	\section{Relationship with odd Lie bialgebras}\label{sec: LieB}
	
	The Koszul dual properad of $\mathsf{LieB}:=\mathsf{LieB}^{-1}$ is given by $\mathcal{G}(s^{-1}\E^*)/(s^{-2}\mathsf{R}^{\perp})$, where $s^{-1}\mathsf{E}^*$ is generated by 
	\begin{equation}
		(s^{-1}\mathsf{E}^*(1,2))_{-1}
		= \vspan \left\langle 
		\begin{tikzpicture}
			[scale=.35,auto=left,baseline=0.35cm]  
			\node (in1) at (1,0) {};
			\node (n2) at (1,1) [colie] {};
			\node[above,yshift=-0.1cm] (out1) at (0,2) {\scriptsize{$1$}};
			\node[above,yshift=-0.1cm] (out2) at (2,2) {\scriptsize{$2$}};
			
			\foreach \from/\to in {in1/n2,n2/out1,n2/out2}
			\draw[gray!20!black] (\from) -- (\to);
		\end{tikzpicture}
		=	-
		\begin{tikzpicture}
			[scale=.35,auto=left,baseline=0.35cm]  
			\node (in1) at (1,0) {};
			\node (n2) at (1,1) [colie] {};
			\node[above,yshift=-0.1cm] (out1) at (0,2) {\scriptsize{$2$}};
			\node[above,yshift=-0.1cm] (out2) at (2,2) {\scriptsize{$1$}};
			
			\foreach \from/\to in {in1/n2,n2/out1,n2/out2}
			\draw[gray!20!black] (\from) -- (\to);
		\end{tikzpicture}
		\right\rangle
		\text{ and } 
			(s^{-1}\mathsf{E}^*(2,1))_{0}
		= \vspan \left\langle 
		\begin{tikzpicture}
			[scale=.35,auto=left,baseline=0.15cm]  
			\node[below,yshift=0.1cm] (in1) at (0,0) {\scriptsize{$1$}};
			\node[below,yshift=0.1cm] (in2) at (2,0) {\scriptsize{$2$}};
			\node (n1) at (1,1) [prelie] {};
			\node (out1) at (1,2)  {};
			
			\foreach \from/\to in {in1/n1,in2/n1,n1/out1}
			\draw[gray!20!black] (\from) -- (\to);
		\end{tikzpicture}
		=
		\begin{tikzpicture}
			[scale=.35,auto=left,baseline=0.15cm]  
			\node[below,yshift=0.1cm] (in1) at (0,0) {\scriptsize{$2$}};
			\node[below,yshift=0.1cm] (in2) at (2,0) {\scriptsize{$1$}};
			\node (n1) at (1,1) [prelie] {};
			\node (out1) at (1,2)  {};
			
			\foreach \from/\to in {in1/n1,in2/n1,n1/out1}
			\draw[gray!20!black] (\from) -- (\to);
		\end{tikzpicture}
		\right\rangle
	\end{equation}
	and  $s^{-2}\mathsf{R}^{\perp}$ is generated by 
	\begin{equation}\label{eq: rel Liebi!}
		\begin{tikzpicture}
			[decoration={
				markings,
				mark=at position 0.6 with {\arrow{>}}},
			>=stealth,gray!20!black,,scale=.3,auto=left,baseline=-0.45cm]  
			\begin{scope}[yscale=-1,xscale=1]
				\node[below,yshift=0.1cm] (in1) at (0,-0.5) {\scriptsize{$1$}};
				\node[below,yshift=0.1cm] (in2) at (1,-0.5) {\scriptsize{$2$}};
				\node[below,yshift=0.1cm] (in3) at (3,-0.5) {\scriptsize{$3$}};
				\coordinate (itin1) at (0,1) {};
				\node (n2) at (2,1.5) [colie] {};
				\node (n1) at (1,2.5) [colie] {};;
				\coordinate (out2) at (1,3)  {};
				
				\foreach \from/\to in {in1/itin1,itin1/n1,in2/n2,n2/n1,in3/n2,out2/n1}
				\draw[gray!20!black] (\from) -- (\to);
			\end{scope}
		\end{tikzpicture}
		-
		\begin{tikzpicture}
			[decoration={
				markings,
				mark=at position 0.6 with {\arrow{>}}},
			>=stealth,gray!20!black,,scale=.3,auto=left,baseline=-0.45cm]  
			\begin{scope}[yscale=-1,xscale=1]
				\node[below,yshift=0.1cm] (in1) at (0,-0.5) {\scriptsize{$3$}};
				\node[below,yshift=0.1cm] (in2) at (1,-0.5) {\scriptsize{$1$}};
				\node[below,yshift=0.1cm] (in3) at (3,-0.5) {\scriptsize{$2$}};
				\coordinate (itin1) at (0,1) {};
				\node (n2) at (2,1.5) [colie] {};
				\node (n1) at (1,2.5) [colie] {};;
				\coordinate (out2) at (1,3)  {};
				
				\foreach \from/\to in {in1/itin1,itin1/n1,in2/n2,n2/n1,in3/n2,out2/n1}
				\draw[gray!20!black] (\from) -- (\to);
			\end{scope}
		\end{tikzpicture}
		;
		\hspace{0.5cm}
		\begin{tikzpicture}
			[decoration={
				markings,
				mark=at position 0.6 with {\arrow{>}}},
			>=stealth,gray!20!black,,scale=.3,auto=left,baseline=0.45cm]  
			\node (in1)[above,yshift=-0.1cm] at (0,-0.5) {\scriptsize{$1$}};
			\node (in2)[above,yshift=-0.1cm] at (1,-0.5) {\scriptsize{$2$}};
			\node (in3)[above,yshift=-0.1cm] at (3,-0.5) {\scriptsize{$3$}};
			\coordinate (itin1) at (0,1) {};
			\node (n2) at (2,1.5) [prelie] {};
			\node (n1) at (1,2.5) [prelie] {};;
			\coordinate (out2) at (1,3)  {};
			
			\foreach \from/\to in {in1/itin1,itin1/n1,in2/n2,n2/n1,in3/n2,out2/n1}
			\draw[gray!20!black] (\from) -- (\to);
		\end{tikzpicture}
		-
		\begin{tikzpicture}
			[decoration={
				markings,
				mark=at position 0.6 with {\arrow{>}}},
			>=stealth,gray!20!black,,scale=.3,auto=left,baseline=0.45cm]  
			\node (in1)[above,yshift=-0.1cm] at (0,-0.5) {\scriptsize{$3$}};
			\node (in2)[above,yshift=-0.1cm] at (1,-0.5) {\scriptsize{$1$}};
			\node (in3)[above,yshift=-0.1cm] at (3,-0.5) {\scriptsize{$2$}};
			\coordinate (itin1) at (0,1) {};
			\node (n2) at (2,1.5) [prelie] {};
			\node (n1) at (1,2.5) [prelie] {};;
			\coordinate (out2) at (1,3)  {};
			
			\foreach \from/\to in {in1/itin1,itin1/n1,in2/n2,n2/n1,in3/n2,out2/n1}
			\draw[gray!20!black] (\from) -- (\to);
		\end{tikzpicture};
		\hspace{0.5cm}
		\begin{tikzpicture}
			[scale=.3,auto=left,baseline=0.45cm]  
			\node[below,yshift=0.1cm] (in1) at (0,0) {\scriptsize{$1$}};
			\node[below,yshift=0.1cm] (in2) at (2,0) {\scriptsize{$2$}};
			\node (n1) at (1,1.25) [prelie] {};;
			\node (n2) at (1,2.25) [colie] {};
			\node[above,yshift=-0.1cm] (out1) at (0,3.5)  {\scriptsize{$1$}};
			\node[above,yshift=-0.1cm] (out2) at (2,3.5)  {\scriptsize{$2$}};
			
			\foreach \from/\to in {in1/n1,in2/n1,n1/n2,n2/out1,n2/out2}
			\draw[gray!20!black] (\from) -- (\to);
		\end{tikzpicture}
		-
		\begin{tikzpicture}
			[scale=.3,auto=left,baseline=0.45cm]  
			\node[below,yshift=0.1cm] (in1) at (1,0) {\scriptsize{$1$}};
			\node[below,yshift=0.1cm] (in2) at (3,0) {\scriptsize{$2$}};
			\coordinate (itin2) at (3,1) {};
			\node (n1) at (1,1) [colie] {};
			\node (n2) at (2,2.5) [prelie] {};;
			\coordinate (itout1) at (0,2.5)  {};
			\node[above,yshift=-0.1cm] (out1) at (0,3.5)  {\scriptsize{$1$}};
			\node[above,yshift=-0.1cm] (out2) at (2,3.5)  {\scriptsize{$2$}};
			
			\foreach \from/\to in {in1/n1,n1/n2,n1/itout1,itout1/out1,n2/out2,in2/itin2,itin2/n2}
			\draw[gray!20!black] (\from) -- (\to);
		\end{tikzpicture}
		\hspace{.5cm}\text{ and }\hspace{.5cm}
		\begin{tikzpicture}
			[scale=.3,auto=left,baseline=0.45cm]  
			\node[below,yshift=0.1cm] (in1) at (0,0) {\scriptsize{$1$}};
			\node[below,yshift=0.1cm] (in2) at (2,0) {\scriptsize{$2$}};
			\node (n1) at (1,1.25) [prelie] {};;
			\node (n2) at (1,2.25) [colie] {};
			\node[above,yshift=-0.1cm] (out1) at (0,3.5)  {\scriptsize{$1$}};
			\node[above,yshift=-0.1cm] (out2) at (2,3.5)  {\scriptsize{$2$}};
			
			\foreach \from/\to in {in1/n1,in2/n1,n1/n2,n2/out1,n2/out2}
			\draw[gray!20!black] (\from) -- (\to);
		\end{tikzpicture}
		-
		\begin{tikzpicture}
			[scale=.3,auto=left,baseline=0.45cm] 
			\node[below,yshift=0.1cm] (in1) at (0,0) {\scriptsize{$1$}};
			\node[below,yshift=0.1cm] (in2) at (2,0) {\scriptsize{$2$}}; 
			\coordinate (itin1) at (0,1) {};
			\node (n2) at (2,1) [colie] {};
			\node (n1) at (1,2.5) [prelie] {};;
			\coordinate (itout2) at (3,2.5)  {};
			\node[above,yshift=-0.1cm] (out1) at (1,3.5)  {\scriptsize{$1$}};
			\node[above,yshift=-0.1cm] (out2) at (3,3.5)  {\scriptsize{$2$}};
			
			\foreach \from/\to in {in1/itin1,itin1/n1,n1/out1,in2/n2,n2/n1,n2/itout2,itout2/out2}
			\draw[gray!20!black] (\from) -- (\to);
		\end{tikzpicture}. 
	\end{equation}

\begin{rem}
	The properad $\mathsf{LieB}^!$	 governs the \emph{involutive Frobenius algebras} with a degree $-1$  product. 
	The involutivity, which corresponds to the triviality of the genus one operation obtained by composing the product with the coproduct, follows from the anti-commutativity and commutativity of the product and coproduct, respectively. 
	The last two relations of \eqref{eq: rel Liebi!} corresponds to the Frobenius relations. 
\end{rem}

\begin{prop}
	The map $f\co \mathcal{G}(\mathsf{E_{BIB^!}}) \to \mathsf{LieB}^!$ given by 
\begin{equation*}
	f\left(
	\begin{tikzpicture}
		[scale=.35,auto=left,baseline=0.35cm]  
		\node (in1) at (1,0) {};
		\node (n2) at (1,1) [colie] {};
		\node[above,yshift=-0.1cm] (out1) at (0,2) {\scriptsize{$1$}};
		\node[above,yshift=-0.1cm] (out2) at (2,2) {\scriptsize{$2$}};
		
		\foreach \from/\to in {in1/n2,n2/out1,n2/out2}
		\draw[gray!20!black] (\from) -- (\to);
	\end{tikzpicture}
	\right)= 
	\begin{tikzpicture}
		[scale=.35,auto=left,baseline=0.35cm]  
		\node (in1) at (1,0) {};
		\node (n2) at (1,1) [colie] {};
		\node[above,yshift=-0.1cm] (out1) at (0,2) {\scriptsize{$1$}};
		\node[above,yshift=-0.1cm] (out2) at (2,2) {\scriptsize{$2$}};
		
		\foreach \from/\to in {in1/n2,n2/out1,n2/out2}
		\draw[gray!20!black] (\from) -- (\to);
	\end{tikzpicture}
	\text{ and } 
	f\left(
	\begin{tikzpicture}
		[scale=.35,auto=left,baseline=0.15cm]  
		\node[below,yshift=0.1cm] (in1) at (0,0) {\scriptsize{$1$}};
		\node[below,yshift=0.1cm] (in2) at (2,0) {\scriptsize{$2$}};
		\node (n1) at (1,1) [prelie] {};
		\node (out1) at (1,2)  {};
		
		\foreach \from/\to in {in1/n1,in2/n1,n1/out1}
		\draw[gray!20!black] (\from) -- (\to);
	\end{tikzpicture}
	\right)= 
	\begin{tikzpicture}
	[scale=.35,auto=left,baseline=0.15cm]  
	\node[below,yshift=0.1cm] (in1) at (0,0) {\scriptsize{$1$}};
	\node[below,yshift=0.1cm] (in2) at (2,0) {\scriptsize{$2$}};
	\node (n1) at (1,1) [prelie] {};
	\node (out1) at (1,2)  {};
	
	\foreach \from/\to in {in1/n1,in2/n1,n1/out1}
	\draw[gray!20!black] (\from) -- (\to);
\end{tikzpicture}
\end{equation*}
 induces a morphism of properads	$\bib^! \to \mathsf{LieB}^!$.
\end{prop}
\begin{proof}
	Straightforward. 
\end{proof}

Let us make explicit the induced morphism 
\begin{equation*}
\cobd{\mathsf{LieB}^{\textnormal{\textexclamdown}}}	\to \cobd{\bib^\antishrieck }. 
\end{equation*}

First, one observes that $\mathsf{LieB}^!(m,n)$ is generated by one element 
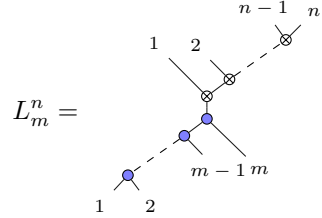
\begin{equation*}
	L^n_m = 
			\begin{tikzpicture}	
		[scale=.3,auto=left,baseline=0.75cm]  
		\begin{scope}[xscale=-1,yscale=-1,xshift=-4cm,yshift=-6cm]
			\node[above,yshift=-.1cm] (in-1) at (-2.75,-1) {\scriptsize{$n$}};
			\node[above,yshift=-.1cm] (in0) at (-.5,-1) {\scriptsize{$n-1$}};
			\node[above,xshift=0mm,yshift=-.2cm] (in2) at (2.5,.25) {\scriptsize{$2$}};
			\node[above,xshift=-1mm,yshift=-.1cm] (in3) at (4,.5) {\scriptsize{$1$}};
			
			\node (n3) at (-1.5,0) [colie] {};
			\node (n2) at (1,1.75) [colie] {};
			\node (n1) at (2,2.5) [colie] {};
			\coordinate (out2) at (2,3)  {};
			\foreach \from/\to in {in0/n3,in-1/n3,in3/n1,n2/n1,in2/n2,out2/n1}
			\draw[gray!20!black] (\from) -- (\to);
			\draw (n2) node {} -- (n3) [dashed] node {};

		\end{scope}
		\node[below,yshift=.1cm] (in-1) at (-2.75,-1) {\scriptsize{$1$}};
		\node[below,yshift=.1cm] (in0) at (-.5,-1) {\scriptsize{$2$}};
		\node[below,xshift=0mm,yshift=.2cm] (in2) at (2.5,.25) {\scriptsize{$m-1$}};
		\node[below,xshift=1mm,yshift=.1cm] (in3) at (4,.5) {\scriptsize{$m$}};
		
		\node (n3) at (-1.5,0) [prelie] {};
		\node (n2) at (1,1.75) [prelie] {};
		\node (n1) at (2,2.5) [prelie] {};
		\coordinate (out2) at (2,3)  {};
		
		\foreach \from/\to in {in0/n3,in-1/n3,in3/n1,n2/n1,in2/n2,out2/n1}
		\draw[gray!20!black] (\from) -- (\to);
		\draw (n2) node {} -- (n3) [dashed] node {};
	\end{tikzpicture}
\end{equation*}
which is anti-symmetric in the $n$ inputs and symmetric in the $m$ outputs; for instance see \cite{Merkulov2004}. 
The (anti-)symmetry results from the (anti-)symmetry and (anti-)associativity of the generators. 

Recall that $rBW_{1,\dots,m}^{(j_1),\dots,(j_m)}$ is the class of the tree \eqref{eq: image of word 2}, twisted by $(-1)^{n(0)}$. 
\begin{prop}
The morphism $f\co  \bib^! \to \mathsf{LieB}^!$ sends any word $rBW_{1,\dots,m}^{(j_1),\dots,(j_m)}$ with $j_m\neq 0$ to $(-1)^{ j_1+ \cdots +j_{m-1} + m-1} L^n_m$. 	
\end{prop}
 \begin{proof}
 	If every $j_k=1$ then this is clear. Indeed, one may remark that the last two relations of \eqref{eq: rel Liebi!} allows us to obtain $L^n_m$, with no signs involved.  
 	In the general case, by only applying the aforementioned relations, one obtains a level tree with 
 	the shape of $L^n_m$ but with the corollas $C^{j_i}$ above:
 	
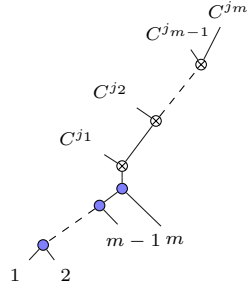
\begin{equation*}
 		\begin{tikzpicture}	
 		[scale=.3,auto=left,baseline=0.45cm]  
 		\begin{scope}[xscale=-1,yscale=-1,xshift=-4cm,yshift=-6cm]
 			\node[above,yshift=-.1cm] (in-1) at (-2.75,-4) {\scriptsize{$C^{j_m}$}};
 			\node[above,yshift=-.1cm] (in0) at (-.5,-3) {\scriptsize{$C^{j_{m-1}}$}};
 			\node[above,xshift=0mm,yshift=-.2cm] (in2) at (2.5,-.75) {\scriptsize{$C^{j_2}$}};
 			\node[above,xshift=0mm,yshift=-.3cm] (in3) at (4,1) {\scriptsize{$C^{j_1}$}};
 			
 			\node (n3) at (-1.5,-2) [colie] {};
 			\node (n2) at (.5,.5) [colie] {};
 			\node (n1) at (2,2.5) [colie] {};
 			\coordinate (out2) at (2,3)  {};
 			\foreach \from/\to in {in0/n3,in-1/n3,in3/n1,n2/n1,in2/n2,out2/n1}
 			\draw[gray!20!black] (\from) -- (\to);
 			\draw (n2) node {} -- (n3) [dashed] node {};

 		\end{scope}
 		\node[below,yshift=.1cm] (in-1) at (-2.75,-1) {\scriptsize{$1$}};
 		\node[below,yshift=.1cm] (in0) at (-.5,-1) {\scriptsize{$2$}};
 		\node[below,xshift=0mm,yshift=.2cm] (in2) at (2.5,.25) {\scriptsize{$m-1$}};
 		\node[below,xshift=1mm,yshift=.1cm] (in3) at (4,.5) {\scriptsize{$m$}};
 		
 		\node (n3) at (-1.5,0) [prelie] {};
 		\node (n2) at (1,1.75) [prelie] {};
 		\node (n1) at (2,2.5) [prelie] {};
 		\coordinate (out2) at (2,3)  {};
 		
 		\foreach \from/\to in {in0/n3,in-1/n3,in3/n1,n2/n1,in2/n2,out2/n1}
 		\draw[gray!20!black] (\from) -- (\to);
 		\draw (n2) node {} -- (n3) [dashed] node {};
 	\end{tikzpicture}.  
 \end{equation*} 	
 	Here, if $j_k=0$, then both $C^{j_k}$ and the crossed vertex it is attached to are not present. 
 	For each $j_s>0$, we ``move"  $C^{j_s}$ by the anti-associativity relation for the crossed vertex. 
 	This produces the sign $(-1)^{j_s-1}$. 
 	We end up with the sign  $(-1)^{ j_1+ \cdots +j_{m-1} + m-1-n(0)}$, hence the result. 
 \end{proof}

\begin{cor}\label{cor: Liebi to bib KD}
The morphism $f^*\co \mathsf{LieB}^{\textnormal{\textexclamdown}}	\to \bib^\antishrieck$ sends the dual of $L^n_m$ to 
\begin{equation*}
	\sum_{[\sigma]\in \mfaktor{\Sigma^{\text{cyc}}_m}{\Sigma_m} } 
	\sum_{\substack{j_1+\cdots + j_m=n\\  j_m\neq 0}}
	\sum_{\tau\in \Sigma_n}
	\operatorname{sgn(\tau)} 
	(-1)^{ j_1+ \cdots +j_{m-1} + m-1}
	rBW_{\sigma(1),\dots,\sigma(m)}^{\tau\cdot j_1,\dots,\tau\cdot j_m}, 
\end{equation*}
where $\operatorname{sgn(\tau)}$ is the signature of the permutation $\tau$, and  $\tau\cdot j_k$ stands for the tuple $(\tau(j_{k-1}+1),\dots,\tau(j_k))$; the first sum is over the right cosets $[\sigma]=\Sigma_m^{\text{cyc}}\sigma$. 	
\end{cor}

The resulting map $\mathbf{\Omega}(f^*)\co \cobd{\mathsf{LieB}^{\textnormal{\textexclamdown}}}	\to \cobd{\bib^\antishrieck }$  is the extension of the anti-symmetrization map $\mathsf{LieB}\to \bib$ described in Remark \ref{rem: Liebi bib}.

	\bibliographystyle{alpha}

\end{document}